\documentclass[reqno]{amsart}

\usepackage[pagewise,mathlines]{lineno}
\usepackage{a4wide}
\usepackage{graphics,graphicx}
\usepackage[english]{babel}
\usepackage{color}
\usepackage{amsfonts,amsthm,amssymb,amsmath,mathrsfs}
\usepackage{todonotes}

 \usepackage{hyperref}
\usepackage[normalem]{ulem}
\usepackage{bm}

\newcommand\redsout{\bgroup\markoverwith{\textcolor{red}{\rule[0.5ex]{2pt}{0.4pt}}}\ULon}

\newcommand{\var}{\mbox{\boldmath$\varepsilon$}}

\newcommand{\R}{\mathbb{R}}

\newcommand{\N}{\mathbb{N}}

\numberwithin{equation}{section}
\numberwithin{equation}{section}
\newtheorem{theorem}{\quad Theorem}[section]
\newtheorem{lemma}[theorem]{\quad Lemma}

\newtheorem{remark}[theorem]{\quad Remark}

\newtheorem{proposition}[theorem]{\quad Proposition}

\newtheorem{definition}{Definition}[section]

\newcommand {\bu}{\mathbf{u}}

\newcommand {\vu}{\mathbf{v}}
\newcommand {\bw}{{\bm w}}
\newcommand {\ru}{\mathbf{R}}

\newcommand{\norm}[1]{\left\|   #1   \right\|}

\newcommand{\paar}[1]{\left(   #1   \right)}

\def\N{\mathbb{N}}

\allowdisplaybreaks

\newcommand{\lam}{\lambda}

\begin{document}

  \thanks{
   $^\ast$Corresponding author. \\Department of Mathematics, Nazarbayev University, Astana 010000, Kazakhstan.\\
	  	E-mail: {\tt vicente.alvarez@nu.edu.kz, amin.esfahani@nu.edu.kz.}
  }
	
	\title[Multi-solitons of Boussinesq equation]{ Multi-solitons of one-dimensional Boussinesq equation
	}   
	
	\maketitle
	
	\begin{center}
		
		{\bf Vicente  Alvarez and  Amin Esfahani$^{\ast}$   
			\vskip.2cm}
	\end{center}
	
	\vskip.3cm
	
	
	\begin{abstract}

The existence of multi-speed solitary waves for the one-dimensional good Boussinesq
equation with a power nonlinearity is proven. These solutions are shown to behave at large times
as a pair of scalar solitary waves traveling at different speeds. Both subcritical and supercritical
cases are treated. The proof is based on the construction of approximations of the multi-speed
solitary waves by solving an equivalent system backward in time and using energy methods to
obtain uniform estimates.

\vspace{6mm}


\end{abstract}
	\vspace{8mm}

	\section{Introduction}
The Boussinesq equation, initially derived to describe the propagation of long waves in shallow water, has been a subject of extensive study in both its classical and modified forms. In particular, the ``good'' Boussinesq equation, which can be written as:
	\begin{equation}\label{sistema1}
					\partial_{tt}u_{1} - \partial_{xx}u_{1} + \partial_{xxxx}u_{1}+\partial_{xx}(\varphi(u_1)) = 0,
	\end{equation}
 has attracted significant attention due to its rich mathematical structure and the phenomena it models \cite{bouss}. Here,   $u_1(x,t)$ represents the wave profile and $\varphi(s)=|s|^{2p}s$ with $p>0$.

The study of soliton solutions, which are localized waves that maintain their shape over time, is a central theme in the theory of nonlinear dispersive equations. Solitons are remarkable due to their stability and the ability to interact with each other without losing their identities, a property often referred to as soliton collision. While single soliton solutions provide insight into the fundamental dynamics of these equations, the existence and interaction of multi-soliton solutions reveal deeper aspects of the underlying nonlinear dynamics. A multi-soliton   of \eqref{sistema1}, is a solution $u_1(x,t)$ to \eqref{sistema1} defined on a semi-infinite interval of time, such that is behaves at infinity like a sum of solitons of \eqref{sistema1} with different speeds.

In this paper, our objective is to investigate the construction of nontrivial asymptotic behaviors of solutions to \eqref{sistema1}. In this context, multi-solitons (see Definition \eqref{multisolitons}) are fundamental solutions that, at large times, behave as a sum of several solitons in the energy space.
 We investigate the parameter ranges corresponding to subcritical and supercritical cases, examining how nonlinearity affects the formation and interaction of solitons. 

  The question of the existence and properties of multi-solitons in nonlinear models has a long history, beginning with the pioneering works of Fermi, Pasta, and Ulam \cite{fermi}, and Kruskal and Zabusky \cite{kruskal1965}. These studies are closely linked to the investigation of completely integrable equations via the inverse scattering transform (IST). For an overview of multi-solitons in the Korteweg–de Vries equation, see the review by Miura \cite{miura1976}, and for multi-solitons in (1.2), see Zakharov and Shabat \cite{zakharov1972}.
Beyond the scope of completely integrable models, there have been significant advances in the study of multi-solitons in non-integrable dispersive equations over the past two decades.  Initially, these were addressed in the \(L^2\)-critical and subcritical cases by Merle \cite{Mer}, Martel \cite{Yvan}, and Martel-Merle \cite{Martel}, leveraging their developed theories of stability and asymptotic stability. The existence of multi-solitons was later extended in \cite{CoteMartel} to the \(L^2\) supercritical case.  In this latter scenario, each individual soliton is unstable, making the multi-soliton a highly unstable solution. See also \cite{cote-munoz}.   For the high-speed excited multi-solitons, we refer to \cite{cote}.

The Cauchy problem associated with  \eqref{sistema1} can be equivalently written as a first-order system given by:
	\begin{equation}\label{sistemaeq}
		\begin{cases}
			\partial_{t}u_{1}= \partial_{x}u_2, \\
			\partial_{t}u_2=\partial_{x}(u_1-\partial_{xx}u_1-\varphi(u_1))\\
   (u_{1}(0),u_{2}(0))=(u_{1}^{0},u_{2}^{0})
		\end{cases}
	\end{equation}
Moreover,   solutions of \eqref{sistema1} and \eqref{sistemaeq} the energy $\mathcal{E}$ and momentum $\mathcal{M}$ are formally conserved under the flow of \eqref{sistemaeq}, where 
\begin{equation}\label{energia}
		\mathcal{E}(\bu)=\frac{1}{2}\int_{\R}\left(|u_1|^2+|u_2|^2+|\partial_{x}u_1|^2-\frac1{p+1}|u_1|^{2p+2}\right)\,dx  
 \end{equation}
	and
\begin{equation}\label{moment}
		\mathcal{M}(\bu)=\frac{1}{2}\int_{\R}u_1 u_2\,dx.
\end{equation}

Several researchers have extensively studied system \eqref{sistemaeq}. The foundational work by Bona and Sachs \cite{bonasachs} utilized Kato's abstract methods to establish local and global well-posedness of the Cauchy problem for small data. These results were subsequently enhanced in \cite{linares}, where global well-posedness of \eqref{sistemaeq} in the energy space $H=H^1(\R)\times L^2(\R)$ for small data was demonstrated. More recently, the inverse scattering transform and a Riemann-Hilbert approach for \eqref{sistema1} with quadratic nonlinearity were developed in \cite{chalen}. The  local well-posedness of \eqref{sistemaeq} in the energy space $H$ was established by Liu in \cite{liu93}, and was improved recently in
 \cite{cgs}. 
		\begin{theorem}\label{LWP2}
			Let $\sigma_0= \frac{1}{2}-\frac{1}{p} $  and $\max\{0,\sigma_0\}\leq s\leq 1$. For any $\mathbf{u}_0 \in H^s(\R)\times H^{s-1}(\R)$, 
   there exists a maximal time interval $(-T{\star}, 
			T^{\star})$ and a unique solution $\mathbf{u}$ to    \eqref{sistemaeq} in $\mathbf{C}((-T_{\star}, T^{\star});H^s(\R)\times H^{s-1}(\R))$. Moreover, $T^{\star}=\infty$ or $T^{\star}<\infty$ and  \[\lim_{t\to T^{\star}}\|\mathbf{u}(t)\|_{ H^s(\R)\times H^{s-1}(\R)}=\infty.\]
			A similar result occurs if $T_{\star}=-\infty$ or $T_{\star}<\infty$.
			Additionally, the solution $\mathbf{u} \in H$ of \eqref{sistemaeq} conserves energy and momentum.
				\end{theorem}

A solitary wave (soliton) associated with  system \eqref{sistemaeq} is a solution $\ru\in H$ in the form of
		\[ 
		\ru(x,t)=(\Phi_{\omega}\left(x-\omega t-x_0\right),\Psi_{\omega}\left(x-\omega t-x_0\right)),
		\]
		where  $x_{0} \in \R$  is the inital wave-position and   $\omega$ is the wave-speed with $|\omega|<1$; so that $\ru$  satisfies the     elliptic system 
 \begin{equation}\label{sistema2}
 \begin{cases}
				-  \Phi_{\omega}''+ \Phi_{\omega}-\left|\Phi_{\omega}\right|^{2p} \Phi_{\omega}+ \omega \Psi_{\omega}=0, \\
				\Psi_{\omega}+ \omega \Phi_{\omega}=0.
 \end{cases}     
 \end{equation}
		So, the solutions of this system  are clearly in the form $\left(\Phi_{\omega},- \omega \Phi_{\omega}\right)$, where $\Phi_{\omega}$ satisfies the scalar equation
		\begin{equation}\label{ellipticeq}
			-   \Phi_{\omega}''+\left(1-\omega^{2}\right) \Phi_{\omega}-\left|\Phi_{\omega}\right|^{2p} \Phi_{\omega}=0.
		\end{equation} 
		 
		It is well-known that $\Phi (x)=\paar{(p+1){\rm sech}^2(px)}^{\frac{1}{2p}}$ is the unique (ground state) solution of 
		\[
		- \Phi'' +\Phi -\left|\Phi \right|^{2p} \Phi =0;
		\]
		so that  $\Phi_\omega(x)=(1-\omega^2)^{\frac1{2p}}\Phi(\sqrt{1-\omega^2}x) $ is solution of \eqref{ellipticeq}, and
		\[
		\ru(x,t)=(\Phi_{\omega}\left(x-\omega t-x_0\right),\Psi_{\omega}\left(x-\omega t-x_0\right)),
		\]
		with $\Psi_\omega=-\omega\Phi_\omega$
		is a solution of \eqref{sistemaeq}.
		Recall that a solution $\mathbf{\Phi}$ of \eqref{sistema2} is called a ground state if it minimizes the action $\mathbf{J}$, that is,
		\[
  \mathbf{J}(\mathbf{\Phi})=\min\left\{\mathbf{J}(\mathbf{u}),\, \mathbf{u}\in H\setminus\{0\} \,\,\text{is a solution of \eqref{sistema2}}\right\},
		\]
		where $\mathbf{J}(\mathbf{u})=\mathcal{E}(\bu)+\omega\mathcal{M}(\bu)$.
		
		The following regularity result for solutions of   \eqref{ellipticeq} is well-known (see e.g. \cite[Theorem 8.1.1]{cazenave}).
		\begin{proposition}\label{decaimentoquadratico}
			Consider $\Phi \in H^1(\R) $ a solution of \eqref{ellipticeq}.
			Then,
			\begin{itemize}
				\item[(i)] $\Phi \in W^{3, p}(\R)$ for $2 \leq p<\infty$. In particular, $\Phi \in C^{2}(\R)$ and $\left|\partial_x^{\beta} \Phi(x)\right|\stackrel{|x|\to \infty}{\longrightarrow } 0$ for all $|\beta| \leq 2$.
				\item[(ii)] For any  $0<\epsilon <1$,
				\[
				e^{\epsilon|x|} \left(\left|\Phi(x)\right|+\left|\partial_{x} \Phi(x)\right|\right) \in L^{\infty}(\R).
				\]    
			\end{itemize} 
		\end{proposition} 
		The stability of solitary waves of \eqref{sistemaeq}  has been a subject of various studies. Bona and Sachs \cite{bonasachs} employed the fundamental theory developed in \cite{gss} to prove that for \( p < 2\) and \(2\omega^2>p \), the solitary waves  are orbitally stable. Liu proved the orbital instability of solitary waves in \cite{liu93} for \(  p < 2\) and \(2\omega^2 < p \), or \(p \geq 2\) and \(\omega^2 < 1\). Furthermore, he demonstrated in \cite{liu95} that the solitary wave is strongly unstable by blow-up if \(\omega = 0\). Liu, Ohta, and Todorova \cite{lot} showed that for any $p>0$ such that \(0 < 2(p+1)\omega^2 < p\), the solitary wave is strongly unstable. Recently,   the orbital instability in the degenerate case \(2\omega^2 = p \) for \(p < 2\) was proved  in \cite{lowx}. See also \cite{stud} for the stability of angular waves of \eqref{sistema1} in the high-dimensional case.

 Regardless of the
instability of the solitons, by leveraging techniques from the study of KdV, NLS, and Klein-Gordon equations, we aim to shed light on the behavior of multi-solitons of  \eqref{sistemaeq}. 
		
		\section*{Main result, tools and strategy}

			For $j\in \mathbb{N}$, let $|\omega_{j}|\leq 1$, $x_{j}  \in \R$, and $\mathbf{\Phi}_{\omega_{j}}=(\Phi_{\omega_{j}}^{(1)},\Phi_{\omega_{j}}^{(2)}) \in H$ be a solution of  \eqref{sistema2}.
			We will denote the corresponding soliton associated with the \eqref{sistema2}  the traveling along the line $x=x_j-\omega_jt$  is defined by    $\mathbf{R}_j=\left(R_j^{(1)},R_{j}^{(2)}\right)$ such that 
			\[
			R_{j}^{(m)}=\Phi_{\omega_{j}}^{(m)}\left(x- \omega_{j} t-x_{j}\right), \qquad     m=1,2,
			\]
			with  $R_{j}^{(2)}=-\omega_jR_{j}^{(1)}$.
		   
						Our main goal is to demonstrate the existence of solutions to the system \eqref{sistemaeq} that exhibit asymptotic behavior similar to a sum of different solitons. In other words, we aim to find solutions that can be approximated by a combination of solitary waves with different wave speeds.
			
			\begin{definition}\label{multisolitons}
				Let $N \in \mathbb{N} \setminus\{0,1\}$. Consider $\omega_{1}, \ldots, \omega_{N}\in\R$, $x_{1}, \ldots, x_{N} \in \R$, and $\mathbf{\Phi}_{\omega_{1}}, \ldots, \mathbf{\Phi}_{\omega_{N}} \in H$ be solutions of \eqref{sistema2}. We denote by
\[
				R_{1}(x,t):=\sum_{j=1}^{N} R_{j}^{(1)}(x,t), \quad 
				R_{2}(x,t):=\sum_{j=1}^{N} R_{j}^{(2)}(x,t).
\]
				A multi-soliton is a solution $\mathbf{u}$ of \eqref{sistema1} for which there exists $T_{0} \in \R$ such that $\mathbf{u}$ is defined on $[T_{0}, +\infty)$ and
				\[\lim_{t \rightarrow +\infty}\|\bu(t)-\ru(t)\|_{H}=0,\]
				where $\ru=(R_{1},R_{2})$.
			\end{definition}
			
			It is important to note that $\mathbf{R}_{j}=\left(R_j^{(1)},R_{j}^{(2)}\right)$ is a solution of system \eqref{sistema1}, but due to the nonlinearity of the system, the vector solution $\mathbf{R}$  
  does not remain a solution.

			\subsection{Construction}
First, we will lay the groundwork by presenting the main result, followed by constructing arguments to support this finding.
			\begin{theorem}\label{maintheorem}
				For $j\in \{1,2,\ldots,N\}$ ,  $|\omega_{j}|\leq 1$, and $x_{j} \in \R$, let $\left(\Phi_{\omega_{j}}\right)$ be the associated ground state profiles. Define
\[
				R_{j}^{(m)}(t, x):=\Phi_{\omega_{j}}^{(m)}\left(x-\omega_{j} t-x_{j}\right), \qquad m=1,2,
\]
				and 
\[
				\begin{aligned}
					\omega_{\star}:=\frac{1}{256}\min \left\{1-\omega_{j}^2, \left|\omega_{j}-\omega_{k}\right|:  j, k=1, \ldots, N, j \neq k\right\}. 
				\end{aligned}
\]
				If $\omega_{j} \neq \omega_{k}$ for any $j \neq k$, then there exist $T_{0} \in \R$ and a solution $\bu$ for \eqref{sistemaeq} defined on $\left[T_{0},+\infty\right)$ such that for all $t \in\left[T_{0},+\infty\right)$ the following estimate holds
				\begin{equation}\label{desprin}
					\left\|\bu(t)- \ru(t)\right\|_{H}
					\leq e^{- \omega_{\star}^{\frac32} t}.
				\end{equation}
			\end{theorem}
			Before proceeding with the demonstration of this result, it is necessary to clarify that the proof will be conducted under two premises: firstly, assuming the subcritical case $p<2$ for the system \eqref{sistema1}, and subsequently, under conditions of the supercritical case $p>2$.
   
			\vspace{3mm}
   
			To proceed with the proof of the main theorem, we first define a sequence of solutions known as approximate solutions. After deriving some estimates and taking the limit, we construct the solution we seek. Specifically, let ${T^{n}}$ be an increasing sequence converging to infinity. Our strategy involves solving a final-value problem associated with \eqref{sistema2}, where the final data consists of a pair of solitons at time $T^{n}$. More precisely, for each $n \in \mathbb{N}$, let $\mathbf{u}^{n}$ denote the solution of \eqref{sistema2} defined on the interval $(T_{n}^{\star},T^{n}]$, where $T_{n}^{\star}$ is the maximum time of existence for each $n$, and such that $\mathbf{u}^{n}(T^{n}) =\mathbf{R}(T^{n})$.
			
			We will  establish that our approximate solutions $\mathbf{u}^{n}$ indeed satisfy the conclusion of the main theorem.
						\begin{proposition}[Uniform Estimates]\label{UniformEstimates}
				There exist $T_{0}\in \R$ and $n_{0} \in\N$ such that, for every $n \geq n_{0}$, each approximate solution $\mathbf{u}^{n}$ is defined in $[T_{0},T^{n}]$ and for all $t \in [T_{0},T^{n}]$
				\begin{equation}\label{desiesti}
					\|\mathbf{u}^{n}(t)-\mathbf{R}(t)\|_{H }\leq e^{-\omega_{\star}^{\frac32}t}.
				\end{equation}
			\end{proposition}
   
The proof of the main theorem will be achieved using the following crucial result.

			\begin{proposition}\label{Compactness}
				There exists $\mathbf{u}^{0} \in H $ such that, up to a subsequence, $\mathbf{u}^{n}(T_{0}) \to\mathbf{u}^{0} $ strongly in $H^s(\R)\times H^{s-1}(\R)$.
			\end{proposition}
			
			Now, we define some operators that will be used later on. Indeed,
			let us $\mathcal{S}_{j}: H  \to \R$ the functional defined by
			\begin{equation}\label{defSj}
				\mathcal{S}_{j}(u_1,u_2):=\mathcal{E}(u_1,u_2)+\omega_j \mathcal{M}_b(u_1,u_2).
			\end{equation}
			Note that by the Sobolev embedding, these functionals are well-defined.

			\begin{remark}\label{observa1}
				Differentiating the operator $\mathcal{S}_{j}$ in the sense of Gateaux, we obtain for every $\bu= (u_1,u_2) \in H $ that
				\begin{equation}
					\mathcal{S}_{j}'(\bu)=\begin{pmatrix}
						-\partial_{xx} u_{1}+u_{1}+\omega_j u_2-|u_1|^{2p}u_{1}\\
						u_2+\omega_j u_1
					\end{pmatrix}
				\end{equation}
				Note that $\ru_{j}=(R_{j}^{(1)},R_{j}^{(2)})$   
				satisfies
				\begin{equation}\label{pcrit}
					\begin{cases}
						-\partial_{xx} R_{j}^{(1)}+R_{j}^{(1)}+\omega_j R_{j}^{(2)}-|R_{j}^{(1)}|^{2p}R_{j}^{(1)}=0,\\
						R_{j}^{(2)}+\omega R_{j}^{(1)}=0,
					\end{cases}
				\end{equation}
				so that, for fixed $t$, we see that $\ru_{j}(\cdot,t)$ is a critical point of $\mathcal{S}_{j}$.		
							\end{remark}
       
	Let us define the linear action
			$\mathcal{H}_{j}$, for all $t \in \R$ and $\mathbf{w} \in H$, by
			\begin{equation}\label{defH0}
				\mathcal{H}_{j}(\mathbf{w},t):=\langle \mathcal{S}_{j}''(\Phi_{\omega_j}(t))\mathbf{w},\mathbf{w} \rangle,
			\end{equation}
where
				\begin{equation}
					\mathcal{S}_{j}''(\Phi_{\omega_j})\bw =\begin{pmatrix}
						-\partial_{xx} w_{1}+w_1+\omega_j w_2-(2p+1)|\Phi_{\omega_j}|^{2p}w_1\\
						w_2+\omega_j w_1
					\end{pmatrix}.
				\end{equation}
		Finally we close this section by presenting the following coercivity result from  \cite[Proposition 3.4]{bing}, linked to the operator $\mathcal{H}_j.$
			\begin{proposition}\label{coercivity}
				Let $j \in \N$. Consider $|\omega_j|<1$ and $0<p<2$. If $\mathbf{\eta}=(\eta_1, \eta_2) \in  H $ satisfies
$	 
				\left\langle\mathbf{\eta}, \partial_x \Phi_{\omega_j}\right\rangle=\left\langle\mathbf{\eta}, \mathbf{\Gamma}_{\omega_j}\right\rangle=0,
$ 
				where $\mathbf{\Gamma}_{\omega_j}=\left(\Phi_{\omega_j}^{(1)}, 0\right)$, then
$ 
				\left\langle S_j^{\prime \prime}\left(\mathbf{\Phi}_{\omega_j}\right) \mathbf{\eta}, \mathbf{\eta}\right\rangle \geq C\|\mathbf{\eta}\|_{ H }^2 .
$	 
			\end{proposition}

In the remainder of this paper, we intend to prove Theorem \ref{maintheorem} in two different contexts: in the subcritical case 
$p<2$ and subsequently in the supercritical case 
$p>2$.

In this regard, in Section \ref{subc-section}, we will prove Theorem \ref{maintheorem} in the subcritical case \( p < 2 \) as follows. The strategy to demonstrate this result depends on two main steps. First, we show that the approximate solutions satisfy the estimate \eqref{desprin} on \([T_0, T_n]\) with \( T_0 \) independent of \( n \). Then, we show that the sequence of initial data at \( T_0 \) is compact. Therefore, we can extract an initial datum that gives rise to a solution of the system \eqref{sistema1} satisfying the conclusion of Theorem \ref{maintheorem}. The strategy we will use to obtain these two results relies on applying a bootstrap argument, for which it is necessary to obtain some uniform estimates that will be complemented by the existence of a coercivity property. We will use modulation arguments that will allow us to control some directions obtained in the coercive property and subsequently localize certain expressions to control the different dispersions found in our arguments.

Subsequently, in Section \ref{supercritical}, our objective is the same as in the previous section, with the only difference being that we will work in the supercritical case \( p > 2 \). The strategy we will use in this section will be the same as previously applied, with the difference that under these conditions, it is not possible to control the directions generated in the coercive property in the same way. Therefore, we must resort to applying some spectral theory arguments, using the Pego-Weinstein arguments, to generate a set of eigenfunctions that allows us to obtain a new type of coercive property with better terms to be controlled. In summary, the idea in this section is based on obtaining a new coercive property and subsequently controlling the terms that form it through modulation and parameter localization arguments.

			\section{Existence of multisolitons in $L^2$-Subcritical case $p<2$}\label{subc-section}
			
			To show the existence of multisolitons in the subcritical case, we consider a sequence $T^n\to\infty$ and take $\mathbf{u}^{n}\in H$ as the solution of \eqref{sistema1} such that $\mathbf{u}^{n}(T^n)=\mathbf{R}(T^n)$. In order to prove the bootstrap estimate, we assume the existence of $T_0$ such that if $t_{0}\in [T_{0},T^{n}]$, then for all $t \in\left[t_{0}, T^{n}\right]$
			\begin{equation}\label{2hipoteseBoostrap}
				\left\|\mathbf{u}^{n}(t)-\mathbf{R}(t)\right\|_{ H  } \leq e^{-\omega_{\star}^{\frac32} t }.  
			\end{equation}
			Next, we define $\bar{\mathbf{R}}=\sum_{j=1}^{N}\left(\bar{R}_{j}^{(1)}, \bar{R}_{j}^{(2)}\right)$, where
	\[
			\bar{R}_{j}^{(1)}(x):=\Phi_{\omega_{ j}}^{(1)}(x-x_{j}) 
	\quad\text{and}\quad
			\bar{R}_{j}^{2}(x):=\Phi_{\omega_{ j}}^{(2)}(x-x_{j}),
	 \]
			and for $\alpha>0$, we define the neighborhood $\mathbf{B}(\alpha):=\{\mathbf{u}\in  H :\|\mathbf{u}-\bar{\mathbf{R}}\|_{ H }< \alpha\}.$
					Now we establish a modulation result. This type of result is widely used in various situations (see, for example, \cite{MartelMerleTsai}, \cite{Del}). However, for completeness, we will provide some details of the proofs with adaptations to our case.
			\begin{lemma}\label{lemamodulation0}
				There exist $\alpha_{1}, C > 0$ and $\mathcal{C}^{1}$-class functions
				\[
				\tilde{\omega}_{j}:\mathbf{B}(\alpha)\rightarrow(0,+\infty), \quad \tilde{x}_{j}:\mathbf{B}(\alpha) \rightarrow \R, \quad j=1 ,2, \ldots,N,
				\]
				such that if $\mathbf{u} \in \mathbf{B}(\alpha_{1})$
				then the function $
				\var=\mathbf{u}-\tilde{\mathbf{R}}
				$
				with 
				\[\tilde{\mathbf{R}}=\left(\sum_{j=1}^N \tilde{R}_{j}^{(1)},\sum_{j=1}^N \tilde{R}_{j}^{(2)}\right),\]
				satisfies the orthogonality conditions for all $j=1 ,2, \ldots,N$:
				\begin{equation}\label{orthogonality0}
					\left(\varepsilon_{m}, \tilde{R}_{j}^{(m)}\right)_{L^2(\R)}=\left(\varepsilon_{m}, \partial_{x} \tilde{R}_{j}^{(m)}\right)_{L^2(\R)}=0, \qquad m=1,2,
				\end{equation}
    where    the modulated waves $\tilde{R}_{j}^{(m)}$ are defined as
				\begin{equation}\label{Rtilde}
					\tilde{R}_{j}^{(m)}(x):=\Phi_{\tilde{\omega}_{j}}^{(m)}\left(x-\tilde{x}_{j}\right),\qquad m=1,2.
				\end{equation}
				Moreover, if $\mathbf{u} \in \mathbf{B}(\alpha)$ for all $0<\alpha<\alpha_{1}$, it follows that
				\begin{equation}\label{segundoboostrap0}
					\|\var\|_{H}+\sum_{j=1}^{N}\paar{|\tilde{\omega}_{j}-\omega_{j}|+|\tilde{x}_{j}-x_{j}|}\leq C\alpha.
				\end{equation}
			\end{lemma}
			\begin{proof}
				We fix $m$, $j$, and $\alpha_{1} > 0$. Without loss of generality and for the sake of simplicity, we will remove the dependence on $m$ and $j$ in the solitons. The idea is to apply the implicit function theorem to the function $F:B(\alpha_{1})\times (0,\infty)\times \R \rightarrow \R^{2} $ defined by
\[
				F(u, \omega, y)\equiv  \left(\begin{array}{c}
					F_{1} \\
					F_{2} 
				\end{array}\right)=\left(\begin{array}{c}
					(u-R(\omega, y), R(\omega, y))_{L^2(\R)} \\
					\left(u-R(\omega, y), \partial_{x} R(\omega, y)\right)_{L^2(\R)}
				\end{array}\right),\]
				where $R(\omega,y):= \Phi_{\omega}(x-y)$ and  $\Phi_\omega$ is the ground state of \eqref{ellipticeq}.
								Note that, clearly, $F(q_0)=0$, with $q_{0}=(\bar R, \omega, y)$. Now, since $\Phi_{\omega}$ is an even, positive function, we deduce from the facts
\[
				\int_{\R} \Phi_\omega \partial_\omega \Phi_\omega \,dx=\partial_\omega\left\|\Phi_\omega\right\|_{L^2(\R)}^2 
				\]
				and 
\[
				\left\|\Phi_\omega\right\|_{L^2(\R)}^2=\left(1-\omega^2\right)^{\frac{1}{p}-\frac{1}{2}}\left\|\Phi\right\|_{L^2(\R)}^2 
				\]
    that
\[
				\begin{aligned}
					&\partial_{\omega}F_{1}= -\int_{\R} \partial_{\omega} \Phi_{\omega}(x-y)  \Phi_{\omega}(x-y) dx=-2\int_{0}^{\infty} \partial_{\omega} \Phi_{\omega}(x)  \Phi_{\omega}(x) dx\neq 0,
				\end{aligned}
				\]
		 				Hence, we get
			\[
				\partial_\omega\left\|\Phi_\omega\right\|_{L^2(\R)}^2=-\left(\frac{2}{p}-1\right) \frac{\omega}{1-\omega^2}\left\|\Phi_{\omega}\right\|_{L^2(\R)}^2\neq 0.
				\]
				Also, there holds that
		\[
				\begin{aligned}
					\partial_{y}F_{2}&=-\left(\partial_{y} R(\omega, y), \partial_{x} R(\omega, y)\right)_{L^2(\R)}v=-\left(-\partial_{x} R(\omega, y), \partial_{x} R(\omega, y)\right)_{L^2(\R)}=\left\|\partial_{x} \Phi_{\omega}\right\|_{L^{2}(\R)}^{2}>0.
				\end{aligned}
				\]
				Furthermore, through similar calculations, it becomes apparent that $\partial_{y}F_{1}=0$, leading to $\det(DF(q_{0}))\neq 0$. Consequently, for sufficiently small $\alpha_1>0$, there exist unique parameters $( \tilde{\omega}_{j}, \tilde{x}_{j})$ mapping from $B_{j}(\alpha_{1})$ to $(0,\infty)\times \R $, such that the functions in \eqref{Rtilde} satisfy the orthogonality conditions \eqref{orthogonality0}.

				For the second part,  we take $0 < \alpha < \alpha_{1}$ and $\mathbf{u}\in \mathbf{B}(\alpha)$. According to the Mean Value Theorem,
			\[
				\begin{aligned}
					\left|\tilde{\omega}_{j}-\omega_{j}\right|&=\left|\tilde{\omega}_{j}(u_{j})-\tilde{\omega}_{j}\left(\bar{R}_{j}^{(m)}\right)\right|  \leq C\left\|u_{j}-\bar{R}_{j}^{(m)}\right\|_{H^{1}(\R)} \leq C \alpha.
				\end{aligned}
				\]
				Similarly, we can obtain estimates for $\tilde{x}_{j}$. 
    
    Finally, we derive
				\[\|\varepsilon_{j}\|_{H^{1}(\R)}\leq \|u_{j}-\bar{R}_{j}^{(m)}\|_{H^{1}(\R)}+\|\tilde{R}_{j}^{(m)}-\bar{R}_{j}^{(m)}\|_{H^{1}(\R)}\leq \alpha+\|\tilde{R}_{j}^{(m)}-\bar{R}_{j}^{(m)}\|_{H^{1}(\R)}.\]
		As
\[\begin{aligned}
					&\|\tilde{R}_{j}^{(m)}-\bar{R}_{j}^{(m)}\|_{H^{1}(\R)}v\leq \left\|\Phi_{\tilde{\omega}_{j}}^{(m)}\left(x-\tilde{x}_{j}\right)-\Phi_{\omega_{j}}^{(m)}\left(x-x_{j}\right)\right\|_{H^{1}(\R)}\\
					&\leq \left\|\Phi_{\tilde{\omega}_{j}}^{(m)}\left(x-\tilde{x}_{j}\right)-\Phi_{\omega_{j}}^{(m)}\left(x-\tilde{x}_{j}\right)\right\|_{H^{1}(\R)}+\left\|\Phi_{{\omega}_{j}}^m\left(x-\tilde{x}_{j}\right)-\Phi_{\omega_{j}}^{(m)}\left(x-x_{j}\right)\right\|_{H^{1}(\R)}\\
					&\leq C\left|\tilde{x}_{j}-x_{j}\right|+C\left|\tilde{\omega}_{j}-\omega_{j}\right|v\leq C\alpha,
				\end{aligned}
	\]
				we have that  estimate \eqref{segundoboostrap0} follows immediately.
			\end{proof}
			
			We fix $n \in \mathbb{N}$ and, to simplify notation, we denote $\mathbf{u}^{n}$ by $\mathbf{u}$, where $\mathbf{u}^{n}$ is as above.
			\begin{lemma}\label{lemamodulation}
				There exists $C > 0$ such that if $T_{0}$ is sufficiently large, then there exist $\mathcal{C}^{1}$-class functions
	\[
				\tilde{\omega}_{j}:\left[t_{0}, T^{n}\right] \rightarrow(-1,1), \quad \tilde{x}_{j}:\left[t_{0}, T^{n}\right] \rightarrow \R, \quad j=1,2, \ldots, N,
\]
				such that  
$		 
				\var=\mathbf{u}(t)-\tilde{\mathbf{R}}(t) $ with $ t \in [t_{0}, T^{n}],
$
    satisfies, for $j=1,2, \ldots, N$ and for all $t \in [t_{0}, T^{n}]$, the orthogonality conditions
				\begin{equation}\label{orthogonality}
					\left(\varepsilon_{m}(t), \tilde{R}_{j}^m(t)\right)_{L^2(\R)}=\left(\varepsilon_{m}(t), \partial_{x} \tilde{R}_{j}^m(t)\right)_{L^2(\R)}=0, \qquad m=1,2,
				\end{equation}
    where $ \tilde{\mathbf{R}}=\sum_{j=1}^N\left( \tilde{R}_{j}^{(1)}, \tilde{R}_{j}^{(2)}\right)$
     and the the modulated waves are defined as
				\[\tilde{R}_{j}^{(m)}(x,t):= \Phi_{\tilde{\omega}_{j}(t)}^{m}\left(x-\omega_{j
				}t-\tilde{x}_{j}(t)\right), \qquad m=1,2.
				\]
			 Furthermore, there hold  for any $t \in \left[t_{0}, T^{n}\right]$ that
				\begin{equation}\label{segundoboostrap}
					\|\var\|_{H}+\sum_{j=1}^{N}\paar{|\tilde{\omega}_{j}(t)-\omega_{j}|+|\tilde{x}_{j}(t)-x_{j}|}\leq Ce^{-\omega_{\star}^{\frac32}t}.
				\end{equation}
				and
				\begin{equation}\label{modulated}
					\begin{gathered}
						\sum_{j=1}^{N}\left(\left|\partial_{t} \tilde{\omega}_{j}(t)\right|+\left|\partial_{t } \tilde{x}_{j}(t)\right|\right) 
						\leq C \left(\|\mathbf{\var}(t)\|_{H} +e^{-3\omega_{\star}^{\frac32}t}\right).
					\end{gathered}
				\end{equation}
			\end{lemma}
			\begin{proof}
				Let us fix $m$, and take $t \in [t_{0}, T^{n}]$. From \eqref{2hipoteseBoostrap}, we have
				\[
				\left\| u_{j}\left( t\right)- \sum_{j=1}^{N}\Phi_{\omega_{j}}^{(m)}\left(x-\omega_{j
				}t-x_{j}\right)\right\|_{H^{1}(\R)}\leq e^{-\omega_{\star}^{\frac32}t},
				\]
				which is equivalent to
				\[
				\left\|  u_{j}\left( x+\omega_{j }t\right)-\sum_{j=1}^{N}\Phi_{\omega_{j}}^{(m)}\left(x-x_{j}\right)\right\|_{H^{1}(\R)}\leq e^{-\omega_{\star}^{\frac32}t}.
				\]
				This means that the function
		\[
  \eta_{j}(t)=u_{j}\left( \cdot+\omega_{j }t\right),\]
				belongs to the ball $B_{j}(\alpha(t))$, where $\alpha(t)=e^{-\omega_{\star}^{\frac32}t}$. Since $\alpha(t)\to 0$ as $t\to\infty$ and the map $t \mapsto \mathbf{u}(t)$ is continuous in $H$, if $T_0$ is sufficiently large, we can apply the process of Lemma \ref{lemamodulation0} for all $j \in \{1,2,\ldots,N\}$ to obtain continuous functions
				$(\tilde{\omega}_{j}, \tilde{x}_{j}): [t_{0},T^{n}] \to (-1,1)  \times \R$
				given by $\tilde{\omega}_{j}(t)=\tilde{\omega}_{j}(\eta_{j}(t))$ and $ \tilde{x}_{j}(t)=\tilde{x}_{j}(\eta_{j}(t))$ 
				such that
	\[
				\left(\eta_{j}(t)- \Phi_{\tilde{\omega}_{j}}^{(m)}(t)\left(x-\tilde{x}_{j}(t)\right),  \Phi_{\tilde{\omega}_{j}}^{(m)}\left(x-\tilde{x}_{j}(t)\right)\right)_{L^2(\R)}=0,
	\]
	\[
				\left(\eta_{j}(t)- \Phi_{\tilde{\omega}_{j}}^{(m)}(t)\left(x-\tilde{x}_{j}(t)\right),\partial_x  \Phi_{\tilde{\omega}_{j}}^{(m)}(t)\left(x-\tilde{x}_{j}(t)\right)\right)_{L^2(\R)}=0,
\]
				which in turn are equivalent to the orthogonality conditions in \eqref{orthogonality}. 
				Indeed, since  
\[
				\begin{aligned}
					& 0=\left(\eta_{j}(t)- \Phi_{\tilde{\omega}_{j}}^{(m)}(t)\left(x-\tilde{x}_{j}(t)\right),  \Phi_{\tilde{\omega}_{j}}^{(m)}(t)\left(x-\tilde{x}_{j}(t)\right)\right)_{2} \\
					& =\left( u_{j}(t)- \Phi_{\tilde{\omega}_{j}}^{(m)}\left(x-\omega_{j }t-\tilde{x}_{j}(t)\right),  \Phi_{\tilde{\omega}_{j}}^{(m)}\left(x-\omega_{j }t-\tilde{x}_{j}(t)\right)\right) \\
					&=\left(u_{j}(t)-\tilde{R}_{j}^{(m)}(t), \tilde{R}_{j}^{(m)}(t)\right).
				\end{aligned}
\]
				Proceeding analogously, we obtain the remaining orthogonality conditions. Estimate \eqref{segundoboostrap} follows exactly as in Lemma \ref{lemamodulation0}.
								Moreover, 
\[
				\begin{aligned}
					\left|\tilde{\omega}_{j}(t)-\omega_{j}\right|&=\left|\tilde{\omega}_{j}(\eta_{j}(t))-\tilde{\omega}_{j}\left(\bar{R}_{j}^{(m)}\right)\right| \lesssim\left\|\eta_{j}(t)-\bar{R}_{j}^{(m)}\right\|_{H^{1}(\R)}  =  \left\|u_{j}(t)-R_{j}^{m}(t)\right\|_{H^{1}(\R)}\lesssim e^{-\omega_{\star}^{\frac32}t}.
				\end{aligned}
	\]
				Similarly, we obtain estimates for $\tilde{x}_{j}(t)$.
				
				Finally, note that
\[\left\|\varepsilon_{1}(t)\right\|_{H^{1}(\R)}\leq \left\|u_{1}(t)-\sum_{j=1}^{N}R_{j}^{(1)}(t)\right\|_{H^{1}(\R)}+\left\|\sum_{j=1}^{N}\tilde{R}_{j}^{(1)}(t)-\sum_{j=1}^{N}R_{j}^{(1)}(t)\right\|_{H^{1}(\R)}.\]
								Since
 \begin{equation}\label{parametros}
	 \begin{split}
 \|\tilde{R}_{j}^{(1)}(t)-R_{j}^{(1)}(t)\|_{H^{1}(\R)} &=\left\| \Phi_{\tilde{\omega}_{j}}^{(1)}\left(x-\omega_{j }t-\tilde{x}_{j}(t)\right)- \Phi_{\omega_{j}}^{(1)}\left(x-\omega_{j }t-x_{j}\right)\right\|_{H^{1}(\R)}\\
						&\lesssim\left|\tilde{x}_{j}(t)-x_{j}\right|+ \left|\tilde{\omega}_{j}(t)-\omega_{j}\right| \lesssim e^{-\omega_{\star}^{\frac32}t},
					\end{split}
				\end{equation}
				we conclude that $\|\varepsilon_{1}(t)\|_{H^{1}(\R)}\lesssim e^{-\omega_{\star}^{\frac32}t}$. Similarly, we obtain this estimate for $\varepsilon_2$, from which we obtain \eqref{segundoboostrap}.
				
				
				To conclude the proof of the lemma, it is necessary to show the estimate \eqref{modulated}. The following evolution equation is satisfied,
				\begin{equation}\label{ecuacionLN}
					\begin{split}
						&\partial_{t} \varepsilon_{1}- \partial_{x}\varepsilon_{2} 
						=  -\partial_{t} \tilde{R}_{1}+\partial_{x } \tilde{R}_{2}\\
						&\partial_{t} \varepsilon_{2}- \partial_{x}\left(\varepsilon_{1} -\partial_{xx}\varepsilon_1-\varphi(\var_{1},\tilde{R}_1)\right)
						=  -\partial_{t} \tilde{R}_{2}+\partial_{x } \left(\tilde{R}_{1}-\partial_{xx}\tilde{R}_1\right)
					\end{split}
				\end{equation}
				where
				\begin{equation*}
					\begin{aligned}
						\varphi(\var_{1},\tilde{R}_1)&=|\tilde{R}_{1}|^{2p}\varepsilon_{1} +2p(|\tilde{R}_{1}|^2+\theta(2\mathcal{R} e(\tilde{R}_{1}\bar{\varepsilon}_{1})+|\varepsilon_{1}|^2))^{p-1}\, \mathcal{R} e(\tilde{R} _{1}\bar{\varepsilon}_{1})\varepsilon_{1}\\
						&\, \, \, +p(|\tilde{R}_{1}|^2+\theta(2\mathcal{R} e(\tilde{R}_{1}\bar{\varepsilon}_{1} )+|\varepsilon_{1}|^2))^{p-1}|\varepsilon_{1}|^2\varepsilon_{1}\\
						&\,\, \,+2p(|\tilde{R}_{1}|^2+\theta(2\mathcal{R} e(\tilde{R}_{1}\bar{\varepsilon}_{1} )+|\varepsilon_{1}|^2))^{p-1}\, \mathcal{R} e(\tilde{R}_{1}\bar{\varepsilon}_{1})\tilde{R}_ {1}\\
						&\, \, \,+p(|\tilde{R}_{1}|^2+\theta(2\mathcal{R} e(\tilde{R}_{1}\bar{\varepsilon}_{1} )+|\varepsilon_{1}|^2))^{p-1}|\varepsilon_{1}|^2\tilde{R}_{1}+\left|\tilde{R}_{1}\right|^{2p_1} \tilde{R}_{1}.
					\end{aligned}
				\end{equation*}
				Also, since (we will omit summation of $j=1$ to $N$)
				\begin{equation}\label{equamod}
					\begin{aligned}
						&-\partial_{t} \tilde{R}_{1}+\partial_{x } \tilde{R}_{2}=-\partial_{t} \tilde{\omega}_{j} \partial_{\omega} \Phi_{\tilde{\omega}_{j}}^{(1)} -\left(-\omega_j-\partial_{t} \tilde{x}_{j}\right) \partial_{x} \Phi_{\tilde{\omega}_{j}}^{(1)}-\omega_j\partial_{x } \Phi_{\tilde{\omega}_{j}}^{(1)},\\
						&-\partial_{t} \tilde{R}_{2}+\partial_{x } \left(\tilde{R}_{1}-\partial_{xx}\tilde{R}_1\right)=-\partial_{t} \tilde{\omega}_{j} \partial_{\omega} \Phi_{\tilde{\omega}_{j}}^2 -\left(-\omega_j-\partial_{t} \tilde{x}_{j}\right) \partial_{x} \Phi_{\tilde{\omega}_{j}}^2+\partial_{x } \tilde{R}_{1}-\partial_{x xx} \tilde{R}_{1},
					\end{aligned}   
				\end{equation}
				we obtain from \eqref{ecuacionLN} that
				\begin{equation}\label{ecuacionLN2}
					\begin{aligned}
						&\partial_{t} \varepsilon_{1}- \partial_{x}\varepsilon_{2}=-\partial_{t} \tilde{\omega}_{j} \partial_{\omega} \Phi_{\tilde{\omega}_{j}}^{(1)} +\partial_{t} \tilde{x}_{j} \partial_{x} \Phi_{\tilde{\omega}_{j}}^{(1)},\\
						&\partial_{t} \varepsilon_{2}- \partial_{x}\left(\varepsilon_{1} -\partial_{xx}\varepsilon_1-\varphi(\var_{1},\tilde{R}_1)\right)=-\partial_{t} \tilde{\omega}_{j} \partial_{\omega} \Phi_{\tilde{\omega}_{j}}^2+\partial_{t} \tilde{x}_{j}\partial_{x} \Phi_{\tilde{\omega}_{j}}^2-\partial_{x xx} \tilde{R}_{1}.
					\end{aligned}   
				\end{equation}
				Now, taking the inner product with $\tilde{R}_{1}$ and $\tilde{R}_{2}$, respectively in \eqref{ecuacionLN2}, we obtain
				\begin{equation}\label{equamod2}
					\begin{split}
						( \partial_{t} \varepsilon_{1},\tilde{R}_{1})_{L^2(\mathbb{R})}&-( \partial_{x}\varepsilon_{2},\tilde{R}_{1})_{L^2(\mathbb{R})}  =-\partial_{t} \tilde{\omega}_{j}( \partial_{\omega} \Phi_{\tilde{\omega}_{j}}^{(1)},\tilde{R}_1)_{L^2(\mathbb{R})} +\partial_{t} \tilde{x}_{j} (\partial_{x} \Phi_{\tilde{\omega}_{j}}^{(1)},\tilde{R}_1)_{L^2(\mathbb{R})}).
					\end{split}
				\end{equation}
				and
				\begin{equation}\label{2equamod2}
					\begin{split}
						( \partial_{t} \varepsilon_{2}&,\tilde{R}_{2})_{L^2(\mathbb{R})}-( \partial_{x}\varepsilon_{1},\tilde{R}_{2})_{L^2(\mathbb{R})}  +(\partial_{xxx}\varepsilon_1,\tilde{R}_{2})_{L^2(\mathbb{R})} +(\partial_{x}\varphi(\var_{1},\tilde{R}_1),\tilde{R}_{2})_{L^2(\mathbb{R})} \\&=-\partial_{t} \tilde{\omega}_{j} (\partial_{\omega} \Phi_{\tilde{\omega}_{j}}^2,\tilde{R}_{2})_{L^2(\mathbb{R})}  +\partial_{t} \tilde{x}_{j} (\partial_{x} \Phi_{\tilde{\omega}_{j}}^2,\tilde{R}_{2})_{L^2(\mathbb{R})} -(\partial_{x xx} \tilde{R}_{1},\tilde{R}_{2})_{L^2(\mathbb{R})} ).
					\end{split}
				\end{equation}
				On the other hand, taking the inner product with $\varepsilon_1$ and $\varepsilon_2$ in \eqref{equamod} and considering that $(\varepsilon_{1},\tilde{R}_{1})_{L^2(\mathbb{R})}=0$, we have  
				\begin{equation}\label{eq1}
					\begin{aligned}
						-\left(\varepsilon_{1},  \partial_{t} \tilde{R}_{1}\right)_{L^2(\mathbb{R})}&+\left(\varepsilon_{1},\partial_{x} \tilde {R}_{2}\right)_{L^2(\mathbb{R})} =-\partial_{t} \tilde{\omega}_{j}\left(\varepsilon_{1},  \partial_{\omega} \Phi_{\tilde{\omega}_{j}}^{(1)}\right)_{L^2(\mathbb{R})}  +\partial_{t} \tilde{x}_{j}\left(\varepsilon_{1},  \partial_{x}\Phi_{\tilde{\omega}_{j}}^{(1)}\right)_ {L^2(\mathbb{R})}.
					\end{aligned}
				\end{equation}
				and
				\begin{equation}\label{eq2}
					\begin{aligned}
						-\left(\varepsilon_{2},  \partial_{t} \tilde{R}_{2}\right)_{L^2(\mathbb{R})}&+\left(\varepsilon_{2},\partial_{x} \tilde {R}_{1}\right)_{L^2(\mathbb{R})}\\ &=-\partial_{t} \tilde{\omega}_{j}\left(\varepsilon_{2},  \partial_{\omega} \Phi_{\tilde{\omega}_{j}}^2\right)_{L^2(\mathbb{R})}  +\partial_{t} \tilde{x}_{j}\left(\varepsilon_{1},  \partial_{x}\Phi_{\tilde{\omega}_{j}}^2\right)_ {L^2(\mathbb{R})}-(\varepsilon_2,\partial_{x xx}\tilde{R}_1)_ {L^2(\mathbb{R})}.
					\end{aligned}
				\end{equation}
				Since \[\left(\partial_{t}\varepsilon_{1}, \tilde{R}_{1}\right)_{L^2(\mathbb{R})}=-\left(\varepsilon_{ 1}, \partial_{t} \tilde{R}_{1}\right)_{L^2(\mathbb{R})},\qquad\left(\partial_{t}\varepsilon_{2}, \tilde{R}_{2}\right)_{L^2(\mathbb{R})}=-\left(\varepsilon_{ 2}, \partial_{t} \tilde{R}_{2}\right)_{L^2(\mathbb{R})},\]
    then by combining \eqref{equamod2}-\eqref{eq2}
				we obtain 
				\begin{equation*}
					\begin{aligned}
						&-\partial_{t} \tilde{\omega}_{j} (\partial_{\omega} \Phi_{\tilde{\omega}_{j}}^{(1)},\tilde{R}_1)_{L^2(\mathbb{R})} +\partial_{t} \tilde{x}_{j} (\partial_{x} \Phi_{\tilde{\omega}_{j}}^{(1)},\tilde{R}_1)_{L^2(\mathbb{R})}) -\partial_{t} \tilde{\omega}_{j} (\partial_{\omega} \Phi_{\tilde{\omega}_{j}}^2,\tilde{R}_{2})_{L^2(\mathbb{R})} \\
						&\quad +\partial_{t} \tilde{x}_{j} (\partial_{x} \Phi_{\tilde{\omega}_{j}}^2,\tilde{R}_{2})_{L^2(\mathbb{R})}  +(\partial_{x xx} \tilde{R}_{1},\tilde{R}_{2})_{L^2(\mathbb{R})} )\\
						&\quad -(\partial_{xxx}\varepsilon_1,\tilde{R}_{2})_{L^2(\mathbb{R})} -(\partial_{x}\varphi(\var_{1},\tilde{R}_1),\tilde{R}_{2})_{L^2(\mathbb{R})}\\
						&=-\partial_{t} \tilde{\omega}_{j}\left(\varepsilon_{1},  \partial_{\omega} \Phi_{\tilde{\omega}_{j}}^{(1)}\right)_{L^2(\mathbb{R})} -\partial_{t} \tilde{\omega}_{j}\left(\varepsilon_{2},  \partial_{\omega} \Phi_{\tilde{\omega}_{j}}^2\right)_{L^2(\mathbb{R})}+\partial_{t} \tilde{x}_{j}\left(\varepsilon_{1},  \partial_{x}\Phi_{\tilde{\omega}_{j}}^2\right)_ {L^2(\mathbb{R})} -(\varepsilon_2,\partial_{x xx}\tilde{R}_1)_ {L^2(\mathbb{R})}\\&\quad+\partial_{t} \tilde{x}_{j}\left(\varepsilon_{1},  \partial_{x}\Phi_{\tilde{\omega}_{j}}^{(1)}\right)_ {L^2(\mathbb{R})}.
					\end{aligned}  
				\end{equation*}
				Therefore, since $\Phi_{\omega}$ is even and $\partial_{x}\Phi_{\omega}$ is old, then
				\begin{equation*}
					\begin{aligned}
						&-\partial_{t} \tilde{\omega}_{j} \left((\partial_{\omega} \Phi_{\tilde{\omega}_{j}}^{(1)},\tilde{R}_1)_{L^2(\mathbb{R})} +(\partial_{\omega} \Phi_{\tilde{\omega}_{j}}^2,\tilde{R}_{2})_{L^2(\mathbb{R})} -\left(\varepsilon_{1},  \partial_{\omega} \Phi_{\tilde{\omega}_{j}}^{(1)}\right)_{L^2(\mathbb{R})} -\left(\varepsilon_{2},  \partial_{\omega} \Phi_{\tilde{\omega}_{j}}^2\right)_{L^2(\mathbb{R})}\right)\\
						&\quad +\partial_{t} \tilde{x}_{j} \left(\left(\varepsilon_{1},  \partial_{x}\Phi_{\tilde{\omega}_{j}}^2\right)_ {L^2(\mathbb{R})}+\left(\varepsilon_{1},  \partial_{x}\Phi_{\tilde{\omega}_{j}}^{(1)}\right)_ {L^2(\mathbb{R})}\right)\\
						&=-(\partial_{x xx} \tilde{R}_{1},\tilde{R}_{2})_{L^2(\mathbb{R})} )  +(\partial_{xxx}\varepsilon_1,\tilde{R}_{2})_{L^2(\mathbb{R})} +(\partial_{x}\varphi(\var_{1},\tilde{R}_1),\tilde{R}_{2})_{L^2(\mathbb{R})}+(\varepsilon_2,\partial_{x xx}\tilde{R}_1)_ {L^2(\mathbb{R})}.
					\end{aligned}  
				\end{equation*}
    
				This last equation can be written as
\[
				(M_{11}(t)+a_{11}(t)) \partial_{t} \tilde{\omega}_{j}+a_{12}(t)\partial_{t} \tilde{x}_{j}=b_{1}(t),
\]
				where
\[
a_{11}(t):=\left(\left(\varepsilon_{1},  \partial_{\omega} \tilde{R}_1\right)_{L^2(\mathbb{R})} +\left(\varepsilon_{2},  \partial_{\omega} \tilde{R}_2\right)_{L^2(\mathbb{R})}\right),
\] 
\[
a_{12}(t):=\left(\left(\varepsilon_{1},  \partial_{x}\tilde{R}_2\right)_ {L^2(\mathbb{R})}+\left(\varepsilon_{1},  \partial_{x}\tilde{R}_1\right)_ {L^2(\mathbb{R})}\right) \, \,(=0),
\]
\[
M_{11}(t)=(\partial_{\omega} \tilde{R}_1,\tilde{R}_1)_{L^2(\mathbb{R})} +(\partial_{\omega} \tilde{R}_2,\tilde{R}_{2})_{L^2(\mathbb{R})}
\]
and
\[
\begin{aligned} b_1(t)&= (\partial_{xxx}\varepsilon_1,\tilde{R}_{2})_{L^2(\mathbb{R})} +(\partial_{x}\varphi(\var_{1},\tilde{R}_1),\tilde{R}_{2})_{L^2(\mathbb{R})}+(\varepsilon_2,\partial_{x xx}\tilde{R}_1)_ {L^2(\mathbb{R})}.
				\end{aligned}
\] 
It is clear that $\left|a_{11}(t)\right |+\left|a_{12}(t)\right| \leq C\|\mathbf{\var}(t)\|_{H}+e^{-3 \omega_{\star}^{\frac32}t}$ for all $t \in\left[t_{0}, T^{n}\right]$. To estimate the coefficient $b_{1}(t)$, we observe that all terms can be bounded by
				$\|\varepsilon_{1}(t)\|_{H^{1}(\R)}$, except for the term  $(\left|\tilde{R}_{1}\right|^{2p_1} \tilde{R}_{1},\tilde{R}_{2})_{L^{2}(\mathbb{R})}$, which appears in the last term of the definition of $\varphi$. We can treat this term as follows: since $\tilde{R}_j$ is bounded, we can assert that
\[
				(\left|\tilde{R}_{1}\right|^{2p_1} \tilde{R}_{1},\tilde{R}_{2})_{L^{2}(\mathbb{R})}\leq C\int_{\mathbb{R}} \left|\tilde{R}_{1}\right|\left|\tilde{R}_{2}\right|\, d x \leq C\int_{\mathbb{R}} e^{-\frac{3}{8} \sqrt{\tilde{\omega}_{j}}\left|x-\omega_{j }t-\tilde{x}_{j}\right|} e^{-\frac{3}{8} \sqrt{\tilde{\omega}_{k}}\left|x-\omega_{k}t-\tilde{x}_{k}\right|}  \,dx
\]
				From \eqref{segundoboostrap}, we can assume
				$\tilde{\omega}_{j}(t)\geq \frac{1}{2}\omega_{j}$ (for $T_0$ 
				sufficiently large). Then 
				\[ 
				\begin{aligned}
					(\left|\tilde{R}_{1}\right|^{2p_1} \tilde{R}_{1},\tilde{R}_{2})_{L^{2}(\mathbb{R})}
					&\lesssim\int_{\mathbb{R}} e^{-3\sqrt{1-\omega_{\star}^2}\left\{\left|x-\omega_{j }t-\tilde{x}_{j}\right|+\left|x-\omega_{k}t-\tilde{x}_{k}\right|\right\} } e^{-3 \sqrt{1-{\omega}_{\star}^2}\left\{\left|\omega_{j }t-\omega_{k}t\right|\right\} }  \,dx\\
					&\lesssim e^{-3\omega_{\star}^{\frac32}t}\int_{\mathbb{R}} e^{-3\sqrt{1-{\omega}_{\star}^2}\left\{\left|x-\omega_{j }t-\tilde{x}_{j}\right|+\left|x-\omega_{k}t-\tilde{x}_{k}\right|\right\} } \,dx \lesssim e^{-3 \omega_{\star}^{\frac32}t}.
				\end{aligned}
				\]
				Therefore,
				\[\left|b_{1}(t)\right|\lesssim \|\var(t)\|_{H}+ e^{-3 \omega_{\star}^{\frac32}t} .\]
				Repeating the arguments above, but now taking the inner product in \eqref{ecuacionLN2} with
				$ \partial_{x} \tilde{R}_{1} $, $ \partial_{x} \tilde{R}_{2} $ and in \eqref{equamod} by $ \partial_{x} \var_1 $, $ \partial_{x} \varepsilon_{2} $, we derive
				
	\[
				a_{21}(t)\partial_{t} \tilde{\omega}_{j}+a_{22}(M_{22}(t)+a_{22}(t)) \partial_{t} \tilde{x}_{j}=b_{2}(t),
				\]
				where
		\[a_{21}(t):=\left(\left(\partial_x \varepsilon_{1},  \partial_{\omega} \tilde{R}_1\right)_{L^2(\mathbb{R})} +\left(\partial_x\varepsilon_{2},  \partial_{\omega} \tilde{R}_2\right)_{L^2(\mathbb{R})}\right),\]
		\[ a_{22}(t):=-\left(\left(\partial_x\varepsilon_{1},  \partial_{x}\tilde{R}_2\right)_ {L^2(\mathbb{R})}+\left(\partial_x\varepsilon_{1},  \partial_{x}\tilde{R}_1\right)_ {L^2(\mathbb{R})}\right),\quad M_{22}(t)= \left\|\partial_x\tilde{R}_2\right\|_{L^2(\R)}^2+\left\|\partial_x\tilde{R}_1\right\|_{L^2(\R)}^2\]
 and
\[
\begin{aligned} b_2(t)&=(\partial_{xxx}\tilde{R}_{1},\partial_{x}\tilde{R}_{2})_{L^2(\mathbb{R})} - (\partial_{x}\varepsilon_2,\partial_{x}\tilde{R}_{1})_{L^2(\mathbb{R})} +(\partial_{xxx}\varepsilon_1,\partial_{x}\tilde{R}_{2})_{L^2(\mathbb{R})} \\
					&\quad+(\partial_{x}\varphi(\var_{1},\partial_{x}\tilde{R}_1),\tilde{R}_{2})_{L^2(\mathbb{R})}-(\partial_{x}\varepsilon_2,\partial_{x xx}\tilde{R}_1)_ {L^2(\mathbb{R})},
				\end{aligned}
\] 
				where 
				$\left|a_{21}(t)\right |+\left|a_{22}(t)\right|+\left|b_{2}(t)\right| \leq C\|\mathbf{\var}(t)\|_{H}+e^{-3 \omega_{\star}^{\frac32}t}$ for all  $t \in\left[t_{0}, T^{n}\right]$.
				Therefore, we obtain a system of equations that can be written as
				\begin{equation}\label{matiguality}
					(M(t)+A(t)) X(t)=B(t),   
				\end{equation}
				where $X(t)=\left(\begin{array}{c} \partial_{t} \tilde{\omega}_{j}(t) \\ \omega_j-\partial_{ t} \tilde{x}_{j}(t)\end{array}\right)$
				and
	\[
				M(t)=\left(\begin{array}{ccc}
					(\partial_{\omega} \tilde{R}_1,\tilde{R}_1)_{L^2(\mathbb{R})} +(\partial_{\omega} \tilde{R}_2,\tilde{R}_{2})_{L^2(\mathbb{R})}  & 0 \\
					0 & \left\|\partial_x\tilde{R}_2\right\|_{L^2(\R)}^2+\left\|\partial_x\tilde{R}_1\right\|_{L^2(\R)}^2
				\end{array}\right).
\]
				As, for every $ t \in \left [t_{0}, T^{n} \right]$,  $M(t)$
				is an invertible matrix, then from \eqref{matiguality} we obtain
				\begin{equation*}
					\begin{aligned}
						\|X(t)\|\leq \|M^{-1}(t)B(t)\|+\|M^{-1}(t)A(t)X(t)\|,
					\end{aligned}
				\end{equation*}
				where $M^{-1}(t)$ represents the inverse of the matrix $M(t)$. Now, since $\tilde{R}_{1}$ and $\tilde{R}_{2}$ are bounded, then
				\begin{equation}\label{matdesi}
					\begin{aligned}
						\|X(t)\|\leq C\|B(t)\|+C\|A(t)X(t)\|.
					\end{aligned}
				\end{equation}
				Using the inequality \eqref{segundoboostrap}, it follows that for $T_{0}$ sufficiently large, $\|A(t)\|\leq \frac{1}{2C}$. Thus, from  \eqref{matdesi}
				\begin{equation*}
					\begin{aligned}
						\frac{1}{2}\|X(t)\|\leq C\|B(t)\|.
					\end{aligned}
				\end{equation*}
				Finally, since for all $t\in[t_{0},T^{n}]$, we have $\|B(t)\|\leq C\left(\|\var(t)\|_{ H }+e^{-3\omega_{\star}^{\frac32}t}\right)$, and then
				\begin{equation*}
					\begin{aligned}
						\|X(t)\|\leq C\left(\|\var(t)\|_{ H }+e^{-3\omega_{\star}^{\frac32}t}\right).
					\end{aligned}
				\end{equation*}
				Therefore, we obtain the desired result.
			\end{proof}
			\begin{remark}
				Using the fact that $\mathbf{u}^{n}(T^{n})=\mathbf{R}(T^{n})$ and the uniqueness of the decomposition at time $t=T^{n}$, we necessarily have
\[
\var(T^{n})=0, \, \tilde{\mathbf{R}}(T^{n})=\mathbf{R}(T^{n}), \, \tilde{\omega}_{j}(T^{n})=\omega_{j}, \, \tilde{x}_{j}(T^{n})=x_{j}, ,\,\, \text{with \, $j=1,2,\ldots,N$}.\]
			\end{remark}
			
\subsection{Control Estimates for Solitons}\label{localiz}
			
			In this subsection, we will obtain some control estimates for the solitary waves that compose $\mathbf{u}$. For this purpose, we will use an argument that involves localizing certain quantities.  Therefore, we can assume, without loss of generality, that the propagation speeds of the solitons satisfy $\omega_k \neq \omega_
			m$ for all $k \neq m$. In fact, we will assume that  $ \omega_{1}<\omega_{2}<\cdots<\omega_{N}.$ 
			Let $\psi: \R \rightarrow \R$ be a $C^{\infty}$ cutoff function such that $\psi(s)=0$ for $s<-1$, $\psi(s) \in [0,1]$ if $s \in [-1,1]$, and $\psi(s)=1$ for $s>1$. As in \cite{cote}, we define
$ m_{j}:= (\omega_{j-1}+\omega_{j})/2, $
			Now, we introduce the following cutoff functions, for all $(x,t)\in \R \times\R$, 
	\[
			\begin{array}{ll}
				\psi_{1}(x,t):=1, & \psi_{j}(x,t):=\psi\left(\frac{1}{\sqrt{t}}\left(x -m_{j} t\right)\right), \, \,\text { for} \,\,j=2, \ldots, N.
			\end{array}
			\]
			Next, we define
	$ 
			\phi_{j}=\psi_{j}-\psi_{j+1} $  for   $j=1, \ldots, N-1, $ and $ \phi_{N}=\psi_{N }.
$ 
			By the definition of $\phi_j$, we have ${\rm supp}(\phi_{1})\subset (-\infty,\sqrt{t}+m_{2}t]$, ${\rm supp}(\phi_{N})\subset [-\sqrt{t}+m_{N}t,\infty)$ and ${\rm supp}(\phi_{j})\subset [-\sqrt{t}+m_jt,\sqrt{t}+m_jt]$ for all $j=2,\ldots,N-1$.
			
			Finally, for $j=1,2,\ldots,N$, we define
   the localized  momentum and energy
			\begin{equation}\label{massaloca}
				\mathcal{M}_j(\bu)=\frac{1}{2}\int_{\R}u_1 u_2\phi_{j}\,dx, 
			\end{equation}
			\begin{equation}\label{energialoca}
				\mathcal{E}_j(\bu)=\frac{1}{2}\int_{\R}\left(|u_1|^2+|u_2|^2+|\partial_{x}u_1|^2-2\Psi(u_1)\right)\phi_{j}\,dx.  
			\end{equation}
			And we denote by $\mathcal{S}_{\text{loc}}^{j}$ the localized action defined, for all $\bw \in  H $, by
			\begin{equation}\label{definicaoSj}
				\mathcal{S}_{\operatorname{loc}}^{j}(\bw,t):=\mathcal{E}_j(\bw,t)+\omega_j\mathcal{M}_j (\bw,t).
			\end{equation}
	We also define an action-type functional for multi-solitons by
			\begin{equation}\label{Sdefin}
				\mathcal{S} (\bw,t):=\sum_{j=1}^{N} \mathcal{S}_{\operatorname{loc}}^{j}(\bw,t).    
			\end{equation}
		The result we are about to establish will allow us to control the decay of the sum of solitons, enabling the functional $\mathcal{S}$  to depend on the individual solitons rather than just their sum.
			\begin{lemma}\label{solitons}
				Let $m, n\in \{1,2\}$. There exists $C>0$ such that for all $t$ sufficiently large and for all $j\neq k\in {1, \ldots, N}$,
		\[
				\begin{gathered}
					\int_{\R}\left(\left|R_{k}^{(m)}(t)\right|+\left|\partial_{x} R_{k}^{(m)}(t)\right|\right) \phi_{j}(x,t) \,d x\lesssim e^{-4 \omega_{\star}^{\frac32}t}, \\
					\int_{\R}\left(\left|R_{k}^{(m)}(t)\right|+\left|\partial_{x} R_{k}^{(m)}(t)\right|\right)\left(1-\phi_{k}( x,t)\right)\,dx\lesssim e^{-4\omega_{\star}^{\frac32}  t},\\
					\int_{\R}\left(|R_{k}^{(m)}(t)|+|\partial_{x} R_{k}^{(m)}(t)|\right)\left(|R_{j}^{(n)}(t)|+|\partial_{x} R_{j}^{(n)}(t)|\right)\,dx\lesssim  e^{-4\omega_{\star}^{\frac32}  t}.
				\end{gathered}
	\]
			\end{lemma}
			\begin{proof}
				Let us fix $m=1$, the case $m=2$ follows similarly. Indeed, using Proposition \ref{decaimentoquadratico}, we obtain
				\begin{equation}\label{1interaccion}
					\begin{aligned}
						\int_{\R}\left|\tilde{R}_{k}^{(1)}(t)\right| \phi_{j}(x,t) \,d x &\leq \int_{\R}e^{-\frac{2}{3} \sqrt{1-\tilde{\omega}_{k}^2}| x-\omega_k t-x_{k} |} \phi_{j}(x,t) \,d x \\
						& \leq C \int_{\R} e^{-\frac{2}{3} \sqrt{1-\tilde{\omega}_{k}^2}\left|x-\omega_k t\right|} \phi_{j}(x,t) \,d x.
					\end{aligned}
				\end{equation}
				Suppose $k>j$, with $j\in {2, \ldots, N-1}$.
				If $k=j+1$, using the support properties of $\phi_j$ and \eqref{1interaccion}, we have
\[
				\begin{aligned}
					\int_{\R}\left|\tilde{R}_{k}^{(1)}(t)\right| \phi_{j}(x,t) \,d x &\lesssim \int_{\R}e^{-\frac{2}{3} \sqrt{1-\tilde{\omega}_{k}^2}| x-\omega_{j+1} t |} \phi_{j}(x,t) \,d x \lesssim \int_{-\sqrt{t}+m_{j}t}^{\sqrt{t}+m_{j+1}t}e^{-\frac{2}{3} \sqrt{1-\tilde{\omega}_{k}^2}| x-\omega_{j+1} t |} \, d x\\
					&\lesssim \int_{-\sqrt{t}+m_{j+1}t}^{\sqrt{t}+m_{j+1}t}e^{-\frac{2}{3} \sqrt{1-\tilde{\omega}_{k}^2}| x-\omega_{j+1} t |} \, d x+  \int_{-\infty}^{-\sqrt{t}+m_{j+1}t}e^{-\frac{2}{3} \sqrt{1-\tilde{\omega}_{k}^2}| x-\omega_{j+1} t |} \, d x\\
					&\lesssim \int_{-\sqrt{t}+m_{j+1}t}^{\sqrt{t}+m_{j+1}t}e^{-\frac{2}{3} \sqrt{1-\tilde{\omega}_{k}^2}| x-m_{j+1}t+m_{j+1}t-\omega_{j+1} t |} \, d x+  \int_{-\infty}^{-\sqrt{t}-\frac{\omega_{j+1}-\omega_{j}}{2}t}e^{-\frac{2}{3} \sqrt{1-\tilde{\omega}_{k}^2}| x |} \, d x\\
					&\lesssim\int_{-\sqrt{t}+m_{j+1}t}^{\sqrt{t}+m_{j+1}t}e^{-\frac{2}{3} \sqrt{1-\tilde{\omega}_{k}^2}\left| x-m_{j+1}t+\frac{\omega_{j}-\omega_{j+1}}{2}t \right|} \, d x+  \int_{-\infty}^{-\sqrt{t}-\frac{\omega_{j+1}-\omega_{j}}{2}t}e^{\frac{2}{3} \sqrt{1-\tilde{\omega}_{k}^2} x } \, d x.
				\end{aligned}
\]
				Therefore, for $k=j+1$, we get
				\begin{equation}\label{estimative0}
					\begin{aligned}					\int_{\R}\left|\tilde{R}_{k}^{(1)}(t)\right| \phi_{j}(x,t) \,d x&\lesssim e^{-\frac{1}{3}\sqrt{1-\tilde{\omega}_{k}^2}\omega_{\star}t}\int_{-\sqrt{t}+m_{j+1}t}^{\sqrt{t}+m_{j+1}t}e^{\frac{2}{3} \sqrt{1-\tilde{\omega}_{k}^2}\left| x-m_{j+1}t \right|} \, d x+  e^{-\frac{2}{3} \sqrt{1-\tilde{\omega}_{k}^2} \left(\sqrt{t}+\frac{\omega_{\star}}{2}t\right)}\\
						&\lesssim e^{-\frac{1}{3}\sqrt{1-\tilde{\omega}_{k}^2}\omega_{\star}t}\int_{-\sqrt{t}}^{\sqrt{t}}e^{\frac{2}{3} \sqrt{1-\tilde{\omega}_{k}^2}\left| x \right|} \, d x+  e^{-\frac{2}{3} \sqrt{1-\tilde{\omega}_{k}^2} \left(\sqrt{t}+\frac{\omega_{\star}}{2}t\right)}\\
						&\lesssim e^{-\frac{1}{3}\sqrt{1-\tilde{\omega}_{k}^2}\omega_{\star}t}\left(2\sqrt{t}\right)e^{\frac{2}{3} \sqrt{1-\tilde{\omega}_{k}^2}\sqrt{t}} +  e^{-\frac{2}{3} \sqrt{1-\tilde{\omega}_{k}^2} \left(\sqrt{t}+\frac{\omega_{\star}}{2}t\right)}\\
						&\lesssim e^{-\frac{1}{3}\sqrt{1-\tilde{\omega}_{k}^2}\omega_{\star}t}\left(\left(2\sqrt{t}\right)e^{\frac{2}{3} \sqrt{1-\tilde{\omega}_{k}^2}\sqrt{t}}+e^{-\frac{2}{3} \sqrt{1-\tilde{\omega}_{k}^2} \sqrt{t}}\right).
					\end{aligned}
				\end{equation}
				Since, for $t$ sufficiently large, 
				\[
				\left(2\sqrt{t}\right)e^{\frac{2}{3} \sqrt{1-\tilde{\omega}_{k}^2}\sqrt{t}}+e^{-\frac{2}{3} \sqrt{1-\tilde{\omega}_{k}^2} \sqrt{t}}\leq e^{\frac{1}{12}\sqrt{1-\tilde{\omega}_{k}^2}\omega_{\star}t},
				\]
				then, from \eqref{estimative0}, we have 
				\begin{equation}\label{estimative1}
					\int_{\R}\left|\tilde{R}_{k}^{(1)}(t)\right| \phi_{j}(x,t) \,d x
					\leq e^{-\frac{1}{4}\sqrt{1-\tilde{\omega}_{k}^2}\omega_{\star}t}.
				\end{equation}
				Now, if $k>j+1$, using the support properties of $\phi_{j}$ and \eqref{1interaccion}, we obtain
	\[
				\begin{aligned}					\int_{\R} \left|\tilde{R}_{k}^{(1)}(t)\right| \phi_{j}(x,t) \,d x &\lesssim \int_{\R}e^{-\frac{2}{3} \sqrt{1-\tilde{\omega}_{k}^2}| x-\omega_k t |} \phi_{j}(x,t) \,d x\\
					&= \int_{\R} e^{-\frac{1}{3} \sqrt{1-\tilde{\omega}_{k}^2}| x-\omega_{k} t |}e^{-\frac{1}{3} \sqrt{1-\tilde{\omega}_{k}^2}| x-\omega_k t |} \phi_{j}(x,t) \,d x\\
					&=  \int_{\R}e^{-\frac{1}{3} \sqrt{1-\tilde{\omega}_{k}^2}\left|x-m_{j} t+m_{j} t-\omega_k t \right|} e^{-\frac{1}{3} \sqrt{1-\tilde{\omega}_{k}^2}| x-\omega_k t |} \phi_{j}(x,t) \,d x\\
					&\lesssim \int_{\R}e^{-\frac{1}{3} \sqrt{1-\tilde{\omega}_{k}^2}\left|x-m_{j} t+\frac{\omega_{j-1}+\omega_{j}}{2} t-\omega_k t\right|} e^{-\frac{1}{3} \sqrt{1-\tilde{\omega}_{k}^2}| x-\omega_k t |} \phi_{j}(x,t) \,d x
					\\
					& =  \int_{\R}e^{-\frac{1}{3} \sqrt{1-\tilde{\omega}_{k}^2}\left|x-m_{j} t+\frac{\omega_{j-1}-\omega_k}{2} t+\frac{\omega_{j}-\omega_k}{2} \right|} e^{-\frac{1}{3} \sqrt{1-\tilde{\omega}_{k}^2}| x-\omega_k t |} \phi_{j}(x,t) \,d x\\
					&\lesssim e^{-\frac{1}{3} \sqrt{1-\tilde{\omega}_{k}^2}\frac{| \omega_{j}-\omega_k  |t}{2}}\int_{-\sqrt{t}+m_{j}t}^{\sqrt{t}+m_{j+1}t}e^{\frac{1}{3} \sqrt{1-\tilde{\omega}_{k}^2}\left|x-\frac{\omega_{j}+\omega_k}{2} t \right|} e^{-\frac{1}{3} \sqrt{1-\tilde{\omega}_{k}^2}| x-\omega_k t |} \,d x.
				\end{aligned}
\]
				Therefore, for $k>j+1$, 
\[
				\begin{aligned}
					\int_{\R}&\left|\tilde{R}_{k}^{(1)}(t)\right| \phi_{j}(x,t) \,d x\\
					&\lesssim e^{-\frac{1}{6} \sqrt{1-\tilde{\omega}_{k}^2}\omega_{\star}t}\int_{-\sqrt{t}+\frac{\omega_{j-1}-\omega_k}{2}t}^{\sqrt{t}+\frac{\omega_{j+1}-\omega_k}{2}t}e^{\frac{1}{3} \sqrt{1-\tilde{\omega}_{k}^2}\left|x \right|} e^{-\frac{1}{3} \sqrt{1-\tilde{\omega}_{k}^2}\left| x+\frac{\omega_{j}-\omega_k}{2}t \right|} \,d x\\
					&\lesssim e^{-\frac{1}{6} \sqrt{1-\tilde{\omega}_{k}^2}\omega_{\star}t}\left(\int_{-\sqrt{t}+\frac{\omega_{j-1}-\omega_k}{2}t}^{\sqrt{t}+\frac{\omega_{j+1}-\omega_k}{2}t}e^{\frac{1}{3} \sqrt{1-\tilde{\omega}_{k}^2}\left|x \right|}\,dx\right)^{1/2}\left(\int_{\R} e^{-\frac{1}{3} \sqrt{1-\tilde{\omega}_{k}^2}\left| x+\frac{\omega_{j}-\omega_k}{2}t \right|} \,d x\right)^{1/2}.
				\end{aligned}
	\]
				As $k>j+1$, we have $\omega_k>\omega_{j+1}$, and therefore
				\begin{equation}\label{estimative2}
	 \begin{aligned}						\int_{\R}&\left|\tilde{R}_{k}^{(1)}(t)\right| \phi_{j}(x,t) \,d x \\
						&\lesssim e^{-\frac{1}{3} \sqrt{1-\tilde{\omega}_{k}^2}\omega_{\star}t}\left(\int_{-\sqrt{t}+\frac{\omega_{j-1}-\omega_k}{2}t}^{\sqrt{t}+\frac{\omega_{j+1}-\omega_k}{2}t}e^{-\frac{1}{3} \sqrt{1-\tilde{\omega}_{k}^2}x }\,dx\right)^{1/2}\left(\int_{\R} e^{-\frac{1}{3} \sqrt{1-\tilde{\omega}_{k}^2}\left| x+\frac{\omega_{j}-\omega_k}{2}t \right|} \,d x\right)^{1/2}\\
						&\lesssim e^{-\frac{1}{3} \sqrt{1-\tilde{\omega}_{k}^2} \omega_{\star} t}.  
					\end{aligned}
				\end{equation}
				Thus, from \eqref{estimative1} and \eqref{estimative2}, we have that, for $k>j$,
				\begin{equation*}
					\begin{aligned}
						\int_{\R}\left|\tilde{R}_{k}^{(1)}(t)\right| \phi_{j}(x,t) \,d x \leq C e^{-\frac{1}{4} \sqrt{1-\tilde{\omega}_{k}^2} \omega_{\star} t} \leq Ce^{-4 \omega_{\star}^{\frac32} t},
					\end{aligned}    
				\end{equation*}
				whenever $t$ is sufficiently large.
				Then, from \eqref{segundoboostrap}, for $T_0$ large enough, $|\tilde{\omega}_k|^2\leq |\omega_k|^2,$
				hence
				for $k>j$,
				\begin{equation}\label{2interaccion}
					\begin{aligned}
						\int_{\R}\left|\tilde{R}_{k}^{(1)}(t)\right| \phi_{j}(x,t) \,d x \leq C e^{-\frac{1}{4} \sqrt{1-\tilde{\omega}_{k}^2} \omega_{\star} t} \leq Ce^{-4 \sqrt{\omega_{\star}} \omega_{\star} t}.
					\end{aligned}    
				\end{equation}
				Now, if $j>k$, using the support properties of $\phi_{j}$ and \eqref{1interaccion}, we have
\[
				\begin{aligned}
					\int_{\R}\left|\tilde{R}_{k}^{(1)}(t)\right| \phi_{j}(x,t) \,d x 
					\lesssim \int_{\R} e^{-\frac{2}{3} \sqrt{1-\tilde{\omega}_{k}^2}\left|x-\omega_k t\right|} \phi_{j}(x,t) \,d x = \int_{-\sqrt{t}+m_{j} t}^{\sqrt{t}+m_{j+1} t} e^{-\frac{2}{3} \sqrt{1-\tilde{\omega}_{k}^2}\left|x-\omega_kt\right|} \,d x.
				\end{aligned}\]
				Since $j>k$, $\omega_{j}>\omega_k$, so from the above estimate, it follows that
				\begin{equation}\label{cassoj>k}
					\begin{aligned}
						&\int_{\R}\left|\tilde{R}_{k}^{(1)}(t)\right| \phi_{j}(x,t) \,d x\\&\lesssim \int_{-\sqrt{t}+m_{k+1}t}^{\infty} e^{-\frac{2}{3} \sqrt{1-\tilde{\omega}_{k}^2} | x-\omega_kt|} \,d x \\
						& \lesssim \int_{-\sqrt{t}+m_{k+1} t}^{\sqrt{t}+m_{k+1} t} e^{-\frac{2}{3} \sqrt{1-\tilde{\omega}_{k}^2}\left|x-\omega_k t\right|} \,d x+  \int_{\sqrt{t}+m_{k+1}t}^{\infty} e^{-\frac{2}{3} \sqrt{1-\tilde{\omega}_{k}^2} |x-\omega_kt|} \,d x\\
						& =  \int_{-\sqrt{t}+m_{k+1} t}^{\sqrt{t}+m_{k+1} t} e^{-\frac{3}{2} \sqrt{1-\tilde{\omega}_{k}^2}\left|x-m_{k+1}t+m_{k+1}t-\omega_k t\right|} \,d x+  \int_{\sqrt{t}+\frac{v_{k+1}-\omega_k}{2}t}^{\infty} e^{-\frac{2}{3} \sqrt{1-\tilde{\omega}_{k}^2} |x|} \,d x\\
						&\lesssim e^{-\frac{3}{2} \sqrt{1-\tilde{\omega}_{k}^2} \frac{|v_{k+1}-\omega_k|}{2}t}\int_{-\sqrt{t}+m_{k+1} t}^{\sqrt{t}+m_{k+1} t}e^{\frac{2}{3} \sqrt{1-\tilde{\omega}_{k}^2} |x-m_{k+1}t|}dx+ e^{-\frac{2}{3} \sqrt{1-\tilde{\omega}_{k}^2} \sqrt{t}}e^{-\frac{1}{3} \sqrt{1-\tilde{\omega}_{k}^2} \omega_{\star}t}\\
						&\lesssim e^{-\frac{1}{3} \sqrt{1-\tilde{\omega}_{k}^2} \omega_{\star}t}(2\sqrt{t})e^{\frac{2}{3} \sqrt{1-\tilde{\omega}_{k}^2} \sqrt{t}}+ e^{-\frac{2}{3} \sqrt{1-\tilde{\omega}_{k}^2} \sqrt{t}}e^{-\frac{1}{3} \sqrt{1-\tilde{\omega}_{k}^2} \omega_{\star}t}\\
						&\lesssim e^{-\frac{1}{3} \sqrt{1-\tilde{\omega}_{k}^2} \omega_{\star}t}\left((2\sqrt{t})e^{\frac{2}{3} \sqrt{1-\tilde{\omega}_{k}^2} \sqrt{t}}+e^{-\frac{2}{3} \sqrt{1-\tilde{\omega}_{k}^2} \sqrt{t}}\right).
					\end{aligned}    
				\end{equation}
				Since for $t$ sufficiently large 
	\[(2\sqrt{t})e^{\frac{2}{3} \sqrt{1-\tilde{\omega}_{k}^2} \sqrt{t}}+e^{-\frac{2}{3} \sqrt{1-\tilde{\omega}_{k}^2} \sqrt{t}} \leq e^{\frac{1}{12} \sqrt{1-\tilde{\omega}_{k}^2} \omega_{\star}t},
		\]
				we have that estimate \eqref{cassoj>k} becomes
				\begin{equation}\label{3interaccion}
					\begin{aligned}
						\int_{\R}\left|\tilde{R}_{k}^{(1)}(t)\right| \phi_{j}(x,t) \,d x\leq e^{-4 \sqrt{\omega_{\star}} \omega_{\star}t}.
					\end{aligned}    
				\end{equation}
				Therefore, for $j\neq k$ with $j\in {2, \ldots, N-1}$, from \eqref{2interaccion} and \eqref{3interaccion}
		\[\int_{\R}\left|\tilde{R}_{k}^{(1)}(t)\right| \phi_{j}(x,t) \,d x\lesssim e^{-4 \sqrt{\omega_{\star}} \omega_{\star}t}.\]
				The estimate for the term $\left|\partial_{x} \tilde{R}_{k}^{(1)}\right| \phi_{j}$ is obtained in a similar manner. Furthermore, the cases $j=1$ and $j=N$ can be obtained using a similar argument.
				
				We conclude that for sufficiently large $t$,
	\[\int_{\R}\left(\left|\tilde{R}_{k}^{(1)}(t)\right|+\left|\partial_{x} \tilde{R}_{k}^{(1)}(t)\right|\right) \phi_{j}(x,t) \,d x\lesssim e^{-4 \omega_{\star}^{\frac32}t}, \, \, \text{$j\neq k$}.\]
				The second estimate follows immediately, since for $k \in {1,\ldots,N}$, we have $1-\phi_{k}=\sum_{j\neq k}\phi_{j}$.
				To obtain the last estimate of the present lemma, we can proceed in the same way.
			\end{proof}
			\subsection{Bootstrap Argument for Soliton Summation}
			
		Using our previous findings, we will establish a bootstrap result related to soliton summation. To do this, let us consider a sequence  $T^n\to\infty$ and take $\mathbf{u}^{n}\in H$ as the solution of \eqref{sistema1} such that $\mathbf{u}^{n}(T^n)=\mathbf{R}(T^n)$.
			
			\begin{proposition}\label{Bootstrap3}
				There exist $T_{0}\in \R$ depending only on $\omega_{j}$ and $n_{0} \in\mathbb{N}$ such that, if  the approximate solution $\mathbf{u}^n$ of \eqref{sistema1}, defined on $[T_{0},T^{n}]$, there holds the estimate
				\begin{equation*}
					\|\bu^{n}(t)- \ru(t)\|_{H}\leq e^{-\omega_{\star}^{\frac32}t},\quad\forall n \geq n_{0},\,  t \in[t_{0},T^{n}],
				\end{equation*}
      with $t_{0}\in [T_{0},T^{n}]$,
				then for all $t\in[t_{0},T^{n}]$, 
				\begin{equation*}
					\|\bu^{n}(t)- \ru(t)\|_{H}\leq \frac{1}{2}e^{-\omega_{\star}^{\frac32}t}.
				\end{equation*}
			\end{proposition}

To demonstrate the bootstrap argument, our main objective is to obtain uniform estimates in terms of mass and energy. Suppose $T_{0}>0$ is sufficiently large and choose $n \in \mathbb{N}$ such that $T^{n}>T_{0}$. For convenience, in this section, we will again remove the dependence on $n$ in the approximate sequence and denote it simply by $\mathbf{u}$. For all $t\in [T_{0}, T^{n}]$, we define
			\begin{equation}\label{defv}
				\vu(t)=\bu(t)-\ru(t),
			\end{equation}
			and assume that for $t \in [t_{0},T^{n}]$ (with $t_{0} \in [T_{0},T^{n}]$),
			\begin{equation}\label{hipotesisbootstrap}
				\|\vu(t)\|_{H}\leq e^{-\omega_{\star}^{\frac32}t}.   
			\end{equation}
We need the following preliminary result.
			\begin{lemma}\label{taylorS}
				There exists $T_{0}$ such that if $t_{0} >T_{0}$, then for every $t \in [t_{0},T^{n}]$,
				\begin{equation*}
					\begin{aligned}
						\mathcal{S}( \bu(t))=\sum_{j=1}^N\left\{\mathcal{E}(\ru)+\frac{\omega_j}{2}\int_{\R} R_{j}^{(1)} R_{j}^{(2)} \,dx+ \mathcal{O}\left(|\tilde{\omega}_{j}(t)-\omega_{j}|^{2}\right)\right\}+\mathcal{H}(\var)+\mathcal{O}\left(e^{-3\omega_{\star}^{\frac32}t}\right),      
					\end{aligned}
				\end{equation*}
				where  \[\begin{aligned}
					\mathcal{H}(\var)&=\frac{1}{2}\int_{\R}|\partial_{x}\varepsilon_1|^2 \,dx+\int_{\R}\left|\varepsilon_1\right|^{2} \,dx+\int_{\R}\left|\varepsilon_2\right|^{2} \,dx\\
					&\quad+\sum_{j=1}^N\left(-\frac{1}{2}\int_{\R}\left|\tilde{R}_{j}^{(1)}\right|^{2p}|\varepsilon_{1}|^2\,dx- p\int_{\R}\left|\tilde{R}_{j}^{(1)}\right|^{2p-2}\left(\tilde{R}_{j}^{(1)}\varepsilon_{1}\right)^2 \,dx+\frac{\tilde{\omega}_j}{2}\int_{\R} \varepsilon_1 \varepsilon_2 \phi_{j}\,dx\right).   
				\end{aligned}
			\]
				Moreover, there exists a constant $C>0$ such that, for any  $t_0 \in\left[T_0, T^n\right]$ and for all $t\in[t_0,T^n]$,
				\begin{equation}\label{coe0}
					\mathcal{H}(\var) \geq  C\left\|\var\right\|_{H}^2.
				\end{equation}
			\end{lemma} 
			\begin{proof}
				The idea is to expand the terms in the expression of  $\mathcal{S}$ using $u_{1}=\varepsilon_{1}+\tilde{R}_{1}$ and $u_{2}=\varepsilon_{2}+\tilde{R}_{2}$. Indeed, note  from \eqref{Sdefin} that
	$ 
				\mathcal{S} (\bu):=\sum_{j=1}^{N} \mathcal{S}_{\operatorname{loc}}^{j}(\bu),
$ 
				with \begin{equation*}
					\mathcal{S}_{\operatorname{loc}}^{j}(\bu):=\mathcal{E}_j(\bu)+ \omega_j \mathcal{M}_j (\bu).
				\end{equation*}
				Hence,  we will estimate each term of the above expression separately. First, we observe that
				\begin{equation}\label{enesoma}
					\begin{aligned}
						 \mathcal{E}_j\left(\var+\tilde{\ru}\right)   
						& =\frac{1}{2} \int_{\R}\left|\partial_{x}(\tilde{R}_{1}+\varepsilon_1)\right|^{2} \phi_{j}\,dx +\frac{1}{2}  \int_{\R}\left|\tilde{R}_{2}+\varepsilon_2\right|^{2} \phi_{j}\,dx +\frac{1}{2}\int_{\R}\left|\tilde{R}_{1}+\varepsilon_1\right|^{2} \phi_{j}\,dx \\
						&\quad-\frac{1}{2p+2}\int_{\R}\left|  \tilde{R}_{1}+\varepsilon_1\right|^{2p+2}\phi_{j}\,dx.
					\end{aligned}
				\end{equation}
				Let us expand the terms of the above integrals as
				\begin{equation}\label{id1}
					\begin{aligned}
						\int_{\R}\left|\partial_{x}(\tilde{R}_{1}+\varepsilon_1)\right|^{2} \phi_{j}\,dx &=\int_{\R}|\partial_{x}\tilde{R}_1|^2 \phi_{j}\,dx+\int_{\R}|\partial_{x}\varepsilon_1|^2 \phi_{j}\,dx-2\int_{\R}\partial_{xx}\tilde{R}_1 \varepsilon_1 \phi_{j}\,dx\\
						&=\int_{\R}\left|\sum_{k=1}^{N}\partial_{x}\tilde{R}_{k}^{(1)}\right|^2 \phi_{j}\,dx+\int_{\R}|\partial_{x}\varepsilon_1|^2 \phi_{j}\,dx-2\int_{\R}\sum_{k=1}^{N}\partial_{xx}\tilde{R}_{k}^{(1)}\,\varepsilon_1 \phi_{j}\,dx.
					\end{aligned}
				\end{equation}
				Now, since 
		\[
				\left|\sum_{k=1}^{N} \partial_{x} \tilde{R}_{k}^{(1)}\right|^{2}=\left|\partial_{x} \tilde{R}_{j}^{(1)}\right|^{2}+\sum_{\substack{k=1 \\ k\neq j}}^{N}\left|\partial_{x} \tilde{R}_{k}^{(1)}\right|^{2}+\sum_{\substack{k, m=1 \\ k\neq m}}^{N} \partial_{x}\tilde{R}_{k}^{(1)} \partial_{x} \tilde{R}_{m}^{(1)}.
		\]
				Using the fact that the solitons are bounded (see Proposition \ref{decaimentoquadratico}) and Lemma \ref{solitons}, we obtain
				\begin{equation*}\label{somadeorden4}
					\begin{aligned}
						\sum_{\substack{k=1 \\ s \neq j}}^{N}\int_{\R} \left|\partial_{x}\tilde{R}_{k}^{(1)}\right|^{2}\phi_{j}\,dx+\sum_{\substack{k, m=1 \\ k \neq m}}^{N} \int_{\R} \partial_{x} \tilde{R}_{k}^{(1)} \partial_{x} \tilde{R}_{m}^{(1)}\phi_{j}\,dx 
						 &\lesssim\sum_{\substack{k=1 \\ k \neq j}}^{N}e^{-4\omega_{\star}^{\frac32}t}+ \sum_{\substack{k, m=1 \\ k \neq m}}^{N}e^{-4\omega_{\star}^{\frac32}t} \lesssim e^{-4\omega_{\star}^{\frac32}t}.
					\end{aligned}
				\end{equation*}
				Taking $T_{0}$ sufficiently large, we deduce 
				\begin{equation}\label{resulsoma}
 \begin{aligned}
 	\sum_{\substack{k=1 \\ k \neq j}}^{N}\int_{\R} \left|\partial_{x} \tilde{R}_{k}^{(1)}\right|^{2}\phi_{j}\,dx+\sum_{\substack{k, m=1 \\ k \neq m}}^{N} \int_{\R} \partial_{x} \tilde{R}_{k}^{(1)} \partial_{x} \tilde{R}_{m}^{(1)}\phi_{j}\,dx\leq C e^{-3\omega_{\star}^{\frac32}t},
					\end{aligned}
				\end{equation}
				where $C=C(R_{1}^{(1)},\ldots,R_{N}^{(1)},R_{1}^{(2)},\ldots,R_{N}^{(2)},N)$.
	 Moreover, note that   
  \[\begin{aligned}					\int_{\R}\sum_{k=1}^{N}\partial_{xx}\tilde{R}_{k}^{(1)}\,\varepsilon_1 \phi_{j}\,dx&=\int_{\R}\partial_{xx}\tilde{R}_{j}^{(1)}\,\varepsilon_1 \phi_{j}\,dx+\int_{\R}\sum_{\substack{k=1 \\ k \neq j}}^{N}\partial_{xx}\tilde{R}_{k}^{(1)}\,\varepsilon_1 \phi_{j}\,dx,
				\end{aligned}\] 
				then,  we have from   \eqref{segundoboostrap} and    Lemma \ref{solitons} that
\[\begin{aligned}
					\int_{\R}\sum_{\substack{k=1 \\ k \neq j}}^{N}\partial_{xx}\tilde{R}_{k}^{(1)}\,\varepsilon_1 \phi_{j}\,dx&=-\int_{\R}\sum_{\substack{k=1 \\ k \neq j}}^{N}\partial_{x}\tilde{R}_{k}^{(1)}\,\partial_{x}\varepsilon_1 \phi_{j}\,dx \leq C\left(\int_{\R}\sum_{\substack{k=1 \\ k \neq j}}^{N}\partial_{x}\tilde{R}_{k}^{(1)} \phi_{j}\,dx\right)^{1/2}    
					\leq Ce^{-3\omega_{\star}^{\frac32}t}.
				\end{aligned} \]
				Hence, \begin{equation}\label{id2}\int_{\R}\sum_{k=1}^{N}\partial_{xx}\tilde{R}_{k}^{(1)}\,\varepsilon_1 \phi_{j}\,dx=\int_{\R}\partial_{xx}\tilde{R}_{j}^{(1)}\,\varepsilon_1 \phi_{j}\,dx+O(e^{-3\omega_{\star}^{\frac32}t}).\end{equation}
				On the other hand, note that 
    \[\begin{aligned}
					\int_{\R}\left|\tilde{R}_{2}+\varepsilon_2\right|^{2} \phi_{j}\,dx &=\int_{\R}\left|\tilde{R}_{2}\right|^{2} \phi_{j}\,dx+\int_{\R} \left|\varepsilon_{2}\right|^{2} \phi_{j}\,dx+\int_{\R}\tilde{R}_{2}\varepsilon_{2}\phi_{j}\,dx,
				\end{aligned}\]
				So proceeding similarly to the previous estimates, it follows that 
\[\begin{aligned}
					\int_{\R}\left|\tilde{R}_{2}\right|^{2} \phi_{j}\,dx&=\int_{\R}\left|\sum_{k=1}^N\tilde{R}_{k}^2\right|^{2} \phi_{j}\,dx=\int_{\R}\left|\tilde{R}_{j}^{(2)}\right|^{2} \phi_{j}\,dx+O(e^{-3\omega_{\star}^{\frac32}t}) \\
					&=\int_{\R}\left|\tilde{R}_{j}^{(2)}\right|^{2} \,dx-\sum_{\substack{k=1 \\ k \neq j}}^{N}\int_{\R}\left|\tilde{R}_{k}^2\right|^{2} \phi_{j}\,dx+O(e^{-3\omega_{\star}^{\frac32}t})\\
					&=\int_{\R}\left|\tilde{R}_{j}^{(2)}\right|^{2} \,dx+O(e^{-3\omega_{\star}^{\frac32}t});
				\end{aligned}\]
				and since $(\varepsilon_{2},\tilde{R}_{k}^{2})_{2}=0$, we obtain that \begin{equation}\label{orto1}
					\begin{aligned}
						\int_{\R}\tilde{R}_{2}\varepsilon_{2}\phi_{j}\,dx&=\int_{\R}\tilde{R}_{j}^{(2)}\varepsilon_{2}\phi_{j}\,dx  +\sum_{\substack{k=1 \\ k \neq j}}^{N}\int_{\R}\tilde{R}_{k}^2\varepsilon_{2}\phi_{j}\,dx\\
						&=\int_{\R}\tilde{R}_{j}^{(2)}\varepsilon_{2}\,dx-\sum_{\substack{k=1 \\ k \neq j}}^{N}\int_{\R}\tilde{R}_{j}^{(2)}\varepsilon_{2}\phi_{k}\,dx+O(e^{-3\omega_{\star}^{\frac32}t}) =O(e^{-3\omega_{\star}^{\frac32}t}).
					\end{aligned}    
				\end{equation}
				Therefore, \begin{equation}\label{id3}
					\int_{\R}\left|\tilde{R}_{2}+\varepsilon_2\right|^{2} \phi_{j}\,dx=\int_{\R}\left|\tilde{R}_{j}^{(2)}\right|^{2} \,dx+\int_{\R}\left|\varepsilon_2\right|^{2} \phi_{j}\,dx+O(e^{-3\omega_{\star}^{\frac32}t}).
				\end{equation}
				So, following the same argument, we arrive at
				\begin{equation}\label{id4}
					\int_{\R}\left|\tilde{R}_{1}+\varepsilon_1\right|^{2} \phi_{j}\,dx=\int_{\R}\left|\tilde{R}_{j}^{(1)}\right|^{2} \,dx+\int_{\R}\left|\varepsilon_1\right|^{2} \phi_{j}\,dx+O(e^{-3\omega_{\star}^{\frac32}t}).  
				\end{equation}
				We now proceed to expand the last term of the expression above using the Taylor expansion. Considering 
				$\mathcal{G}(s)=s^{p+1}$, we have
				\begin{equation}\label{expantionG}
\begin{aligned}						\mathcal{G}\left(\left|\varepsilon_{1}+\tilde{R}_{1}\right|^{2}\right)&=\mathcal{G}\left(\left|\tilde{R}_{1}\right|^{2}\right)+\mathcal{G}^{\prime}\left(\left|\tilde{R}_{1}\right|^{2}\right)\left(\left|\varepsilon_{1}\right|^{2}+2\mathcal{R}e\left(\tilde{R}_{1}\varepsilon_{1}\right)\right)\\
						&\quad +\frac{1}{2}\mathcal{G}^{\prime \prime}\left(\left|\tilde{R}_{1}\right|^{2}\right)\left(\left|\varepsilon_{1}\right|^{2}+2\mathcal{R}e\left(\tilde{R}_{1}\varepsilon_{1}\right)\right)^{2}+\mathcal{O}\left( \left\| 2\mathcal{R}e \left(\tilde{R}_{1}\varepsilon_{1}\right)+|\varepsilon_{1}|^{2}\right\|_{H^{1}(\R)}^{3} \right)\\
						&=\mathcal{G}\left(\left|\tilde{R}_{1}\right|^{2}\right)+\mathcal{G}^{\prime}\left(\left|\tilde{R}_{1}\right|^{2}\right)\left(\left|\varepsilon_{1}\right|^{2}+2\mathcal{R}e\left(\tilde{R}_{1}\varepsilon_{1}\right)\right)\\
						&\quad+\frac{1}{2}\mathcal{G}^{\prime \prime}\left(\left|\tilde{R}_{1}\right|^{2}\right)\left|\varepsilon_{1}\right|^{4}+\frac{1}{2}\mathcal{G}^{\prime \prime}\left(\left|\tilde{R}_{1}\right|^{2}\right)\left(2\mathcal{R}e\left(\tilde{R}_{1}\varepsilon_{1}\right)\right)^2\\
						&\quad+\frac{1}{2}\mathcal{G}^{\prime \prime}\left(\left|\tilde{R}_{1}\right|^{2}\right)\mathcal{R}e\left(\tilde{R}_{1}\varepsilon_{1}\right)\left|\varepsilon_{1}\right|^{2}+\mathcal{O}\left( \left\| 2\mathcal{R}e \left(\tilde{R}_{1}\varepsilon_{1}\right)+|\varepsilon_{1}|^{2}\right\|_{H^{1}(\R)}^{3} \right).
					\end{aligned}
				\end{equation}
				Using again \eqref{segundoboostrap} and Lemma \ref{solitons}, we deduce
				
				\begin{equation*}
					\begin{aligned}
						\int_{\R}\mathcal{G}^{\prime \prime }\left(\left|\tilde{R}_{1}\right|^{2}\right)\left|\varepsilon_{1}\right|^{4}\phi_{j}\,dx&=p(p+1)\int_{\R}\left|\tilde{R}_{1}\right|^{2p-2}\left|\varepsilon_{1}\right|^{4}\phi_{j}\,dx \leq C\left\|\varepsilon_{1}\right\|_{H^{1}(\R)}^{4}  \leq Ce^{-4\omega_{\star}^{\frac32}t}.
					\end{aligned}
				\end{equation*}
				Furthermore, since $\tilde{R}_{1}$ is bounded, it follows that 
				\begin{equation*}
					\begin{aligned}
						\int_{\R}\mathcal{G}^{\prime \prime}\left(\left|\tilde{R}_{1}\right|^{2}\right)\left(2\mathcal{R}e\left(\tilde{R}_{1}\varepsilon_{1}\right)\left|\varepsilon_{1}\right|^{2}\right)\phi_{j}\,dx&=2p(p+1)\int_{\R}\left|\tilde{R}_{1}\right|^{2p-2}\mathcal{R}e\left(\tilde{R}_{1}\varepsilon_{1}\right)\left|\varepsilon_{1}\right|^{2}\phi_{j}\,dx\\
						& \leq C\int_{\R}\left|\varepsilon_{1}\right|^{3}\,dx \leq C\left\|\varepsilon_{1}\right\|_{H^{1}(\R)}^{3}  \leq Ce^{-3\omega_{\star}^{\frac32}t},
					\end{aligned}
				\end{equation*}
				and
				\begin{equation*}
					\begin{aligned}
						\left\| 2\mathcal{R}e \left(\tilde{R}_{1}\varepsilon_{1}\right)+|\varepsilon_{1}|^{2}\right\|_{H^{1}(\R)}^{3}&\leq C\left\|\varepsilon_{1}\right\|_{H^{1}(\R)}^{3}+C\left\|\varepsilon_{1}\right\|_{H^{1}(\R)}^{6}\\
						&\leq Ce^{-3\omega_{\star}^{\frac32}t}+ Ce^{-6\omega_{\star}^{\frac32}t}\\
						&\leq Ce^{-3\omega_{\star}^{\frac32}t}.
					\end{aligned}
				\end{equation*}
				Then, \begin{equation}\label{id5}
					\begin{aligned}-\frac{1}{2p+2}\int_{\R}\left|\tilde{R}_1+\varepsilon_1\right|^{2p+2} \phi_{j}\,dx &=-\frac{1}{2}\int_{\R}\left|\tilde{R}_{j}^{(1)}\right|^{2p}|\varepsilon_{1}|^2 \phi_{j}\,dx-\frac{1}{2}\int_{\R}2\left|\tilde{R}_{j}^{(1)}\right|^{2p}\tilde{R}_{j}^{(1)} \varepsilon_1\phi_{j}\,dx\\
						&\quad -\frac{1}{2 p+2} \int_{\R}\left|\tilde{R}_{j}^{(1)}\right|^{2 p+2}\phi_{j}\,dx-p\int_{\R}\left|\tilde{R}_{j}^{(1)}\right|^{2p-2}\left(\tilde{R}_{j}^{(1)}\varepsilon_{1}\right)^2 \phi_{j}\,dx\\
						&\quad + O\left(e^{-3\omega_{\star}^{\frac32}t}\right).\end{aligned}
				\end{equation}
				Therefore, combining \eqref{enesoma}-\eqref{id5}, we obtain that
\begin{equation*}	\begin{aligned}						\mathcal{E}_j\left(\tilde{\ru}+\var\right)
						& =\frac{1}{2}\int_{\R}\left|\partial_{x} \tilde{R}_{j}^{(1)}\right|^{2} \phi_{j}\,dx+\frac{1}{2}\int_{\R}\left| R_{j}^{2}\right|^{2} \phi_{j} \,dx+\frac{1}{2}\int_{\R}\left| \tilde{R}_{j}^{(1)}\right|^{2} \phi_{j} \,dx -\frac{1}{2p+2}   \int_{\R} \left| \tilde{R}_{j}^{(1)}\right|^{2p+2} \phi_{j} \,dx \\
						&\quad +\frac{1}{2}\int_{\R}|\partial_{x}\varepsilon_1|^2 \phi_{j}\,dx-\frac{1}{2}\int_{\R}2\partial_{xx}\tilde{R}_{j}^{(1)}\,\varepsilon_1 \phi_{j}\,dx-\frac{1}{2}\int_{\R}2\left|\tilde{R}_{j}^{(1)}\right|^{2p}\tilde{R}_{j}^{(1)} \varepsilon_1 \phi_{j}\,dx\\&\quad-\frac{1}{2}\int_{\R}\left|\tilde{R}_{j}^{(1)}\right|^{2p}|\varepsilon_{1}|^2 \phi_{j}\,dx -p\int_{\R}\left|\tilde{R}_{j}^{(1)}\right|^{2p-2}\left(\tilde{R}_{j}^{(1)}\varepsilon_{1}\right)^2 \phi_{j}\,dx +O\left(e^{-3 \omega_{\star}^{\frac32} t}\right).
					\end{aligned}  
\end{equation*}
				Notice that any term involving $\tilde{R}_{j}^{(1)}$, $\tilde{R}_{j}^{(2)}$ and $\phi_j$ can be refined leveraging the fact that $\sum_{k=1}^N \phi_k=1$. For instance, using \eqref{solitons}
	\[\begin{aligned}
					\frac{1}{2}\int_{\R}\left|\partial_{x} \tilde{R}_{j}^{(1)}\right|^{2} \phi_{j}\,dx&=\frac{1}{2}\int_{\R}\left|\partial_{x} \tilde{R}_{j}^{(1)}\right|^{2} \,dx-\frac{1}{2}\sum_{\substack{k=1 \\ k \neq j}}^{N}\int_{\R}\left|\partial_{x} \tilde{R}_{j}^{(1)}\right|^{2} \phi_{k}\,dx\\
					&=\frac{1}{2}\int_{\R}\left|\partial_{x} \tilde{R}_{j}^{(1)}\right|^{2} \,dx+O\left(e^{-3 \omega_{\star}^{\frac32} t}\right).
				\end{aligned}\]
				So  that, 
    \begin{equation*}\begin{aligned}						\mathcal{E}_j\left(\tilde{\ru}+\var\right)
						& =\frac{1}{2}\int_{\R}\left|\partial_{x} \tilde{R}_{j}^{(1)}\right|^{2} \,dx+\frac{1}{2}\int_{\R}\left| R_{j}^{2}\right|^{2}  \,dx+\frac{1}{2}\int_{\R}\left| \tilde{R}_{j}^{(1)}\right|^{2} \,dx -\frac{1}{2p+2}   \int_{\R} \left| \tilde{R}_{j}^{(1)}\right|^{2p+2} \,dx \\
						&\quad +\frac{1}{2}\int_{\R}|\partial_{x}\varepsilon_1|^2 \phi_{j}\,dx-\frac{1}{2}\int_{\R}2\partial_{xx}\tilde{R}_{j}^{(1)}\,\varepsilon_1 \,dx-\frac{1}{2}\int_{\R}2\left|\tilde{R}_{j}^{(1)}\right|^{2p}\tilde{R}_{j}^{(1)} \varepsilon_1 \,dx\\&\quad-\frac{1}{2}\int_{\R}\left|\tilde{R}_{j}^{(1)}\right|^{2p}|\varepsilon_{1}|^2\,dx -p\int_{\R}\left|\tilde{R}_{j}^{(1)}\right|^{2p-2}\left(\tilde{R}_{j}^{(1)}\varepsilon_{1}\right)^2 \,dx+\int_{\R}\left|\varepsilon_1\right|^{2} \phi_{j}\,dx\\
						&\quad+\int_{\R}\left|\varepsilon_2\right|^{2} \phi_{j}\,dx+O\left(e^{-3 \omega_{\star}^{\frac32} t}\right).
					\end{aligned}  
				\end{equation*}
 But, since $(\varepsilon_1,\tilde{R}_{j}^{(1)})_{2}=0$,\[\begin{aligned}
					-\frac{1}{2}&\int_{\R}2\partial_{xx}\tilde{R}_{j}^{(1)}\,\varepsilon_1 \,dx-\frac{1}{2}\int_{\R}2\left|\tilde{R}_{j}^{(1)}\right|^{2p}\tilde{R}_{j}^{(1)} \varepsilon_1 \,dx\\&=-\frac{1}{2}\int_{\R}2\partial_{xx}\tilde{R}_{j}^{(1)}\,\varepsilon_1 \,dx-\frac{1}{2}\int_{\R}2\left|\tilde{R}_{j}^{(1)}\right|^{2p}\tilde{R}_{j}^{(1)} \varepsilon_1 \,dx+(1-\tilde{\omega}_j^2)\int_{\R}\tilde{R}_{j}^{(1)} \varepsilon_1\,dx\\
					&=\int_{\R}\left(-\partial_{xx}\tilde{R}_{j}^{(1)}+(1-\tilde{\omega}_j^2)\tilde{R}_{j}^{(1)}-\left|\tilde{R}_{j}^{(1)}\right|^{2p}\tilde{R}_{j}^{(1)}\right)\varepsilon_1\,dx =0,
				\end{aligned}\]
				then \begin{equation}\label{ene}
					\begin{aligned}
						\mathcal{E}_j\left(\tilde{\ru}+\var\right)
						& =\frac{1}{2}\int_{\R}\left|\partial_{x} \tilde{R}_{j}^{(1)}\right|^{2} \,dx+\frac{1}{2}\int_{\R}\left| R_{j}^{2}\right|^{2}  \,dx+\frac{1}{2}\int_{\R}\left| \tilde{R}_{j}^{(1)}\right|^{2} \,dx -\frac{1}{2p+2}   \int_{\R} \left| \tilde{R}_{j}^{(1)}\right|^{2p+2} \,dx \\
						&\quad +\frac{1}{2}\int_{\R}|\partial_{x}\varepsilon_1|^2 \phi_{j}\,dx-\frac{1}{2}\int_{\R}\left|\tilde{R}_{j}^{(1)}\right|^{2p}|\varepsilon_{1}|^2\,dx -p\int_{\R}\left|\tilde{R}_{j}^{(1)}\right|^{2p-2}\left(\tilde{R}_{j}^{(1)}\varepsilon_{1}\right)^2 \,dx\\
						&\quad+\int_{\R}\left|\varepsilon_1\right|^{2} \phi_{j}\,dx+\int_{\R}\left|\varepsilon_2\right|^{2} \phi_{j}\,dx+O\left(e^{-3 \omega_{\star}^{\frac32} t}\right).
					\end{aligned}  
				\end{equation}
				Now, we analyze
\[
				\begin{aligned}
					\omega_j\mathcal{M}_j(\tilde{\ru}+\var) 
					&=\frac{\omega_j}{2}\int_{\R} \left(\tilde{R}_1+\varepsilon_1\right)\left(\tilde{R}_2+\varepsilon_2\right)\phi_{j}\,dx\\
					&=\frac{\omega_j}{2}\int_{\R} \tilde{R}_1 \varepsilon_2 \phi_{j}\,dx+\frac{\omega_j}{2}\int_{\R} \tilde{R}_1 \tilde{R}_2 \phi_{j}\,dx +\frac{\omega_j}{2}\int_{\R} \varepsilon
					_1 \tilde{R}_2 \phi_{j}\,dx+\frac{\omega_j}{2}\int_{\R} \varepsilon_1 \varepsilon_2 \phi_{j}\,dx.
				\end{aligned}
\]
				Applying a similar process to \eqref{orto1}, we have \[\omega_j\int_{\R} \tilde{R}_1 \varepsilon_2 \phi_{j}\,dx+\omega_j\int_{\R} \varepsilon
				_1 \tilde{R}_2 \phi_{j}\,dx=O\left(e^{-3 \omega_{\star}^{\frac32} t}\right).\]
				Also, note that
	\[\begin{aligned}
					\omega_j\int_{\R} \tilde{R}_1 \tilde{R}_2 \phi_{j}\,dx&=\omega_j\int_{\R} \tilde{R}_{j}^{(1)} \tilde{R}_{j}^{(2)} \phi_{j}\,dx+O\left(e^{-3 \omega_{\star}^{\frac32} t}\right)\\
					&=\omega_j\int_{\R} \tilde{R}_{j}^{(1)} \tilde{R}_{j}^{(2)} \,dx-\omega_j\sum_{\substack{k=1 \\ k \neq j}}^{N}\int_{\R} \tilde{R}_{j}^{(1)} \tilde{R}_{j}^{(2)} \phi_{k} \,dx+O\left(e^{-3 \omega_{\star}^{\frac32} t}\right)\\
					&=\omega_j\int_{\R} \tilde{R}_{j}^{(1)} \tilde{R}_{j}^{(2)} \,dx+O\left(e^{-3 \omega_{\star}^{\frac32} t}\right)
				\end{aligned}\]
				and 
\[\begin{aligned}
					\omega_j\int_{\R} \varepsilon_1 \varepsilon_2 \phi_{j}\,dx&=\tilde{\omega}_j\int_{\R} \varepsilon_1 \varepsilon_2 \phi_{j}\,dx+(\tilde{\omega}_j-\omega_j)\int_{\R} \varepsilon_1 \varepsilon_2 \phi_{j}\,dx\\
					&\leq \tilde{\omega}_j\int_{\R} \varepsilon_1 \varepsilon_2 \phi_{j}\,dx+|\tilde{\omega}_j-\omega_j|^2 +\left(\int_{\R} |\varepsilon_1|^2\,dx\right) +\left(\int_{\R} |\varepsilon_2|^2\,dx\right) \\
					&\leq \tilde{\omega}_j\int_{\R} \varepsilon_1 \varepsilon_2 \phi_{j}\,dx+O\left(|\tilde{\omega}_j-\omega_j|^2\right)+Ce^{-3 \omega_{\star}^{\frac32} t}.
				\end{aligned}\]
				Therefore, \begin{equation}\label{mas}
					\omega_j\mathcal{M}_j(\tilde{\ru}+\var)=\frac{\omega_j}{2}\int_{\R} \tilde{R}_{j}^{(1)} \tilde{R}_{j}^{(2)} \,dx+\frac{\tilde{\omega}_j}{2}\int_{\R} \varepsilon_1 \varepsilon_2 \phi_{j}\,dx+O\left(|\tilde{\omega}_j-\omega_j|^2\right)+O\left(e^{-3 \omega_{\star}^{\frac32} t}\right).
				\end{equation}
				So,  we obtain from \eqref{ene} and \eqref{mas} that    
\[ \begin{aligned}
					\mathcal{S}_{\operatorname{loc}}^{j}(\bu)&=\mathcal{E}(\tilde{\ru}_j)+\frac{\omega_j}{2}\int_{\R} \tilde{R}_{j}^{(1)} \tilde{R}_{j}^{(2)} \,dx+\frac{1}{2}\int_{\R}|\partial_{x}\varepsilon_1|^2 \phi_{j}\,dx \\
					&\quad-\frac{1}{2}\int_{\R}\left|\tilde{R}_{j}^{(1)}\right|^{2p}|\varepsilon_{1}|^2\,dx +\int_{\R}\left|\varepsilon_1\right|^{2} \phi_{j}\,dx-p\int_{\R}\left|\tilde{R}_{j}^{(1)}\right|^{2p-2}\left(\tilde{R}_{j}^{(1)}\varepsilon_{1}\right)^2 \,dx+\frac{\tilde{\omega}_j}{2}\int_{\R} \varepsilon_1 \varepsilon_2 \phi_{j}\,dx\\&\quad+\int_{\R}\left|\varepsilon_2\right|^{2} \phi_{j}\,dx+\mathcal{O}\left(|\tilde{\omega}_{j}(t)-\omega_{j}|\right)+O\left(e^{-3 \omega_{\star}^{\frac32} t}\right).  
				\end{aligned}\]
				Then, we derive
    \begin{equation}\label{ExpS}
					\begin{aligned}
						\mathcal{S}( \bu(t))=\sum_{j=1}^N\left\{\mathcal{E}(\tilde{\ru}_j)+\frac{\omega_j}{2}\int_{\R} \tilde{R}_{j}^{(1)} \tilde{R}_{j}^{(2)} \,dx+ \mathcal{O}\left(|\tilde{\omega}_{j}(t)-\omega_{j}|^{2}\right)\right\}+\mathcal{H}(\var)+\mathcal{O}\left(e^{-3\omega_{\star}^{\frac32}t}\right),      
					\end{aligned}
				\end{equation}
				where  \[\begin{aligned}
					\mathcal{H}(\var)&=\frac{1}{2}\int_{\R}|\partial_{x}\varepsilon_1|^2 \,dx+\int_{\R}\left|\varepsilon_1\right|^{2} \,dx+\int_{\R}\left|\varepsilon_2\right|^{2} \,dx\\
					&\quad+\sum_{j=1}^N\left(-\frac{1}{2}\int_{\R}\left|\tilde{R}_{j}^{(1)}\right|^{2p}|\varepsilon_{1}|^2\,dx- p\int_{\R}\left|\tilde{R}_{j}^{(1)}\right|^{2p-2}\left(\tilde{R}_{j}^{(1)}\varepsilon_{1}\right)^2 \,dx+\frac{\tilde{\omega}_j}{2}\int_{\R} \varepsilon_1 \varepsilon_2 \phi_{j}\,dx\right).   
				\end{aligned}
			\]
				To conclude the result, we define the operator
	\[\mathcal{J}(z_1,z_2)=\mathcal{E}(z_1,z_2)+\omega_{j}\int_{\R}z_1 z_2\,dx.\]
				We know that, by performing a similar calculation to Remark \ref{observa1}, $\ru_j$ is a critical point of $\mathcal{J}$. Thus,by applying the Taylor expansion formula, 
				\begin{equation*}
					\begin{aligned}
						\mathcal{J}(\tilde{\ru}_j)=\mathcal{J}(\ru_j)+\frac{1}{2}\mathcal{J}^{\prime \prime}(\ru_j)(\tilde{\omega}_{1}-\omega_{1})^{2}+|\tilde{\omega}_{1}-\omega_{1}|^{2}o\left(|\tilde{\omega}_{1}-\omega_{1}|\right),
					\end{aligned}
				\end{equation*}
				From which we deduce that the expression \eqref{ExpS} can be written as
				\begin{equation*}
					\begin{aligned}
						\mathcal{S}( \bu(t))=\sum_{j=1}^N\left\{\mathcal{E}(\ru_j)+\frac{\omega_j}{2}\int_{\R} R_{j}^{(1)} R_{j}^{(2)} \,dx+ \mathcal{O}\left(|\tilde{\omega}_{j}(t)-\omega_{j}|^{2}\right)\right\}+\mathcal{H}(\var) +\mathcal{O}\left(e^{-3\omega_{\star}^{\frac32}t}\right).
					\end{aligned}
				\end{equation*}

    For the estimate \eqref{coe0}, it is sufficient to use the Proposition \ref{coercivity} in a similar manner as in Section \ref{supercritical} in Proposition  \ref{coercivity2}.
				Thus, we have completed the desired proof.
			\end{proof} 
			
			Before giving the proof of Proposition \ref{Bootstrap3}, we need to establish the following result, which provides an almost conservation law.
			\begin{lemma}\label{conservationquase}
				There exists $C > 0$  such that if $T_{0}$ is sufficiently large, then for all $t \in \left[t_{0}, T^{n}\right]$,
		\[
				\left|\mathcal{M}_j(\mathbf{u}(t))-\mathcal{M}_j(\mathbf{u}(T^{n}))\right| \leq \frac{C}{\sqrt{t}}e^{-2\omega_{\star}^{\frac32} t}, \qquad  j=1,2,\ldots,N .
			\]
			\end{lemma}
			\begin{proof}
				We fix $j\in \{2, \ldots, N\}$. Note that
				\begin{equation}\label{varia1}
					\begin{aligned}
						&\frac{1}{2} \frac{\partial}{\partial t}\int_{\R}u_1 u_2 \psi_{j} \,dx =\frac{1}{2} \int_{\R}\partial_{t}u_1\,u_{2} \psi_{j} \,dx+\frac{1}{2}\int_{\R}u_1\,\partial_{t}u_{2} \psi_{j} \,dx+\frac{1}{2} \int_{\R}u_1 u_2 \partial_{t}\psi_{j}  \,dx.
					\end{aligned}
				\end{equation}
				Since $(u_{1},u_{2})$ is a solution of \eqref{sistema2}, we have
				\begin{equation}\label{var1}
					\begin{aligned}
						\frac{1}{2} &\int_{\R}\partial_{t}u_1\,u_{2} \psi_{j} \,dx+\frac{1}{2}\int_{\R}u_1\,\partial_{t}u_{2} \psi_{j} \,dx\\
						&=\frac{1}{2}\int_{\R}\partial_{x}u_2u_2 \psi_{j}  \,dx+ \frac{1}{2}\int_{\R}\left(\partial_{x}u_1-\partial_{xxx}u_1-(2p+1)|u_1|^{2p}\partial_{x}u_1\right) u_1 \psi_{j}  \,dx.
					\end{aligned}
				\end{equation}
				Moreover, 
				\begin{equation}\label{var2}
					\begin{aligned}
						\frac{1}{2} \int_{\R}u_{1}\,u_{2} \partial_{t} \psi_{j}  \,dx=-\frac{1}{2\sqrt{t}} \int_{\R}u_1u_2\left(\frac{x+m_{j} t}{2 t}\right) \psi_{j}^{\prime} \,dx.
					\end{aligned}  
				\end{equation}
				Thus, combining \eqref{varia1}, \eqref{var1}, and \eqref{var2}, we obtain
		\[
				\begin{aligned}
					\frac{1}{2} \frac{\partial}{\partial t}\int_{\R}u_1 u_2 \psi_{j} \,dx &=\frac{1}{2}\int_{\R}\partial_{x}u_2u_2 \psi_{j}  \,dx+ \frac{1}{2}\int_{\R}\left(\partial_{x}u_1-\partial_{xxx}u_1-(2p+1)|u_1|^{2p}\partial_{x}u_1\right) u_1 \psi_{j}  \,dx\\
					& \quad -\frac{1}{2\sqrt{t}} \int_{\R}u_1u_2\left(\frac{x+m_{j} t}{2 t}\right) \psi_{j}^{\prime} \,dx.
				\end{aligned}
\]
		This means that
				\begin{equation}\label{var3}
					\begin{aligned}
						\frac{1}{2} \frac{\partial}{\partial t}\int_{\R}u_1 u_2 \psi_{j} \,dx &=\frac{1}{2}\int_{\R}\partial_{x}u_2u_2 \psi_{j}  \,dx+\frac{1}{2}\int_{\R}\partial_{x}u_1 u_{1} \psi_{j}  \,dx-\frac{1}{2\sqrt{t}}\int_{\R}|\partial_{x}u_1 |^2 \psi_{j}^{\prime}  \,dx\\
						&\quad-\frac{1}{4}\int_{\R}\partial_{x}u_1 u_{1} \psi_{j}  \,dx-\frac{1}{2\sqrt{t}} \int_{\R}u_1 u_2\left(\frac{x+m_{j} t}{2 t}\right) \psi_{j}^{\prime} \,d x\\
						& \quad -\frac{2p+1}{2}\int_{\R}|u_1|^{2p}\partial_{x}u_1 u_1  \psi_{j}\,dx.
					\end{aligned}    
				\end{equation}
				Note that, \[\begin{aligned}
					\frac{1}{2}\int_{\R}\partial_{x}u_2 u_2 \psi_{j}  \,dx&=-\frac{1}{2}\int_{\R}u_2 \partial_{x}u_2 \psi_{j}  \,dx-\frac{1}{2\sqrt{t}}\int_{\R}|u_2|^2 \psi_{j}^{\prime}\,dx,
				\end{aligned}\]
				so we have \[ \frac{1}{2}\int_{\R}\partial_{x}u_2 u_2 \psi_{j}  \,dx=-\frac{1}{4\sqrt{t}}\int_{\R}|u_2|^2 \psi_{j}^{\prime}\,dx\]
				Similarly, \[ \frac{1}{2}\int_{\R}\partial_{x}u_1 u_1 \psi_{j}  \,dx=-\frac{1}{4\sqrt{t}}\int_{\R}|u_1|^2 \psi_{j}^{\prime}\,dx.\]
		Moreover, by using the fact
\[\begin{aligned}
					-\frac{2p+1}{2}\int_{\R}|u_1|^{2p}\partial_{x}u_1 u_1  \psi_{j}\,dx&=\frac{2p+1}{2}\int_{\R}\partial_{x}(|u_1|^{2p}u_1) u_1  \psi_{j}\,dx+\frac{2p+1}{2\sqrt{t}}\int_{\R}|u_1|^{2p+2} \psi_{j}^{\prime}\,dx\\
					&=\frac{(2p+1)^2}{2}\int_{\R}|u_1|^{2p}\partial_{x}u_1 u_1  \psi_{j}\,dx+\frac{2p+1}{2\sqrt{t}}\int_{\R}|u_1|^{2p+2} \psi_{j}^{\prime}\,dx,
				\end{aligned}\]
we deduce
\[-\frac{2p+1}{2}\int_{\R}|u_1|^{2p}\partial_{x}u_1 u_1  \psi_{j}\,dx=\frac{2p+1}{4p\sqrt{t}}\int_{\R}|u_1|^{2p+2} \psi_{j}^{\prime}\,dx\]
				So, combining these last estimates with \eqref{var3}, we obtain
				\begin{equation}\label{quase1}
					\begin{aligned}
						\frac{1}{2} \frac{\partial}{\partial t}\int_{\R}u_1 u_2 \psi_{j} \,dx
						&=-\frac{1}{4\sqrt{t}}\int_{\R}|u_2|^2 \psi_{j}^{\prime}\,dx-\frac{1}{8\sqrt{t}}\int_{\R}|u_1|^2 \psi_{j}^{\prime}\,dx-\frac{1}{2\sqrt{t}}\int_{\R}|\partial_{x}u_1 |^2 \psi_{j}^{\prime}  \,dx\\
						&\quad-\frac{1}{2\sqrt{t}} \int_{\R}u_1 u_2\left(\frac{x+m_{j} t}{2 t}\right) \psi_{j}^{\prime} \,d x+\frac{2p+1}{4p\sqrt{t}}\int_{\R}|u_1|^{2p+2} \psi_{j}^{\prime}\,dx.
					\end{aligned}    
				\end{equation}
				Therefore, 
				\begin{equation}\label{estimativamassaloc}
					\begin{aligned}
						\frac{1}{2} \frac{\partial}{\partial t}\int_{\R}u_1 u_2 \psi_{j} \,dx
						\lesssim
						\frac{1}{\sqrt{t}} \int_{\R}\left(\left|\partial_{x} u_{1}\right|^{2}+\left|u_{1}\right|^{2}\right) \psi_{j}^{\prime} \,dx +\frac{1}{\sqrt{t}} \int_{\R}\left|u_{2}\right|^{2}\psi_{j}^{\prime} \,dx +\frac{2p+1}{4p\sqrt{t}}\int_{\R}|u_1|^{2p+2} \psi_{j}^{\prime}\,dx.
					\end{aligned}
				\end{equation}
				Define $\Omega_{j}=\left[m_{j} t-\sqrt{t}, m_{j} t+\sqrt{t}\right]$. From \eqref{estimativamassaloc} and the support properties of $\psi_{j}^{\prime}$, we arrive at
				\begin{equation}\label{estmassa}
					\begin{aligned}
						\left|\frac{1}{2} \frac{\partial}{\partial t}\int_{\R}u_1 u_2 \psi_{j} \,dx\right|
						&\lesssim \frac{1}{\sqrt{t}} \int_{\Omega_{j}}\left(\left|\partial_{x} u_{1}\right|^{2}+\left|u_{1}\right|^{2}\right)  \,dx +\frac{1}{\sqrt{t}} \int_{\Omega_{j}}\left|u_{2}\right|^{2}\,dx +\frac{2p+1}{4p\sqrt{t}}\int_{\Omega_{j}}|u_1|^{2p+2} \,dx.  
					\end{aligned}
				\end{equation}
				Now, we observe by using the Gagliardo–Nirenberg inequality  that
				\begin{equation*}
					\begin{aligned}
						\int_{\Omega_{j}}\left|u_{1}\right|^{2p+2}\,dx&\lesssim    \left( \int_{\Omega_{j}}\left|\partial_{x}  u_{1}\right|^{2}+\left|u_{1}\right|^{2}\,dx\right)^{p} + \left( \int_{\Omega_{j}}\left|u_{1}\right|^{2}\,dx\right)^{p+2}.
					\end{aligned}
				\end{equation*}
			If we substitute $u_{j}=\varepsilon_{j}+\tilde{R}_{j}$ into \eqref{estmassa}, we have
				\begin{equation*}
					\begin{aligned}
						\left|\frac{1}{2} \frac{\partial}{\partial t}\int_{\R}u_1 u_2 \psi_{j} \,dx\right|   
						&\leq \frac{C}{\sqrt{t}} \int_{\Omega_{j}}\left(\left|\partial_{x} \tilde{R}_{1}\right|^{2}+\left|\tilde{R}_{1}\right|^{2}\right)  \,dx +\frac{C}{\sqrt{t}} \int_{\Omega_{j}}\left|\tilde{R}_{2}\right|^{2}\,dx\\
						& \quad +\frac{C}{\sqrt{t}}\left(\int_{\Omega_{j}}\left(\left|\partial_{x} \tilde{R}_{1}\right|^{2}+\left|\tilde{R}_{1}\right|^{2}\right)  \,dx\right)^{p}+\frac{C}{\sqrt{t}}\left(\int_{\Omega_{j}}\left|\tilde{R}_{1}\right|^{2}  \,dx\right)^{p+2}\\&\quad+\frac{C}{\sqrt{t}}\|\bu-\tilde{\ru}\|_{H}^{2}+\frac{C}{\sqrt{t}}\|u_1-\tilde{R}_1\|_{H^1(\R)}^{2p}+\frac{C}{\sqrt{t}}\|u_1-\tilde{R}_1\|_{H^1(\R)}^{2p+2}.
					\end{aligned}
				\end{equation*}
				Hence, \begin{equation}\label{termomas1}
					\begin{aligned}
						\left|\frac{1}{2} \frac{\partial}{\partial t}\int_{\R}u_1 u_2 \psi_{j} \,dx\right|   
						&\leq \frac{C}{\sqrt{t}} \int_{\Omega_{j}}\left(\left|\partial_{x} \tilde{R}_{1}\right|^{2}+\left|\tilde{R}_{1}\right|^{2}\right)  \,dx +\frac{C}{\sqrt{t}} \int_{\Omega_{j}}\left|\tilde{R}_{2}\right|^{2}\,dx\\
						& \quad +\frac{C}{\sqrt{t}}\left(\int_{\Omega_{j}}\left(\left|\partial_{x} \tilde{R}_{1}\right|^{2}+\left|\tilde{R}_{1}\right|^{2}\right)  \,dx\right)^{p}+\frac{C}{\sqrt{t}}\left(\int_{\Omega_{j}}\left|\tilde{R}_{1}\right|^{2}  \,dx\right)^{p+2}\\&\quad+\frac{C}{\sqrt{t}}e^{-2\omega_{\star}^{\frac32}t}+\frac{C}{\sqrt{t}}e^{-2p\omega_{\star}^{\frac32}t}+\frac{C}{\sqrt{t}}e^{-(2p+2)\omega_{\star}^{\frac32}t}.
					\end{aligned}
				\end{equation}
				Now, we estimate the terms involving $\tilde{R}_{1}$.    The case of $\tilde{R}_{2}$, we proceed similarly.
    
    Note that, using the decay of solitons (see Proposition \ref{decaimentoquadratico}),
				\begin{equation}\label{Rnosupor}
					\begin{split}					\int_{\Omega_{j}}\left|\tilde{R}_{1}\right|^{2}  \,dx &\leq C\sum_{k=1}^{N}\int_{\Omega_{j}}\left|\tilde{R}_{k}^{(1)}\right|  \,dx \leq  \sum_{k=1}^{N}\int_{\Omega_{j}}e^{-\frac{2}{3}\sqrt{1-\tilde{\omega}_{k}^2}|x-\omega_{k}t-x_{k}|}\,dx \\& \leq  \sum_{k=1}^{N}\int_{\Omega_{j}}e^{-\frac{2}{3}\sqrt{1-\tilde{\omega}_{k}^2}|x-\omega_{k}t-x_{k}|}\,dx.
					\end{split}
				\end{equation}
	Our aim is to estimate $\int_{\Omega_{j}}e^{-\frac{2}{3}\sqrt{1-\tilde{\omega}_{k}^2}|x-\omega_{k}t-x_{k}|},dx$. To do so, we will use the same arguments used in Lemma \ref{solitons}, so some facts will be written without much justification. In fact, if $j<k$, for $T_{0}$ sufficiently large,
				\begin{equation}\label{integralexpo}
					\begin{aligned}
						\int_{\Omega_{j}}e^{-\frac{2}{3}\sqrt{1-\tilde{\omega}_{k}^2}|x-\omega_{k}t-x_{k}|}\,dx&\leq C\int_{-\sqrt{t}+m_{j}t}^{\sqrt{t}+m_{j}t} e^{-\frac{2}{3}\sqrt{1-\tilde{\omega}_{k}^2}|x-\omega_{k}t|}\,dx \\
						&=C\int_{-\sqrt{t}+m_{j}t}^{\sqrt{t}+m_{j}t} e^{-\frac{2}{3}\sqrt{1-\tilde{\omega}_{k}^2}|x-m_{j}t+m_{j}t-\omega_{k}t|}\,dx \\
						&\leq C\int_{-\sqrt{t}+m_{j}t}^{\sqrt{t}+m_{j}t} e^{-\frac{2}{3}\sqrt{1-\tilde{\omega}_{k}^2}\left|x-m_{j}t+\frac{\omega_{j-1}-\omega_{k}}{2}t+\frac{\omega_{j}-\omega_{k}}{2}t\right|}\,dx\\
						&\leq Ce^{-\frac{1}{3}\sqrt{1-\tilde{\omega}_{k}^2}\omega_{\star}t}\int_{-\sqrt{t}+m_{j}t}^{\sqrt{t}+m_{j}t} e^{\frac{2}{3}\sqrt{1-\tilde{\omega}_{k}^2}\left|x-m_{j}t+\frac{\omega_{j}-\omega_{k}}{2}t\right|}\,dx\\
						&\leq Ce^{-\frac{1}{3}\sqrt{1-\tilde{\omega}_{k}^2}\omega_{\star}t}\int_{-\sqrt{t}+\frac{\omega_{j}-\omega_{k}}{2}t}^{\sqrt{t}+\frac{\omega_{j}-\omega_{k}}{2}t} e^{\frac{2}{3}\sqrt{1-\tilde{\omega}_{k}^2}|x|}\,dx\\
						&\leq Ce^{-\frac{1}{3}\sqrt{1-\tilde{\omega}_{k}^2}\omega_{\star}t}\int_{\R} e^{-\frac{2}{3}\sqrt{1-\tilde{\omega}_{k}^2}x}\,dx\\
						&\leq Ce^{-4\omega_{\star}^{\frac32}t}.
					\end{aligned}
				\end{equation}
				Now, if $j>k$, proceeding as in \eqref{cassoj>k}, we have for $T_{0}$ sufficiently large,
				\begin{equation}\label{integralexpo2}
					\begin{aligned}
						\int_{\Omega_{j}}e^{-\frac{2}{3}\sqrt{1-\tilde{\omega}_{k}^2}|x-\omega_{k}t-x_{k}|}\,dx&\leq C\int_{-\sqrt{t}+m_{j}t}^{\sqrt{t}+m_{j}t} e^{-\frac{2}{3}\sqrt{1-\tilde{\omega}_{k}^2}|x-\omega_{k}t|}\,dx\\
						&\leq C\int_{-\sqrt{t}+m_{k+1}t}^{\infty} e^{-\frac{2}{3}\sqrt{1-\tilde{\omega}_{k}^2}|x-\omega_{k}t|}\,dx \\
						&\leq  C \int_{-\sqrt{t}+m_{k+1} t}^{\sqrt{t}+m_{k+1} t} e^{-\frac{2}{3} \sqrt{1-\tilde{\omega}_{k}^2}\left|x-\omega_{k} t\right|} \,d x+C \int_{\sqrt{t}+m_{k+1}t}^{\infty} e^{-\frac{2}{3} \sqrt{1-\tilde{\omega}_{k}^2} |x-\omega_{k}t|} \,d x\\
						&\leq Ce^{-\frac{1}{3} \sqrt{1-\tilde{\omega}_{k}^2} \omega_{\star}t}(2\sqrt{t})e^{\frac{2}{3} \sqrt{1-\tilde{\omega}_{k}^2} \sqrt{t}}+Ce^{-\frac{2}{3} \sqrt{1-\tilde{\omega}_{k}^2} \sqrt{t}}e^{-\frac{1}{3} \sqrt{1-\tilde{\omega}_{k}^2} \omega_{\star}t}\\
						&\leq Ce^{-\frac{1}{3} \sqrt{1-\tilde{\omega}_{k}^2} \omega_{\star}t}\left((2\sqrt{t})e^{\frac{2}{3} \sqrt{1-\tilde{\omega}_{k}^2} \sqrt{t}}+e^{-\frac{2}{3} \sqrt{1-\tilde{\omega}_{k}^2} \sqrt{t}}\right)\\
						&\leq Ce^{-4\omega_{\star}^{\frac32}t}.
					\end{aligned}
				\end{equation}
				Lastly, if $j=k$, for $T_{0}$ sufficiently large, we have
				\begin{equation}\label{integralexpo3}
					\begin{aligned}
						\int_{\Omega_{k}}e^{-\frac{2}{3}\sqrt{1-\tilde{\omega}_{k}^2}|x-\omega_{k}t-x_{k}|}\,dx&\leq C\int_{-\sqrt{t}+m_{k}t}^{\sqrt{t}+m_{k}t} e^{-\frac{2}{3}\sqrt{1-\tilde{\omega}_{k}^2}|x-\omega_{k}t|}\,dx \\
						&=C\int_{-\sqrt{t}+m_{k}t}^{\sqrt{t}+m_{k}t} e^{-\frac{2}{3}\sqrt{1-\tilde{\omega}_{k}^2}|x-m_{k}t+m_{k}t-\omega_{k}t|}\,dx \\
						&\leq C\int_{-\sqrt{t}+m_{k}t}^{\sqrt{t}+m_{k}t} e^{-\frac{2}{3}\sqrt{1-\tilde{\omega}_{k}^2}\left|x-m_{k}t+\frac{\omega_{k-1}-\omega_{k}}{2}t\right|}\,dx\\
						&\leq Ce^{-\frac{1}{3}\sqrt{1-\tilde{\omega}_{k}^2}\omega_{\star}t}\int_{-\sqrt{t}+m_{k}t}^{\sqrt{t}+m_{k}t} e^{\frac{2}{3}\sqrt{1-\tilde{\omega}_{k}^2}\left|x-m_{k}t\right|}\,dx\\
						&\leq Ce^{-\frac{1}{3}\sqrt{1-\tilde{\omega}_{k}^2}\omega_{\star}t}\int_{-\sqrt{t}}^{\sqrt{t}} e^{\frac{2}{3}\sqrt{1-\tilde{\omega}_{k}^2}|x|}\,dx\\
						&\leq Ce^{-\frac{1}{3}\sqrt{1-\tilde{\omega}_{k}^2}\omega_{\star}t}\left(2\sqrt{t}e^{\frac{2}{3}\sqrt{1-\tilde{\omega}_{k}^2}\sqrt{t}}\right)\\
						&\leq Ce^{-4\omega_{\star}^{\frac32}t}.
					\end{aligned}
				\end{equation}
				Combining \eqref{Rnosupor}-\eqref{integralexpo3}, we conclude that
				\begin{equation}\label{R1}
					\begin{aligned}
						\int_{\Omega_{j}}\left|\tilde{R}_{1}\right|^{2}  \,dx \leq Ce^{-4\omega_{\star}^{\frac32}t}.
					\end{aligned}
				\end{equation}
				Similarly, we obtain
				\begin{equation}\label{R3}
					\begin{aligned}
						\int_{\Omega_{j}}\left|\partial_{x} \tilde{R}_{1}\right|^{2}  \,dx \leq Ce^{-4\omega_{\star}^{\frac32}t}&\leq  Ce^{-2\omega_{\star}^{\frac32}t}. 
					\end{aligned}
				\end{equation}
				Also, proceeding as in the estimates above, we arrive at  
				\begin{equation}\label{R5}
					\begin{aligned}
						\int_{\Omega_{j}}|R_{2}|^{2}\,dx \leq  Ce^{-2\omega_{\star}^{\frac32}t}.
					\end{aligned}
				\end{equation}
				Therefore, using  \eqref{segundoboostrap} and estimates \eqref{R1}-\eqref{R5} in \eqref{termomas1}, we conclude that
				\begin{equation}\label{termomas2}
					\begin{aligned}
						& \left|\frac{1}{2} \frac{\partial}{\partial t}\int_{\R}u_1 u_2 \psi_{j} \,dx\right|  \leq \frac{C}{\sqrt{t}}e^{-2\omega_{\star}^{\frac32}t}.
					\end{aligned}
				\end{equation}
				Note that equations  \eqref{termomas2}  holds for $j=1$, since in this case we can use the conservation of  momentum through the flow of system \eqref{sistema1}.
				
				Consequently, for $j\in \{1,2,\ldots,N\}$ and $T_{0}$ sufficiently large,
\[\left| \frac{\partial}{\partial t}\int_{\R}\mathcal{M}_j(\bu)\right|\leq \frac{C}{\sqrt {t}} e^{-2\omega_{\star}^{\frac32}t},\]
				which shows the desired result.
			\end{proof} 
			
			\begin{proof}[Proof of   Proposition  \ref{Bootstrap3}]
				Let us fix $t \in \left[t_{0}, T^{n}\right]$. Note that by using the estimate \eqref{parametros}, we have
				\begin{equation}\label{estimateuniforme1}
					\begin{aligned}
						\left\|\bu^{n}(t)-\ru(t)\right\|_{H} &\leq\left\|\bu^{n}(t)-\tilde{\ru}(t)\right\|_{H}+\left\|\tilde{\ru}(t)-\ru(t)\right\|_{H} \\
						&\leq \left\|\var(t)\right\|_{H}+C\sum_{j=1}^N|\tilde{\omega}_{j}(t)-\omega_{j}|+C\sum_{j=1}^N|\tilde{x}_{j}(t)-x_{j}|.
					\end{aligned}
				\end{equation}
				Now, using the fact that $\left(\varepsilon_{1}(t),\tilde{R}_{1}(t)\right)_{2}=0$, we have
				\begin{equation}\label{pp1}
					\begin{aligned}
						\int_{\R}u_{1}u_2 \phi_j d x 
						& =\int_{\R} \tilde{R}_{j}^{(1)} \tilde{R}_{j}^{(2)} \,dx+\int_{\R} \varepsilon_1 \varepsilon_2 \phi_{j}\,dx+O\left(e^{-3 \omega_{\star}^{\frac32} t}\right).
					\end{aligned}
				\end{equation}
				From Lemma \ref{conservationquase}, we know that 
				\begin{equation}\label{pp2}
					\begin{aligned}
						& \left|\int_{\R} u_{1}(t)u_{2}(t) \phi_j\,d x-\int_{\R} u_{1}(T^n)u_{2}(T^n) \phi_j\,d x \right|\leq \frac{C}{\sqrt{t}} e^{-2 \omega_{\star}^{\frac32} t}.
					\end{aligned}
				\end{equation}
				Substituting \eqref{pp1} into \eqref{pp2} and using the fact that $\varepsilon_{1}\left(T^{n}\right)=0$ and $\tilde{R}_{j}^{(1)}\left(T^{n}\right)=R_{j}(T^n)$, we arrive at
				\begin{equation}\label{pp3}
					\begin{aligned}
						& \left|\int_{\R} \tilde{R}_{j}^{(1)} \tilde{R}_{j}^{(2)} \,dx-\int_{\R} R_{j}^{(1)} R_{j}^{(2)} \,dx\right|  \leq\left\|\varepsilon_{1}(t)\right\|_{H^{1}(\R)}^{2}+\left\|\varepsilon_{2}(t)\right\|_{L^2(\R)}^{2}+\frac{C}{\sqrt{t}} e^{-2 \omega_{\star}^{\frac32}t}.
					\end{aligned}
				\end{equation}
				By Taylor expansion, since $\tilde{R}_{j}^{(2)}=-\tilde{\omega}_j\tilde{R}_{j}^{(1)}$, then
				\begin{equation}\label{pp4}
					\begin{aligned}
						\int_{\R} \tilde{R}_{j}^{(1)} \tilde{R}_{j}^{(2)} \,dx&=-\tilde{\omega}_{j}\int_{\R} \Phi_{\tilde{\omega}_{j}}^2 \,dx \\
						&=-\tilde{\omega}_{j}\left.\int_{\R}\Phi_{\omega_{j}}^{2}\, d x-\tilde{\omega}_{j}\left.\left(\partial_{\omega} \int_{\R}\Phi_{\omega}^{2}\right)\right|_{\omega=\omega_{j}}\right)\left( \tilde{\omega}_{j}(t)- \omega_{j}\right) \\
						&\quad+\left(\tilde{\omega}_{j}(t)-\omega_{j}\right) o\left(\left|\tilde{\omega}_{j}(t)-\omega_{j}\right|^{2}\right)  \\
						&=-\omega_{j}\int_{\R}\Phi_{\omega_{j}}^{2}\, d x+(\tilde{\omega}_j-\omega_{j})\left.\int_{\R}\Phi_{\omega_{j}}^{2}\, d x-\tilde{\omega}_{j}\left.\left(\partial_{\omega} \int_{\R}\Phi_{\omega}^{2}\right)\right|_{\omega=\omega_{j}}\right)\left( \tilde{\omega}_{j}(t)- \omega_{j}\right) \\
						&\quad+\left(\tilde{\omega}_{j}(t)-\omega_{j}\right) o\left(\left|\tilde{\omega}_{j}(t)-\omega_{j}\right|^{2}\right)
					\end{aligned}
				\end{equation}
				
				Once $\left|\tilde{\omega}_{j}(t)-\omega_{j}\right|^{2}$ is sufficiently small (for $T_0$ large enough), then \eqref{pp4} implies, 
\[
				\left|\int_{\R} \tilde{R}_{j}^{(1)} \tilde{R}_{j}^{(2)} \,dx-\int_{\R} R_{j}^{(1)} R_{j}^{(2)} \,dx\right| \geq C\left|\tilde{\omega}_{j}(t)-\omega_{j}\right|,
\]
				which, combined with \eqref{pp3}, leads to
				\begin{equation}\label{omega}
					\left|\tilde{\omega}_{j}(t)-\omega_{j}\right| \leq C\left\|\var(t)\right\|_{H}^{2}+\frac{C}{\sqrt{t}} e^{-2 \omega_{\star}^{\frac32}t}.   
				\end{equation}
				Next, we observe that from the definition of $\mathcal{E}_b$ (see \eqref{energia}) and of $\mathcal{S}$, we deduce,
\[
				\begin{aligned}
					\left|\mathcal{S}(\mathbf{u}(t))-\mathcal{S}(\mathbf{u}(T^{n}))\right|\leq\left| \mathcal{E}_b(\mathbf{u}(t))-\mathcal{E}_b(\mathbf{u}(T^{n}))\right|+C\left|  \sum_{j=1}^{N}\left(\mathcal{M}_j(\mathbf{u}(t))-\mathcal{M}_j(\mathbf{u}(T^{n}))\right)\right|.
				\end{aligned}
\]
				The conservation of energy and Lemma \ref{conservationquase} imply
\[
				\begin{aligned}
					\left|\mathcal{S}(\mathbf{u}(t))-\mathcal{S}(\mathbf{u}(T^{n}))\right|\leq \frac{C}{\sqrt{t}} e^{-2 \omega_{\star}^{\frac32}t}.
				\end{aligned}
\]
				It follows from the above estimate that
				\begin{equation}\label{SuTn}
					\begin{aligned}
						\mathcal{S}(\mathbf{u}(t))= \mathcal{S}(\mathbf{u}(T^{n}))+ \mathcal{O}\left( \frac{1}{\sqrt{t}}e^{-2 \omega_{\star}^{\frac32}t}\right).
					\end{aligned}
				\end{equation}
				Therefore, from \eqref{SuTn} and Lemma \ref{taylorS},
\[
				\begin{aligned}
					&\left\|\mathbf{\var}(t)\right\|_{H}^{2}\leq C\mathcal{H}(\var) \\
					&=\mathcal{S}(\mathbf{u}(t))-\sum_{j=1}^N\left\{\mathcal{E}(\ru_{j})+ \frac{\omega_{j}}{2}\int_{\R}R_{j}^{(1)} R_{j}^{(2)}\,dx+ \mathcal{O}\left(|\tilde{\omega}_{j}(t)-\omega_{j}|^{2}\right)\right\}-\mathcal{O}\left(\frac{1}{\sqrt{t}}e^{-2\omega_{\star}^{\frac32}t}\right)\\
					&\leq  \mathcal{S}(\mathbf{u}(T^{n})) -\sum_{j=1}^N\left\{\mathcal{E}(\ru_{j})+  \frac{\omega_{j}}{2}\int_{\R}R_{j}^{(1)} R_{j}^{(2)}\,dx+ \mathcal{O}\left(|\tilde{\omega}_{j}(t)-\omega_{j}|^{2}\right)\right\} +\frac{C}{\sqrt{t}}e^{-2\omega_{\star}^{\frac32}t}.
				\end{aligned}
\]
				Since $\tilde{\omega}_j(T^n)=\omega_j$, $\var(T^n)=0$ and from  Lemma \ref{taylorS} we have
\[
				\mathcal{S}(\mathbf{u}(T^{n})) =\sum_{j=1}^N\left\{\mathcal{E}(\ru_j)+  \frac{\omega_{j}}{2}\int_{\R}R_{j}^{(1)} R_{j}^{(2)}\,dx\right\} +\mathcal{O}\left(\frac{1}{\sqrt{t}}e^{-2\omega_{\star}^{\frac32}t}\right),
\]
				and from this, combined with the last inequality and \eqref{omega}, we obtain
				\begin{equation}\label{epsilonestimative}
					\begin{aligned}
						\left\|\mathbf{\var}(t)\right\|_{H}^{2}& \leq \sum_{j=1}^N\mathcal{O}\left(|\tilde{\omega}_{j}(t)-\omega_{j}|^{2}\right)+\frac{C}{\sqrt{t}}e^{-2\omega_{\star}^{\frac32}t} \leq Ce^{-4\omega_{\star}^{\frac32}t}+\frac{C}{\sqrt{t}}e^{-2\omega_{\star}^{\frac32}t} \leq \frac{C}{\sqrt{t}}e^{-2\omega_{\star}^{\frac32}t},
					\end{aligned}
				\end{equation}
				for $T_0$ sufficiently large. This last estimate plays a crucial role in proving the proposition. In fact, from \eqref{modulated}, \eqref{omega}, and \eqref{epsilonestimative}, we have for $T_{0}$ sufficiently large,
\[
				\begin{aligned}
					\left|\partial_{t } \tilde{x}_{j}(t)\right|&\leq C\left\|\var(t)\right\|_{H}+Ce^{-3 \omega_{\star}^{\frac32}t} \leq \frac{C}{\sqrt{t}} e^{-2 \omega_{\star}^{\frac32}t}+C\left\|\var(t)\right\|_{H} \leq \frac{C}{\sqrt{t}} e^{-2 \omega_{\star}^{\frac32}t}+\frac{C}{\sqrt{t}} e^{- \omega_{\star}^{\frac32}t} \leq \frac{C}{\sqrt{t}} e^{- \omega_{\star}^{\frac32}t}.
				\end{aligned}
\]
				Since $\tilde{x}_j(T^n) = x_j$, according to the fundamental theorem of calculus and the above estimate, 
				\begin{equation}\label{estimativeCal}
					\begin{aligned}
						\left|\tilde{x}_{j}(t)-x_{j}\right|&=\left|\int_{t}^{T^n}\partial_{s}\tilde{x}_{j}(s)ds   \right| \leq  \int_{t}^{T^n}\left|\partial_{s}\tilde{x}_{1}(s)   \right|ds\\
						& \leq \frac{C}{\sqrt{t}} \int_{t}^{T^n}e^{- \omega_{\star}^{\frac32}s}\,ds   \leq \frac{C}{\sqrt{t}}\left(e^{-\omega_{\star}^{\frac32}t}-e^{-\omega_{\star}^{\frac32}T_{n}}\right)\\
						&\leq \frac{C}{\sqrt{t}} e^{- \omega_{\star}^{\frac32}t}.
					\end{aligned}  
				\end{equation}
				Therefore, using \eqref{omega}, \eqref{epsilonestimative}, and \eqref{estimativeCal} in the estimate \eqref{estimateuniforme1}, it follows that
\[
				\begin{aligned}
					\left\|\bu^{n}(t)-\ru(t)\right\|_{H} &\leq \frac{C}{\sqrt{t}}e^{-\omega_{\star}^{\frac32}t},
				\end{aligned}
\]
				for all $t \in \left[t_{0}, T^{n}\right]$.
				Therefore, for $T_{0}$ sufficiently large,
				\begin{equation*}
					\begin{aligned}
						\left\|\mathbf{u}^{n}(t) -\mathbf{R}(t) \right \|_{H}\leq  \frac{C}{\sqrt{t}}e^{- \sqrt{\omega_{\star} }\omega_{\star} t} , \, \, \text{for all $t \in\left[t_{0}, T^{n}\right]$.}
					\end{aligned}
				\end{equation*}
				Taking $t \geq (2C)^2$, it follows that for all $t \in [t_0, T^n]$
\[
				\begin{aligned}
					\left\|\mathbf{u}^{n}(t)-\mathbf{R}(t)\right \|_{H} \leq   \frac{1}{2} e^{- \omega_{\star}^{\frac32}t}, 
				\end{aligned}
\]
				which shows the desired result.
			\end{proof}
			\subsection{Proof of the main theorem}
			As we have already established, to prove the main theorem, we  need initially prove Propositions \ref{UniformEstimates} and \ref{Compactness}.
  
			\begin{proof}[Proof of Proposition \ref{UniformEstimates}]
				Let $T_{0}$, $n_{0}$  be given by  Proposition \ref{Bootstrap3}. Set $n \in \N$ with $n \geq n_{0}$. As we will see, the argument below shows that as long as the approximate solution $\mathbf{u}^{n}$ exist it satisfies \eqref{desiesti}, which in turn implies that it does not blow up in finite time. Consequently,  we may assume that it is defined in $[T_{0},T^{n}]$ and \[\mathbf{u}^{n} \in \mathbf{C}([T_{0},T^{n}],H^1(\R)) .\]
				Furthermore, since  $\mathbf{R}$ is continuous with respect to time, it follows that for every $t$ sufficiently close to $T^{n} $ (from the left) 
				\begin{equation*}
					\|\mathbf{u}^{n}(t)-\mathbf{R}(t)\|_{ H}\\ \leq \|\mathbf{u}^{n}(t)-\mathbf{u}^{n}(T^{n})\|_{H} + \|\mathbf{R}(T^{n})-\mathbf{R}(t)\|_{H} \leq e^{-\sqrt{\omega_{\star}}v_{\star}t}.
				\end{equation*}
				Now, let us consider 
				\[ t_{\sharp}:= \inf\left\{t_{\dagger}\in [T_{0},T^{n}]: \, \mbox{\eqref{desiesti} is satisfied for
					all $t \in [t_{\dagger},T^{n}]$}\right\}.\]
				Clearly $t_{\sharp} < T^{n}$. So, to prove the proposition we  just need to show that $t_{\sharp}=T_{0}$.
				Assume by contradiction that $t_{\sharp} >T_{0}$. Then, by Proposition \ref{Bootstrap3}, for all $t \in [t_{\sharp},T^{n}]$, we have  \begin{equation*}
					\|\mathbf{u}^{n}(t)-\mathbf{R}(t)\|_{H}\leq \frac{1}{2}e^{-\sqrt{\omega_{\star}}v_{\star}t}.
				\end{equation*}
				Consequently,
				\begin{equation*}
					\|\mathbf{u}^{n}(t_{\sharp}))-\mathbf{R}(t_{\sharp})\|_{H}\leq \frac{1}{2}e^{-\sqrt{\omega_{\star}}v_{ \star}t_{\sharp}}.
				\end{equation*}
				By the above argument, the continuity of $\mathbf{u}^{n}$ implies that for $t$ close enough to $t_{\sharp}$ (from the left),
				\begin{equation*}
					\|\mathbf{u}^{n}(t)-\mathbf{R}(t)\|_{H}\leq e^{-\sqrt{\omega_{\star}}v_{\star}t},
				\end{equation*}
				which contradicts the minimality of $t_{\sharp}$. Thus, one should have $t_{\sharp}=T_{0}$ and the proof is completed.
			\end{proof}
			\begin{lemma}[$\mathbf{L}^{2}$-Compactness]\label{L2-Compactness}
				Let $\delta >0$. There exists $r_{\delta}>0$ such that for every sufficiently large $n$,
				\[ \int\limits_{|x|> r_{\delta}} \left(|\partial_x u_1^n(T_{0})|^2+|u_1^n(T_{0})|^2+|u_2^n(T_{0})|^2\right)\,dx \leq \delta. \]
			\end{lemma}
			\begin{proof}
				Set  $n$ large enough such that estimate \eqref{desiesti} holds. We may assume there is  $T_{\delta}\in [T_{0},T^{n}]$ such that $e^{-2\sqrt{\omega_{\star}}v_{\star}T_{\delta}} \leq \frac{\delta}{16}$ (otherwise we may take a large $n$). It is clear that
				\begin{equation}\label{1.}
					\|(u_1^{n}(T_{\delta}),u_2^{n}(T_{\delta}))-(R_1(T_{\delta}),R_2(T_{\delta}))\|_{H^1(\R)\times L^2(\R)}^2\leq e^{-2\sqrt{\omega_{\star}}v_{\star}T_{\delta}} \leq \frac{\delta}{16}.
				\end{equation}
				Also, since $R_{1}$ and $R_{2}$ decay exponentially we may guarantee the existence of $\rho_{\delta}>0$ such that
				\begin{equation}\label{33.}
					\int\limits_{|x|> \rho_{\delta}}  \left(|R_1(T_{\delta})|^2+|R_2(T_{\delta})|^2\right)\,dx \leq \frac{\delta}{16}.
				\end{equation}
				Estimates \eqref{1.} and \eqref{33.} then imply
				\begin{equation}\label{44.}
					\int\limits_{|x|> \rho_{\delta}}  \left(|\partial_x u_1^n(T_{\delta})|^2+|u_1^n(T_{\delta})|^2+|u_2^n(T_{\delta})|^2\right)\,dx \leq \frac{\delta}{4}.
				\end{equation}
				Our goal now is to show  that inequality \eqref{44.}  holds with $T_{0}$ instead of $T_{\delta}$, with a possible change of the parameter $\rho_\delta$.

    On the other hand, let $\tau_{\rho_{\delta}
}$ be a cutoff function such that

\[
\tau_{\rho_{\delta}}(x)=\left\{\begin{array}{ll}
1 & \text { for }|x|>2 \rho_{\delta} \\
0 & \text { for }|x|<\rho_{\delta}
\end{array} \quad \text { and } \quad\left|\tau_{\rho_{\delta}}^{\prime}\right|<\frac{C}.{\rho_{\delta}}\right.
\]
Hence, from \eqref{44.}, \[\left\|(u_1^n\left(T_{\delta}\right),u_2^n\left(T_{\delta}\right)) \tau_{\rho_{\delta}}\right\|_{H^{1}(\R) \times L^{2}(\R)}^{2} \leq \frac{\delta}{2}.\]
Let us denote by $U_{n}^{T_{\delta},\rho_{\delta}}=(u_1^n\left(T_{\delta}\right),u_2^n\left(T_{\delta}\right)) \tau_{\rho_{\delta}}$ and $U_{n}^{\rho_{\delta}}$ its corresponding final data, then
 by the wellposedness theory of $H^{1}(\R) \times L^{2}(\R)$, due to the continuous dependence on the final data, we have 
\[
\left\|U_{n}^{\rho_{\delta}}(t)\right\|_{H^{1}(\R) \times L^{2}(\R)}^{2} \leq C \delta.
\]

Therefore,  by uniqueness on light cones, $U_{n}^{\rho_{\delta}}$ and $(u_1^n,u_2^n)$ coincide on $\left\{(t, x) \in \mathbb{R} \times \mathbb{R}^{n}:|x|>2 \rho_{\delta}+\left(T_{\delta}-t\right)\right\}$ for $t \in\left(-\infty, T_{\delta}\right)$, this implies that

\begin{equation*}
\int_{|x|>2 \rho_{\delta}+\left(T_{\delta}-T_{0}\right)}\left|\partial_x u_{1}^n\left(T_{0}\right)\right|^{2}+\left|u_{1}^n\left(T_{0}\right)\right|^{2}+\left| u_{2}^n\left(T_{0}\right)\right|^{2}\, d x \leq\left\|U_{n}^{\rho_{\delta}}(t)\right\|_{H^{1}(\R) \times L^{2}(\R)}^{2} \leq C \delta.
\end{equation*}

We could choose $r_{\delta}:=\max \left\{r_{\delta}, 2 \rho_{\delta}+\left(T_{\delta}-T_{0}\right)\right\}$ to deduce that
\begin{equation*}
\int_{|x|>r_{\delta}}\left|\partial_x u_{1}^n\left(T_{0}\right)\right|^{2}+\left|u_{1}^n\left(T_{0}\right)\right|^{2}+\left| u_{2}^n\left(T_{0}\right)\right|^{2}\, d x \leq\left\|U_{n}^{\rho_{\delta}}(t)\right\|_{H^{1}(\R) \times L^{2}(\R)}^{2} \leq  \frac{\delta}{2}.
\end{equation*}
			\end{proof}
			Now, we are ready to prove   Proposition \ref{Compactness}.
			\begin{proof}[Proof of Proposition \ref{Compactness}]
				In view of \eqref{desiesti} we have that $\mathbf{u}^n(T_{0})$ is bounded in $H$. Hence, up to a subsequence, there is $\mathbf{u}^{0} \in H$ such that
			$					\mathbf{u}^n(T_{0}) \rightharpoonup \mathbf{u}^{0}$ in $H.	$		 
				The idea now is to show that this convergence is, up to a subsequence, strongly  in $H^s(\R)\times H^{s-1}(\R)$ for $0<s<1$.

 Let us $\delta>0$. From Lemma \ref{L2-Compactness} we infer that for $n$ large enough, there holds\begin{equation}\label{est.u0}
     \left\|\bu^{n}(T_0)\right\|_{H^1(|x|>r_{\delta})\times L^2(|x|>r_{\delta})}^2+\left\|\bu^{0}\right\|_{H^1(|x|>r_{\delta})\times L^2(|x|>r_{\delta})}^2 \leq \frac{\delta}{2}.
 \end{equation}
Define $\chi_{\delta}: \mathbb{R}^{n} \rightarrow[0,1]$ a cutoff function such that \[\chi_{\delta}(x)=\left\{\begin{array}{ll}
1 & \text { for }|x|< r_{\delta} \\
0 & \text { for }|x|>2r_{\delta}
\end{array}\right.\]  and $\left| \chi_{\delta}^{\prime}\right| \leq 1$. Then, 

\[
\left\|\bu^{n}(T_0)-\bu^{0}\right\|_{H^{s}(\R) \times H^{s-1}(\R)} \leq\left\|\left(\bu^{n}(T_0)-\bu^{0}\right) \chi_{\delta}\right\|_{H^{s}(\R) \times H^{s-1}(\R)}+\left\|\left(\bu^{n}(T_0)-\bu^{0}\right)\left(1-\chi_{\delta}\right)\right\|_{H^{s}(\R) \times H^{s-1}(\R)}
\]

Since  $H^{\sigma}(\Omega) \hookrightarrow H^{\sigma-1}(\Omega)$ is compact with $\Omega=\{x:|x|<r_{\delta}\}$  bounded, for $n$ large enough and possibly up to a subsequence, we obtain

\[
\left\|\left(\bu^{n}(T_0)-\bu^{0}\right) \chi_{\delta}\right\|_{H^{s}(\R) \times H^{s-1}(\R)} \leq \frac{\delta}{2}.
\]
Moreover, using \eqref{est.u0}, it follows that 

\[
\left\|\left(\bu^{n}(T_0)-\bu^{0}\right)\left(1-\chi_{\delta}\right)\right\|_{H^{s}(\R) \times H^{s-1}(\R)} \leq\left\|\left(\bu^{n}(T_0)-\bu^{0}\right)\left(1-\chi_{\delta}\right)\right\|_{H^{1}(\R) \times L^{2}(\R)} \leq \frac{\delta}{2}.
\]
Therefore, we get the desired result.   
			\end{proof}
			Finally, we will prove the main theorem.
			\begin{proof}[Proof of Theorem \ref{maintheorem}]
				Proposition \ref{UniformEstimates} guarantees the existence of $T_{0}\in \R$ and $n_{ 0} \in\N$ such that, for every $n \geq n_{0}$ and each $t \in [T_{0},T^{n}]$,
				\begin{equation}\label{des}
					\|\mathbf{u}^{n}(t)-\mathbf{R}(t)\|_{H }\leq e^{-\sqrt{\omega_{\star}}v_{\star}t}.
				\end{equation}
				Let $\mathbf{u}^{0}$ be as obtained in Proposition \ref{Compactness}. From the local well-posedness of \eqref{sistemaeq} in $H^s(\R)\times H^{s-1}(\R)$, for $\frac{1}{2}<s\leq 1$, there is a solution $\mathbf{u}$ of  \eqref{sistemaeq}  with initial condition $\mathbf{u}(T_{0})= \mathbf{u}^{0}$ defined in an interval $[T_{0},T^{\star})$, where $T^{\star}$ is the maximum time of existence. We will show that $T^{\star}= +\infty$ and that $\mathbf{u}$ satisfies estimate \eqref{desprin}. In fact, suppose $T^{\star}<\infty$. Then, using  Proposition \ref{Compactness} we have
	\[\mathbf{u}^{n}(T_{0}) \rightarrow \mathbf{u}^{0} \quad \mbox{in} \quad H^s (\R)\times H^{s-1}(\R), \quad 0\leq  s<1. \]
				By the continuous dependence on the initial data, for $t \in [T_{0}, T^{\star})$,
				\[ \mathbf{u}^n(t) \rightarrow \mathbf{u}(t) \quad \mbox{in} \quad H^s (\R)\times H^{s-1}(\R), \quad 0< s<1. \]
				Now, by  \eqref{des} we get (taking $n$ large enough such that $T^{n}>T^{\star}$) that $\mathbf{u}^n(t)$ is bounded in $H$ for all $t\in [T_{0},T^{\star })$, giving that, up to a subsequence, it converges weakly in $H$. Therefore, for all $ t \in [T_{0}, T^{\star})$,
				\[ \mathbf{u}^n(t) \rightharpoonup \mathbf{u}(t) \quad \mbox{in} \quad H. \]
				So,  from \eqref{des},
				\begin{equation*}
					\|\mathbf{u}(t)-\mathbf{R}(t)\|_{H} \leq 
					\liminf_{n \rightarrow +\infty} \norm{\mathbf{u}^{n}(t)-\mathbf{R}(t)}_{H}
					\leq e^{-\sqrt{\omega_{\star}}v_{\star}t}.
				\end{equation*}
				Hence, $\mathbf{u}$ is bounded in $H$  for all $t\in [T_{0},T^{\star })$. But, this contradicts the blow-up alternative  in $H$ given by Theorem \ref{LWP2}. This means that $T^{\star}=+\infty.$ Clearly, the above argument also shows that \eqref{desprin} holds for all $t\in[T_0,\infty)$. The proof of the main theorem is thus completed.
			\end{proof}
			\section{ Supercritical case $p>2$} \label{supercritical}
			In this section, we aim to prove Theorem \ref{maintheorem} under the assumption of being in supercritical condition $p>2$. To achieve this, we will need to slightly modify the definition of $\omega_{\star}$, in fact, we will need to add new components derived from the new coercivity property.
			
			In fact, in the previous results, the condition of $p<2$ was notable, specifically when obtaining the coercivity property. In fact, in the subcritical case, we obtained  for $j$ fixed,
\[
			\left\langle S_j^{\prime \prime}\left(\mathbf{\Phi}_{\omega_j}\right) \mathbf{\eta}, \mathbf{\eta}\right\rangle \geq C\|\mathbf{\eta}\|_{ H }^2 .
\]
if  \begin{equation}\label{cond0}
				\left\langle\mathbf{\eta}, \partial_x \Phi_{\omega_j}\right\rangle=\left\langle\mathbf{\eta}, \mathbf{\Gamma}_{\omega_j}\right\rangle=0,
			\end{equation}
			where $\mathbf{\Gamma}_{\omega_j}=\left(\Phi_{\omega_j}^{(1)}, 0\right)$  obtained through modulation arguments. On the contrary, in the supercritical $L^2(\R)$ scenario, it is not possible to achieve uniform estimates in a similar manner due to the failure of the aforementioned property. Just like in the previous case, we can control directions around $\partial_x \ru$, while controlling the other direction through the scaling parameter, as in the subcritical case, is not feasible. In this case, we aim to establish a variant of the coercivity property that allows us to obtain new uniform estimates.
			
			In this regard, we recall the definition of the operator $S_j^{\prime \prime}$. In fact, \[\mathcal{S}_{j}''(\bm{\Phi}_{\omega_j})=\begin{pmatrix}
				L& \omega_j\\
				\omega_j & I 
			\end{pmatrix},\]
			where, $Lw =-\partial_{xx} w+w-(2p+1)|\Phi_{\omega_j}|^{2p}w,$ or equivalently 
			\begin{equation*}
				\mathcal{S}_{j}''(\bm{\Phi}_{\omega_j})\bw =\begin{pmatrix}
					-\partial_{xx} w_{1}+w_1+\omega_j w_2-(2p+1)|\Phi_{\omega_j}|^{2p}w_1\\
					w_2+\omega_j w_1
				\end{pmatrix},
			\end{equation*}
			from which, for  $\mathcal{L}_j :=\mathcal{S}_{j}''(\Phi_{\omega_j})$, 
	\[\begin{aligned}
				 \langle \mathcal{L}_j(\bw), \bw \rangle  =\int_{\R}|\partial_x w_1|^2\,dx+\int_{\R}| w_1|^2\,dx+2\omega_j\int_{\R} w_2 w_1\,dx+\int_{\R}| w_2|^2\,dx-(2p+1)\int_{\R}|\Phi_{\omega_j}|^{2p}w_1^2\,dx.
			\end{aligned}\]
			Now, \eqref{sistema2} is equivalent to the  Hamiltonian system $\bu_t=J\mathcal{E}'(\bu)$, where
			\[
  J=  \begin{bmatrix}
				0&\partial_x\\
				\partial_x&0
			\end{bmatrix},
			\]
			and  the solitary waves satisfy
			\[
			\bu_t=JS'(\bu)=J(\mathcal{E}(\bu)+\omega_j \mathcal{M}(\bu))',\qquad J=  \begin{bmatrix}
				0&\partial_\xi\\
				\partial_\xi&0
			\end{bmatrix},\quad \text{with} \,\,\xi=x-\omega_j t.
			\]
			We consider the evolution of small
			perturbations of the solitary wave, writing  
$		 
			\bu=\bm{\Phi}_{\omega}+\vu,
$
			then
			\begin{equation}
				\lam\vu=JS_j''(\vu)
			\end{equation}
			that is
			\[
			\begin{cases}
				\omega_j\partial_\xi u+\partial_\xi v=\lambda u,\\
				\partial_\xi Lu +\omega_j\partial_\xi v=\lambda v,
			\end{cases}
			\]
			that gives
			\[
			\partial_\xi ^2 (L -\omega_j^2)u -\lambda^2u+2c\lambda \partial_\xi u=0.
			\]
			Therefore, by virtue of the study conducted in \cite{WEIN}, we know that there exist two eigenfunctions $\mathbf{Y}_j^{\pm}=(Y_j^{\pm,1},Y_j^{\pm,2})$ of $\partial_x \mathcal{L}_j$, with $\mathbf{Y}_j^{-}(x)=\mathbf{Y}_j^{+}(-x)$, such that,
$ J \mathcal{L}_j (\mathbf{Y}_j^{\pm})=\pm \lambda_{0}^j\mathbf{Y}_j^{\pm}.$			Now, if we consider $\mathbf{Z}_j^{\pm}=\mathcal{L}_j(\mathbf{Y}_j^{\pm})$, then we obtain $  \mathcal{L}_j (J\mathbf{Z}_j^{\pm})=\pm \lambda_{0}^j \mathbf{Z}_j^{\pm}.$ 
			It is worth noting that, from \cite{WEIN}, the functions $\mathbf{Z}_j^{\pm}$ satisfy $\|\mathbf{Z}_j^{\pm}\|_{L^2(\R)}^2=1$ and decay to zero at infinity. Moreover, for $\epsilon>0$, \[|\mathbf{Z}_j^{\pm}|+|\partial_x \mathbf{Z}_j^{\pm}|\leq e^{-\epsilon |x|}.\]
   Next, we will establish a result that is a consequence of the theory established by Pego-Weinstein in \cite{WEIN}.
   \begin{lemma}
      The following properties hold:

(i) $\bm{Z}_j^{ \pm}$ are  eigenfunctions of $\mathcal{L}_{j} J.$

(ii)  For all  $\eta_{0}>0$, $x \in \mathbb{R}$ and  $m=1,2$,
\[
\left|Y_{j}^{ \pm,m}(x)\right|+\left|\partial_{x} Y_{j}^{ \pm,m}(x)\right|+\left|Z_{j}^{ \pm,m}(x)\right|+\left|\partial_{x} Z_{j}^{ \pm,m}(x)\right| \leq C e^{-\eta_{0} \sqrt{c}|x|}.
\]

(iii) 
$\left(\bm{Y}_{j}^{\pm}, \bm{Z}_{j}^{\pm}\right)_{L^{2}(\R)\times L^{2}(\R)}=0$ and $\left(\bm{Z}_{j}^{\pm}, \partial_x\bm{\Phi}_{\omega_j}\right)_{L^{2}(\R)\times L^{2}(\R)}=0.$

(iv)   \[\left(\bm{Y}_{j}^{+}, \bm{Z}_{j}^{-}\right)_{L^{2}(\R)\times L^{2}(\R)}=\left(\bm{Y}_{j}^{-}, \bm{Z}_{j}^{+}\right)_{L^{2}(\R)\times L^{2}(\R)}=1.\]
   \end{lemma}
   \begin{proof}
       Items $(i)$ and $(ii)$ follow from the strategy used in \cite{WEIN}. On the other hand, let us consider $(iii)$.

In fact, note that
\[\begin{aligned}
    \left(\bm{Y}_{j}^{+}, \bm{Z}_{j}^{+}\right)_{L^{2}(\R)\times L^{2}(\R)} &= \frac{1}{\lambda_{0}^j}\left(\bm{Y}_{j}^{+}, \mathcal{L}_j J\bm{Z}_{j}^{+}\right)_{L^{2}(\R)\times L^{2}(\R)} \\&= \frac{1}{\lambda_{0}^j}\left( \mathcal{L}_j\bm{Y}_{j}^{+}, J\bm{Z}_{j}^{+}\right)_{L^{2}(\R)\times L^{2}(\R)} \\
    &=-\frac{1}{\lambda_{0}^j}\left(J \mathcal{L}_j\bm{Y}_{j}^{+}, \bm{Z}_{j}^{+}\right)_{L^{2}(\R)\times L^{2}(\R)}\\
    & = -\frac{1}{\lambda_{0}^j}\lambda_{0}^j\left(\bm{Y}_{j}^{+}, \bm{Z}_{j}^{+}\right)_{L^{2}(\R)\times L^{2}(\R)}.
\end{aligned}\]
Hence, $\left(\bm{Y}_{j}^{+}, \bm{Z}_{j}^{+}\right)_{L^{2}(\R)\times L^{2}(\R)} =0.$ Then, applying a similar argument, it is not difficult to demonstrate the remaining results.
       \end{proof}
			Now, we are going to establish the Coercivity result that we will apply in this scenario.

Before establishing the coercivity result, we should derive some properties related to $\mathcal{L}_{j}$. They are standard results and follow as a consequence of the  result \cite{mwein}.

\begin{lemma}\label{kernel-lem}
    The kernel of $\mathcal{L}_{j}$ satisfies that

\[
\operatorname{Ker}\left(\mathcal{L}_{j}\right)=\left\{C \partial_{x} \bf{\Phi}_{\omega_{j}}: C \in \mathbb{R}\right\}
\]
\end{lemma}

\begin{proof}
  Let  $\bf{f} \in\left\{C \partial_{x} \bf{\Phi}_{\omega_{j}}: C \in\right.$ $\mathbb{R}\}$,  then from \eqref{ellipticeq},
\begin{equation*}
\mathcal{L}_{j} \bf{f}=\mathcal{L}_{j}\left(C \partial_{x} \bf{\Phi}_{\omega_{j}}\right)=C\binom{\partial_{x}\left(-\partial_{x x} \Phi_{\omega_j}+\left(1-\omega^{2}\right) \Phi_{\omega_j}-\Phi_{\omega_j}^{2p+1}\right)}{-\omega_j \partial_{x}\Phi_{\omega_j}+\omega_j \partial_{x}\Phi_{\omega_j}}=\binom{0}{0}
\end{equation*}

It follows that   $\bf{f}$ is in the kernel of $\mathcal{L}_{j}$. This is, 
\[
\left\{C \partial_{x} \bf{\Phi}_{\omega_{j}}: C \in \mathbb{R}\right\}
 \subset \operatorname{Ker}\left(\mathcal{L}_{j}\right). \]
 Now, let $\bf{f}=(f_1,f_2) \in \operatorname{Ker}\left(\mathcal{L}_{j}\right)$, so  we have
\begin{equation}\label{kernel1}
 \left\{\begin{aligned}
-\partial_{x x} f_1+ f_1-(2p+1) \Phi_{\omega_j}^{2p} f_1+\omega_jf_2 & =0  \\
f_2+\omega_j f_1 & =0,
\end{aligned}\right.   
\end{equation}
Or equivalently
\[-\partial_{x x} f_1+ (1-\omega_j^2)f_1-(2p+1) \Phi_{\omega_j}^{2p} f_1 =0.\]
Then, the only solution for \eqref{kernel1} are
\[
\left\{\begin{array}{l}
f_1=C \partial_{x} \Phi_{\omega_j}, \\
f_2=-C \omega_j \partial_{x} \Phi_{\omega_j}, \quad C \in \mathbb{R}
\end{array}\right.
\]
This implies that $\bf{f} \in\left\{C \partial_{x} \bf{\Phi}_{\omega_{j}}: C \in \mathbb{R}\right\}$, and we have
\[
\operatorname{Ker}\left(\mathcal{L}_{j}\right) \subset\left\{C \partial_{x} \bf{\Phi}_{\omega_{j}}: C \in \mathbb{R}\right\}.
\]
And the proof is complete.   
\end{proof}
Next, we will prove that there exists a unique negative eigenvalue for $\mathcal{L}_{j}$.

\begin{lemma}\label{eigenuniq}
  $\mathcal{L}_{j}$  has  only one negative eigenvalue. 
\end{lemma}
\begin{proof}
Initially, we should recall that the operator $-\partial_{xx} + (1 - \omega_j^2) - (2p + 1) \Phi_{\omega_j}^{2p}$ has only one negative eigenvalue, see for instance \cite{mwein}. In this case, let us denote such eigenvalue by $\lambda^{-}$. Then there exists a unique associated eigenvector $\psi_0\in H^{1}(\mathbb{R})$ such that
\begin{equation}\label{eigennegat}
-\partial_{x x} \psi_0+\left(1-\omega_j^{2}\right) \psi_0-(2p+1) \Phi_{\omega_j}^{2p} \psi_0=\lambda^{-} \psi_0.
\end{equation}
Therefore, for  $\bf{\Phi}_{\omega_{j}}=(\Phi_{\omega_j},-\omega_j \Phi_{\omega_j})$, we have 
\[
\begin{split}
  \left\langle \mathcal{L}_{j} \bf{\Phi}_{\omega_{j}}, \bf{\Phi}_{\omega_{j}}\right\rangle  
&  =\int_{\mathbb{R}}\left(-\partial_{x x} \Phi_{\omega_j} \Phi_{\omega_j}+(1-\omega_j^2)\Phi_{\omega_j}\Phi_{\omega_j}-(2p+1) \Phi_{\omega_j}^{2p+2}\right) \,dx   =-2p\int_{\R}\Phi_{\omega_j}^{2p+2}\,dx<0.
\end{split}
\]
This implies that $\mathcal{L}_{j}$ possesses at least one negative eigenvalue, denoted by $\lambda_{0}$.  Assume its associated eigenvector $\bm{\Gamma}_0=\left(\xi_{0}, \eta_{0}\right)$, that is,
$
\mathcal{L}_{j} \bm{\Gamma}_0=\lambda \bm{\Gamma}_0.
$
Using the expression of $\mathcal{L}_{j}$  again, the last equality yields
\[
\left.\begin{aligned}
-\partial_{x x} \xi_{0}+\xi_{0}-(2p+1) \Phi_{\omega_j}^{2p} \xi_{0}+\omega_j \eta_{0} & =\lambda_0 \xi_{0} \\
\eta_{0}+\omega_j \xi_{0} & =\lambda_0 \eta_{0}
\end{aligned}\right\}
\]
Hence, from the second equality, we obtain $\eta_{0}=-\frac{\omega_j}{1-\lambda_0} \xi_{0}$. Therefore, it follows from the first equality that 

\begin{equation*}
    -\partial_{x x} \xi_{0}+\left(1-\omega_j^{2}\right) \xi_{0}-(2p+1) \Phi_{\omega_j}^{2p} \xi_{0}=\lambda_0\left(\frac{\omega_j^{2}}{1-\lambda_0}+1\right) \xi_{0}.
\end{equation*}
Therefore,  comparing this last equation with\eqref{eigennegat} it follows that $\left(\lambda_0, \bm{\Gamma}_0\right)$ is exactly the pair satisfying
\begin{equation*}
\lambda_0=\frac{1}{2}\left(\lambda^{-}+\omega_j^{2}+1-\sqrt{(\lambda^{-})^{2}+2\left(\omega_j^{2}-1\right) \lambda^{-}+\left(\omega_j^{2}+1\right)^{2}}\right)
\end{equation*}
and $\bf{\Gamma}_0=\binom{\psi_0}{\frac{\omega_j \psi_0}{\lambda_0-1}} $.
\end{proof}
We denote $\bm{\Gamma}_0$ as the unique eigenfunction associated with the only negative eigenvalue $\lambda_0$ of $\mathcal{L}_j$. Therefore, we proceed with the following result.
			\begin{lemma}\label{precoer}
				There exits $C>0$ such that \[\begin{aligned}
					\langle \mathcal{L}_j (\bw),\bw\rangle &\geq C\left(\|\bw\|_{H}^2-\left(\bw,\bm{\Gamma_0}\right)_{L^2(\R)}^2-\left(\bw,\partial_x \Phi\right)_{L^2(\R)}^2\right),
				\end{aligned}\]
			 for all $\bm{w}\in H$.
			\end{lemma}
			\begin{proof}
				
				Initially, let  us assume that $\left(\bw,\bm{\Gamma_0}\right)_{L^2(\R)}=\left(\bw,\partial_x \Phi\right)_{L^2(\R)}=0.$
				
				Notice that 
\[
				S_j^{\prime \prime}\left(\bm{\Phi}_{\omega_j}\right)=\left(\begin{array}{cc}
					-\partial_{x x}+I -(2p+1)|\Phi_{\omega_j}|^{2p}& \omega_j \\
					\omega_j & I
				\end{array}\right),
\]
				where
\[H_j=\left(\begin{array}{cc}
					-\partial_{x x}+I & \omega_j \\
					\omega_j & I
				\end{array}\right) \text { and } L_j=\left(\begin{array}{cc}
					-(2p+1)|\Phi_{\omega_j}|^{2p} & 0 \\
					0 & 0
				\end{array}\right).
\]
								Hence, $L_j$ is a compact perturbation of the selfadjoint operator $H_j$.
								Now, we have that for any ${\bm f}=(f_1,f_2) \in H$,
\[
				\begin{aligned}
					\langle L {\bm f}, {\bm f}\rangle 
					& =\left\|\partial_x f_1\right\|_{L^2(\R)}^2+\|f_1\|_{L^2(\R)}^2+2 \omega_j\langle f_1, f_2\rangle+\|f_2\|_{L^2(\R)}^2   =\|{\bm f}\|_{ H }^2+2 \omega_j\langle f_1, f_2\rangle .
				\end{aligned}
\]
				For the term $2 \omega_j\langle f_1, f_2\rangle$, applying Hölder's and Young's inequalities, we have
\[
				|2 \omega_j\langle f_1, f_2\rangle| \leq|\omega|\|{\bm f}\|_{ H }^2 .
				\]
								Then, it follows that
\[
				\langle L {\bm f}, {\bm f}\rangle \geq(1-|\omega_j|)\|{\bm f}\|_{ H }^2 .
				\]
 Since $|\omega_j|<1$, we get
				that there exists $\delta>0$ such that the essential spectrum of $H_j$ is $[\delta,+\infty)$. By Weyl's theorem, $S_j^{\prime \prime}\left(\bm{\Phi}_{\omega_j}\right)$ and $H_j$ share the same essential spectrum. So we obtain the essential spectrum of $S_j^{\prime \prime}\left(\bm{\Phi}_{\omega_j} \right)$. Recall that we have obtained the only one negative eigenvalue $\lambda_0$ of $S_j^{\prime \prime}\left(\bm{\Phi}_{\omega_j}\right)$ in  Lemma 3.2 in \cite{bing}   and the kernel of $S_j^{\prime \prime}\left(\bm{\Phi}_{\omega_j}\right)$ in  Lemma 3.1 in \cite{bing}. So the discrete spectrum of $S_j^{\prime \prime}\left(\bm{\Phi}_{\omega_j}\right)$ is $\lambda_0, 0$, and the essential spectrum is $[\delta,+\infty)$.
				So, using the Spectral Theorem, and Lemmas \ref{kernel-lem} and \ref{eigenuniq}, we have that for all $\bw=(w_1,w_2) \in H$,
	\[\bw=a\bm{\Gamma_0}+b\partial_x \Phi+p.\]
 Hence, since $ \left(\bw,\bm{\Gamma_0}\right)_{L^2(\R)}=\left(\bw,\partial_x \Phi\right)_{L^2(\R)}=0,$
				it follows that $\bw = p$. Therefore, 
	\[\langle \mathcal{L}_j (\bw),\bw \rangle \geq C\|\bw\|_{H}^2.\]
				Then, to conclude the proof, it suffices to take $W=span\{\bm{\Gamma_0},\partial_x \bm{\Phi}\}$, Subsequently, considering $L^2(\mathbb{R}) \times L^2(\mathbb{R}) = W \oplus W^{\perp}$, and finally, by standard arguments, the proof follows.
			\end{proof}
			\begin{proposition}\label{coercivity1}
				There exists $C>0$ such that
\[\begin{aligned}
					\langle \mathcal{L}_j (\bw),\bw\rangle &\geq C\|\bw\|_{H}^2-\left(\bw,\bm{Z}_j^{+}\right)_{L^2(\R)}^2-\left(\bw,\bm{Z}_j^{-}\right)_{L^2(\R)}^2-\left(\bw,\partial_x \Phi\right)_{L^2(\R)}^2.
				\end{aligned}\]
			\end{proposition}
			\begin{proof}
		We first show that \begin{equation}\label{ortho1}
					\mathcal{L}_j (\bw),\bw\rangle \geq C\|\bw\|_{H}^2,
				\end{equation}
				whenever 
				\begin{equation}\label{ortho2}\left(\mathcal{L}_j(\bw),\bm{Y}_j^{+}\right)_{L^2(\R)}=\left(\mathcal{L}_j(\bw),\bm{Y}_j^{-}\right)_{L^2(\R)}=\left(\bw,\partial_x \Phi\right)_{L^2(\R)}=0.  
				\end{equation}
				
				In fact, we recall from Lemma \ref{precoer} that there exists $C>0$ such that, for any $\bw \in  H $
				\begin{equation}\label{lemma4.1}
					\begin{aligned}
						\langle \mathcal{L}_j (\bw),\bw\rangle &\geq C\left(\|\bw\|_{H}^2-\left(\bw,\bf{\Gamma_0}\right)_{L^2(\R)}^2-\left(\bw,\partial_x \Phi\right)_{L^2(\R)}^2\right).
					\end{aligned}   
				\end{equation}
				
				Now,  from the definition of $\mathcal{L}_j$, for any $\bw \in  H $
				\begin{equation}\label{defLj}
					\begin{split} 
						\langle\mathcal{L}_j \bw, \bw\rangle&=\left\|\partial_x  w_1\right\|_{L^2}^2+\left\|w_1\right\|_{L^2}^2+\left\|w_2\right\|_{L^2}^2+2 \omega_j \int_{\R} w_1 w_2\,dx-(2p+1)\int_{\R} |\bm{\Phi}_{\omega_j}|^{2p} w_1^2\,dx  \\&
						\geq(1-|\omega_j|)\|\bw\|_{H}^2-(2p+1)\int_{\R} |\bm{\Phi}_{\omega_j}|^{2p} w_1^2\,dx.
					\end{split}   
				\end{equation}
				Now,    assume by contradiction that \eqref{ortho1} is not valid. Then, there exists a sequence of functions $V_n=\left(v_{1, n}, v_{2, n}\right) \in  H $ satisfying the orthogonality conditions \eqref{ortho2} and the inequality
				\begin{equation}\label{contra}
					\left\langle\mathcal{L}_j V_n, V_n\right\rangle<\frac{1}{n}\left\|V_n\right\|^2.  
				\end{equation}
				From, \eqref{defLj} and \eqref{contra}, it follows that for $n$ large, \[(2p+1)\int_{\R} |\bm{\Phi}_{\omega_j}|^{2p} v_{1,n}^2\,dx>0\] and without loss of generality, we can assume that
				
		\[(2p+1)\int_{\R} |\bm{\Phi}_{\omega_j}|^{2p} v_{1,n}^2\,dx=1.
				\]
				Hence,  the sequence $\left(V_n\right)_n$ is bounded in $ H $. Up to extraction of a subsequence, it converges weakly to a function $V=(V_1,V_2) \in  H $ satisfying the orthogonality conditions \eqref{ortho2}. By the Rellich Theorem, 
				\[(2p+1)\int_{\R} |\bm{\Phi}_{\omega_j}|^{2p} V_{1}^2\,dx=1.
	\]
 So, we have  $V \not \equiv 0$. Moreover, by
 weak convergence property, it holds \[\langle\mathcal{L}_j V, V\rangle \leq \liminf _{n \rightarrow \infty}\left\langle\mathcal{L}_j V_n, V_n\right\rangle \leq 0.\] Now, let $\bm{W}=\bm{Y}^{+}+\bm{Y}^{-}$ and 
\[
				A=\operatorname{Span}\left\{V, \bm{\Gamma_0},\bm{W},\partial_x \bm{\Phi}\right\}.
\]
				Therefore, $\mathcal{L}_j|_{A}$ is nonpositive. Now,  from  \eqref{lemma4.1} which says that $\mathcal{L}_j$ is positive under $A^{\perp}$, which is contradiction.
								Now,  we have
\[L^2(\R)\times L^2(\R)=A\oplus A^{\perp}, \]
				where $A=\{\bm{ Y}_j^{+},\bm{Y}_j^{-},\partial_x \bm{\Phi}\}$.
				Then for all $\bw \in L^2(\R)\times L^2(\R)$, exist $\lambda_1\leq \lambda_2 \leq \lambda_3$  such that 
\[\bw=\lambda_1\mathbf{Y}_j^{+}+\lambda_2\mathbf{Y}_j^{-}+\lambda_3\partial_{x}\Phi+\mathbf{p}.\]
				Hence, \[\lambda_{1}=-\langle \bw,\mathbf{Z}_j^{+}\rangle, \, \, \lambda_2=- \langle \bw,\mathbf{Z}_j^{-}\rangle \, \, \text{and}\, \, \lambda_3=\langle \bw,\partial_x \Phi\rangle.\]
				Therefore, 
				\begin{equation*}
					\begin{split}
						\langle \mathcal{L}_j \bw,\bw\rangle &= \lambda_{1}\langle \mathcal{L}_j \bw,\mathbf{Y}_j^{+}\rangle+ \lambda_{2}\langle \mathcal{L}_j \bw,\mathbf{Y}_j^{-}\rangle+\langle \mathcal{L}_j \mathbf{p},\mathbf{p}\rangle\\
						& \geq \lambda_{1}\langle \bw,\mathbf{Z}_j^{+}\rangle+ \lambda_{2}\langle \bw,\mathbf{Z}_j^{-}\rangle+ C\|\mathbf{p}\|_{L^2(\R)\times L^2(\R)}^2\\
						&=-\lambda_{1}^2-\lambda_{2}^2+ C(\|\bw\|_{L^2(\R)\times L^2(\R)}^2-\lambda_{1}^2-\lambda_{2}^2 -\lambda_{3}^2)\\
						&= C\|\bw\|_{L^2(\R)\times L^2(\R)}^2 -(1+C)\lambda_{1}^2-(1+C)\lambda_{2}^2-C\lambda_{3}^2\\
						&\geq  C\|\bw\|_{L^2(\R)\times L^2(\R)}^2 -C(\lambda_{1}^2+ \lambda_{2}^2 + \lambda_{3}^2),
					\end{split}
				\end{equation*}
				with $\lambda_{k}=(\bw,\mathbf{Z}_j^{\pm})_{L^2(\R)}$ for all $k=1,2$ and $\lambda_{3}=(\bw,\partial_x \ru)_{L^2(\R)}$ . Then,
				\begin{equation}\label{cond1}
					\langle \mathcal{L}_j \bw,\bw\rangle + C(\lambda_{1}^2+ \lambda_{2}^2 + \lambda_{3}^2) \geq C\|\bw\|_{L^2(\R)\times L^2(\R)}^2,
				\end{equation}
				Moreover, we know that \[\begin{aligned}
					&\langle \mathcal{L}_j(\bw), \bw \rangle \\&=\int_{\R}|\partial_x w_1|^2\,dx+\int_{\R}| w_1|^2\,dx+2\omega_j\int_{\R} w_2 w_1\,dx+\int_{\R}| w_2|^2\,dx-(2p+1)\int_{\R}|\Phi_{\omega_j}|^{2p}w_1^2\,dx.
				\end{aligned}\]
				then,  \[\begin{aligned}
					 \langle \mathcal{L}_j(\bw), \bw \rangle +(2p+1)\int_{\R}|\Phi_{\omega_j}|^{2p}w_1^2\,dx  &=\int_{\R}|\partial_x w_1|^2\,dx+\int_{\R}| w_1|^2\,dx+2\omega_j\int_{\R} w_2 w_1\,dx+\int_{\R}| w_2|^2\,dx\\
					&\geq (1-\omega_j)\|\bw\|_{H}^2.
				\end{aligned}\]
				Hence, from \eqref{cond1}
\[\begin{aligned}
					\|\bw\|_{H}^2 &\leq C\left(\langle \mathcal{L}_j(\bw), \bw \rangle +(2p+1)\int_{\R}|\Phi_{\omega_j}|^{2p}w_1^2\,dx\right)\\
					& \leq C\langle\mathcal{L}_j(\bw), \bw \rangle + C\int_{\R}| w_1|^2\,dx\\
					& \leq C\langle\mathcal{L}_j(\bw), \bw \rangle + C \|\bw\|_{L^2(\R)\times L^2(\R)}^2\\
					&\leq  \langle \mathcal{L}_j \bw,\bw\rangle + C(\lambda_{1}^2+ \lambda_{2}^2 + \lambda_{3}^2),
				\end{aligned}\]
				which implies the desired result.
			\end{proof}
			Once the new coercivity condition is established, we proceed to state the main theorem under these conditions.
			\begin{theorem}\label{2maintheorem}
				For $j\in \{1,2,\ldots,N\}$ let us  $|\omega_{j}|\leq 1$, $x_{j} \in \R$ and let $\left(\Phi_{\omega_{j}}\right)$ be the associated ground state profiles. Denote the corresponding solitons by
	\[
				R_{j}^{(m)}(t, x):=\Phi_{\omega_{j}}^{(m)}\left(x-\omega_{j} t-x_{j}\right).
\]
				Define 
			\[
				\begin{aligned}
					\omega_{\star}:=\frac{1}{256}\min \left\{(\lambda_0^j)^{\frac32}\omega_j,1-\omega_{j}^2, \left|\omega_{j}-\omega_{k}\right|:  j, k=1, \ldots, N, j \neq k\right\}. 
				\end{aligned}
		\]
				If $\omega_{j} \neq \omega_{k}$ for any $j \neq k$, then there exist $T_{0} \in \R$ and a solution $\bu$ for \eqref{sistemaeq} defined in $\left[T_{0},+\infty\right)$ such that for all $t \in\left[T_{0},+\infty\right)$ the following estimate holds
				\begin{equation}\label{2desprin}
					\left\|\bu(t)- \ru(t)\right\|_{H^1 (\R)\times L^2(\R)}
					\leq e^{- \omega_{\star}^{\frac32} t}.
				\end{equation}
			\end{theorem}
			The demonstration of this result is identical to the proof of Theorem \ref{maintheorem} with the difference of some new parameters that do not affect the development of the proof. Additionally, it will be evident that the process to be performed to obtain the bases for its demonstration is similar to those carried out in Section \ref{subc-section}. Therefore, we proceed to obtain results similar to those previously established with some alterations to be observed below.
   
			We will now establish a result of parameter modulation that allows us to obtain some new estimates.
			
			\begin{lemma}\label{2lemamodulation}
				There exists $C > 0$ such that if $T_{0}$ is sufficiently large, then there exist $\mathcal{C}^{1}$-class functions
\[
				\tilde{x}_{j}:\left[t_{0}, T^{n}\right] \rightarrow \R, \quad j=1,2, \ldots, N
\]
				such that if $\tilde{R}_{j}^{(m)}$ it is the modulated wave
	\[\tilde{R}_{j}^{(m)}(x,t):= \Phi_{\omega_{j}}^{m}\left(x-\omega_{j
				}t-\tilde{x}_{j}(t)\right), \, \, \text{where} \, \, m=1,2
\]
				then for all $t \in [t_{0}, T^{n}]$, the function defined by
\[
				\var=\mathbf{u}(t)-\tilde{\mathbf{R}}(t)
\]
				where 
\[\tilde{\mathbf{R}}=\left(\sum_{j=1}^N \tilde{R}_{j}^{(1)},\sum_{j=1}^N \tilde{R}_{j}^{(2)}\right),\]
				satisfies for $j=1,2, \ldots, N$ and for all $t \in [t_{0}, T^{n}]$ the orthogonality conditions
				\begin{equation}\label{2orthogonality}
					\left(\varepsilon_{m}(t), \partial_{x} \tilde{R}_{j}^m(t)\right)_{L^2(\R)}=0 \qquad \text{for} \, \, m=1,2.
				\end{equation}
				Furthermore, for every $t \in \left[t_{0}, T^{n}\right]$,
				\begin{equation}\label{2segundoboostrap}
					\|\var\|_{H}+\sum_{j=1}^{N}|\tilde{x}_{j}(t)-x_{j}|\leq Ce^{-\omega_{\star}^{\frac32}t}.
				\end{equation}
				and
				\begin{equation}\label{2modulated}
					\begin{gathered}
						\sum_{j=1}^{N}\left|\partial_{t } \tilde{x}_{j}(t)\right|
						\leq C \left(\|\mathbf{\var}(t)\|_{H} +e^{-3\omega_{\star}^{\frac32}t}\right).
					\end{gathered}
				\end{equation}
			\end{lemma}
			\begin{proof}
				The proof is similar to   Lemma \ref{lemamodulation}.
			\end{proof}
   \begin{remark}
				Using the fact that $\mathbf{u}^{n}(T^{n})-\tilde{\mathbf{R}}(T^{n})\equiv 0$ and the uniqueness of the decomposition at time $t=T^{n}$, we necessarily have
				\[ \, \tilde{\mathbf{R}}(T^{n})=\mathbf{R}(T^{n})\, \, \text{ and}\, \, \, \tilde{x}_{j}(T^{n})=x_{j}  ,\,\, \text{with \, $j=1,2,\ldots,N$}.\]   
			\end{remark}
		 
			Following the strategy of modifying the coercivity property, we should choose to redefine the approximate sequences used as the main argument in the strategy in the subcritical case. Now, our strategy lies in applying a similar strategy to the one already used but making some specific changes. In fact, let us  consider $T_{n}$  be an increasing sequence converging to infinity. For each $n \in \N$, denote by  $\mathbf{u}^{n}$ the solution of  \eqref{sistema2} defined in the interval  $(T_{n}^{\star},T^{n}]$ with $T_{n}^{\star}$ being  the maximum time of existence for each $n$ and such that  \[\mathbf{u}^{n}(T^{n}) =\tilde{\mathbf{R}}(T^{n})+\sum_{j=1}^{N}\alpha_{j}^{+,n} \tilde{\bm{\Upsilon}}_{j}^{+},\]
			where \[\begin{aligned}
			 & \bm{\Upsilon}_{j}^{+}(t,x)=(1-\omega_{j}^2)^{1/2p}\bm{Y}_{j}^{+}(\sqrt{1-\omega_{j}^2}(x-\omega_j t-x_j) \,\, \text{and}\,\,\\&\tilde{\bm{\Upsilon}}_{j}^{+}(t,x)=(1-\omega_{j}^2)^{1/2p}\bm{Y}_{j}^{+}(\sqrt{1-\omega_{j}^2}(x-\omega_j t-\tilde{x}_j).  
			\end{aligned}\]
			
We will initially establish a fundamental result that is crucial in the proof of   Theorem \ref{maintheorem}.
			
			\begin{proposition}\label{UniformEstimates2}
				There exist $T_{0}\in \R$ and $n_{0} \in\N$ such that, for every $n \geq n_{0}$, exists $\bm{\alpha}_n =\left(\alpha_j^{+,n}\right)_{j\in \{1,\ldots,N\}} \in \R^{2N}$  with $\|\bm{\alpha}_{n}\|^2 \leq C  e^{-\omega_{\star}^{3}t} $ and each approximate solution $\mathbf{u}^{n}$ is defined in $[T_{0},T^{n}]$ and for all $t \in [T_{0},T^{n}]$
				\begin{equation*}
					\|\mathbf{u}^{n}(t)-\mathbf{R}(t)\|_{H}\leq e^{-\omega_{\star}^{\frac32}t}.
				\end{equation*}
			\end{proposition}
			The proof of this result is strongly linked to the following lemma. In fact, as in the subcritical case, we should prove that the approximate solutions satisfy a bootstrap argument.
			\begin{lemma}\label{2Bootstrap3}
				There exist $T_{0}\in \R$ depending only on $\omega_{j}$ and $n_{0} \in\mathbb{N}$ such that, if for every $n \geq n_{0}$, every approximate solution $\mathbf{u}^n$ is defined on $[T_{0},T^{n}]$, and for every $t \in[t_{0},T^{n}]$, with $t_{0}\in [T_{0},T^{n}]$,
				\begin{equation*}
					\|\bu^{n}(t)- \ru(t)\|_{H}\leq e^{-\omega_{\star}^{\frac32}t},
				\end{equation*}
				then for all $t\in[t_{0},T^{n}]$, 
				\begin{equation*}
					\|\bu^{n}(t)- \ru(t)\|_{H}\leq \frac{1}{2}e^{-\omega_{\star}^{\frac32}t}.
				\end{equation*}
			\end{lemma}
			To prove this result, we need to establish some preliminary results beforehand. Indeed,  for the sake of simplicity, we will drop the index $n$ for the rest of this section (except for $T_n$).

			Let us then consider the following notations:
\[\bm{\Psi}_{j}^{\pm}(t,x)=(1-\omega_{j}^2)^{1/2p}\bm{Z}_{j}^{+}(\sqrt{1-\omega_{j}^2}(x-\omega_j t-x_j),\]
		and	\[\tilde{\bm{\Psi}}_{j}^{\pm}(t,x)=(1-\omega_{j}^2)^{1/2p}\bm{Z}_{j}^{+}(\sqrt{1-\omega_{j}^2}(x-\omega_j t-\tilde{x}_j). \]

			Then, as in the case of soliton interactions, we have the following lemma,
			\begin{lemma}\label{solitons2}
				Let $m\in {1,2}$. There exists $C>0$ such that for all $t$ sufficiently large and for all $j\neq k\in {1, \ldots, N}$,
	\[
				\begin{aligned}
					&\int_{\R}\left(|\tilde{R}_{k}^{(m)}|+|\tilde{\Psi}_{k}^{\pm,m}|+|\partial_{x} \tilde{R}_{k}^{(m)}|+|\partial_{x} \tilde{\Psi}_{k}^{\pm,m}|\right)\left(|\tilde{R}_{j}^{(n)}|+|\tilde{\Psi}_{j}^{\pm,m}|+|\partial_{x} \tilde{R}_{J}^{(m)}|+|\partial_{x} \tilde{\Psi}_{j}^{\pm,m}|\right)\,dx\\&\leq Ce^{-4\omega_{\star}^{\frac32}  t}.
				\end{aligned}
\]
				and \[
				\begin{aligned}
					&\int_{\R}\left(|\tilde{\Upsilon}_{k}^{(m)}|+|\tilde{\Psi}_{k}^{\pm,m}|+|\partial_{x} \tilde{\Upsilon}_{k}^{(m)}|+|\partial_{x} \tilde{\Psi}_{k}^{\pm,m}|\right)\left(|\tilde{\Upsilon}_{j}^{(n)}|+|\tilde{\Psi}_{j}^{\pm,m}|+|\partial_{x} \tilde{\Upsilon}_{J}^{(m)}|+|\partial_{x} \tilde{\Psi}_{j}^{\pm,m}|\right)\,dx\\&\leq Ce^{-4\omega_{\star}^{\frac32}  t}.
				\end{aligned}
\]
			\end{lemma}
			\begin{proof}
				The proof of this result is similar to Lemma \ref{solitons}.
			\end{proof}
			Now, we consider $\var =\bu - \tilde{\ru}$, then we have from \eqref{sistema2} that
			\begin{equation}\label{equaepsi}
				\left\{\begin{array}{l}
					\partial_{t} \varepsilon_1=\partial_{x}\varepsilon_2 +\partial_x \tilde{R}_2 -\partial_t \tilde{R}_1\\
					\partial_{t} \varepsilon_2=\partial_{x}\left(\varepsilon_1+\tilde{R}_1-\partial_{xx}\tilde{R}_1 -\partial_{xx} \varepsilon_1 -f(\varepsilon_1+\tilde{R}_1)\right)-\partial_t \tilde{R}_2\\
					\var(T_n)=\sum_{j=1}^N \alpha_{j}^n \tilde{\bm{\Upsilon}}_j^{+}(T_n).
				\end{array}\right.  
			\end{equation}
			Moreover, for all $j \in \{1,\ldots,N\}$, we denote $\bm{\gamma}_{j}^{\pm}=(\gamma_{j}^{\pm,1},\gamma_{j}^{\pm,2})$ by
			
\[
			\gamma_{j}^{\pm,1}(t):=\int_{\R} \varepsilon_1(t) \cdot \tilde{\Psi}_{j}^{ \pm,2}(t)\,dx \quad \text{and} \, \, \gamma_{j}^{\pm,2}(t):=\int_{\R} \varepsilon_2(t) \cdot \tilde{\Psi}_{j}^{ \pm,1}(t)\,dx
\]
			and \[\bm{\gamma}^{\pm}(t)=\left(\bm{\gamma}_{j}^{\pm}(t)\right)_{j=1,2,\ldots,N}.\]

			\begin{lemma}\label{finaldata}
				For $n \geq n_{0}$ large enough, the following holds. For all $\bm{a}^{-} \in \mathbb{R}^{2N}$, there exists a unique $\bm{\alpha} \in \mathbb{R}^{2N}$ such that $\|\bm{\alpha}\| \leq 2\left\|{\bm a}^{-}\right\|$ and $\bm{\gamma}^{-}\left(T_{n}\right)={\bm a}^{-}$.  
			\end{lemma}
			\begin{proof}
				The following lemma establishes a one-to-one mapping between the choice of $b$ and the condition $\gamma^{-}(T_n)={\bm a}^{-}$, for some ${\bm a}^{-}$.
\[ 
				\begin{aligned}
					\Phi: \quad \mathbb{R}^{N} & \rightarrow \mathbb{R}^{2N} \\
					\bm{\alpha}=\left(\alpha_{l}^{+}\right)_{l \leq N} & \mapsto\left(\gamma_{k}^{-}\left(T_n\right)\right)_{k \leq N}
				\end{aligned}
				\]
				
				Its matrix in the canonical basis is
\[
				\operatorname{Mat} \Phi=\left(\begin{array}{cccc}
					1 & \int \tilde{\Upsilon}_{2}^{+} \tilde{\Psi}_{1}^{-}\left(T_n\right) & \cdots & \int \tilde{\Upsilon}_{N}^{+} \tilde{\Psi}_{1}^{-}\left(T_n\right) \\
					\int \tilde{\Upsilon}_{1}^{+} \tilde{\Psi}_{2}^{-}\left(T_n\right) & 1 & \cdots & \vdots \\
					\vdots & \vdots & \ddots & \vdots \\
					\int \tilde{\Upsilon}_{1}^{+} \tilde{\Psi}_{N}^{-}\left(T_n\right) & \cdots & \cdots & 1
				\end{array}\right)
\]
				
				But from Lemma \ref{solitons2}, we have, for $k \neq l$,
\[
				\left|\int \tilde{\Upsilon}_{l}^{ \pm} \tilde{\Psi}_{k}^{ \pm}\left(T_n\right)\right| \leq C_{0} e^{-3\omega_{\star}^{\frac32} T_n}
\]
 	with $C_{0}$ independent of $n$, and so by taking $n_{0}$ large enough, we have $\Phi=\operatorname{Id}+A_{n}$ where $\left\|A_{n}\right\| \leq \frac{1}{2}$. Thus $\Phi$ is invertible and $\left\|\Phi^{-1}\right\| \leq 2$. Finally, for a given ${\bm a}^{-} \in \mathbb{R}^{N}$, it is enough to define $\bm{\alpha}$ by ${\bm\alpha}=\Phi^{-1}\left({\bm a}^{-}\right)$ to conclude the proof
			\end{proof}
						\begin{remark}\label{remarksuper}
				The following estimates at $T_{n}$ hold:
				\begin{itemize}
					\item $\left|\gamma_{k}^{+}\left(T_n\right)\right| \leq C e^{-2 \omega_{\star}^{\frac32}T_n}\|\bm{\alpha}\|$ for all $k \in \{ 1, \ldots, N \}$,
					\item $\left\|\var\left(T_n\right)\right\|_{H} \leq C\|\bm{\alpha}\|$.
				\end{itemize}
				
			\end{remark}
			Let $T_{0}>0$ independent of $n$ to be determined later in the proof, $\|{\bm a}^{-}\|\leq  e^{-2 \omega_{\star} T_n}$ to be chosen, $\bm{\alpha}$ be given by Lemma \ref{finaldata} and $\bu$ be the corresponding solution of \eqref{equaepsi}. We now define the maximal time interval $\left[T\left({\bm a}^{-}\right), T_n\right]$ on which suitable exponential estimates hold.
			
			\begin{definition}\label{definia}
				Let $T\left({\bm a}^{-}\right)$ be the infimum of $T \geq T_{0}$ such that for all $t \in\left[T, T_n\right]$, both the $			\| \var(t)\|_{H} \leq e^{-\omega_{\star}^{\frac32} t} \quad \text { and } \quad \|{\bm\gamma}^{-}(t)\|_{H} \leq e^{-2\omega_{\star}^{\frac32} t}
	$ hold.   
			\end{definition}

			Observe that Proposition \ref{UniformEstimates2} is proved if for all $n$, we can find ${\bm a}^{-}$such that $T\left({\bm a}^{-}\right)=T_{0}$. The rest of the proof is devoted to prove the existence of such a value of ${\bm a}^{-}$.
			To continue, once the modulation result allowing us to estimate the direction $\langle \var,\partial_x \tilde{\ru}\rangle$ is established, we should estimate the other directions in the coercivity lemma.
			\begin{lemma}\label{direction}
				For all $t \in [T({\bm a}^{-}),T_n]$ and $j\in \{1,2,\ldots,N\}$. The following estimate holds.
	\[\left|\frac{d}{dt}\gamma_j^{\pm}(t)\pm \lambda_{0}^j \omega_j^{\frac32}\gamma_j^{\pm}(t)\right|\leq \|\var\|_{L^2(\R)\times L^2(\R)}^2+e^{-3\omega_{\star}^{\frac32}t}.\]
			\end{lemma}
			\begin{proof}
				Let $t \in [T({\bm a}^{-}),T_n]$ and fixed  $j\in \{1,2,\ldots,N\}$. Note that,
\[
 \begin{split}
 \frac{d}{dt}\gamma_{j}^{\pm,1}(t)+ \frac{d}{dt}\gamma_{j}^{\pm,2}(t)&=\frac{d}{dt}\int_{\R} \varepsilon_1(t)  \tilde{\Psi}_{j}^{ \pm,2}(t)\,dx+\frac{d}{dt}\int_{\R} \varepsilon_2(t)  \tilde{\Psi}_{j}^{ \pm,1}(t)\,dx\\
 &=\int_{\R} \partial_t \varepsilon_1(t) \cdot \tilde{\Psi}_{j}^{ \pm,2}(t)\,dx+\int_{\R} \varepsilon_1(t) \partial_t \tilde{\Psi}_{j}^{ \pm,2}(t)\,dx\\
 &\quad+\int_{\R} \partial_t \varepsilon_2(t) \cdot \tilde{\Psi}_{j}^{ \pm,1}(t)\,dx+\int_{\R} \varepsilon_2(t) \partial_t \tilde{\Psi}_{j}^{ \pm,1}(t)\,dx\\
 &=\int_{\R} \left(\partial_{x}\varepsilon_2 +\partial_x \tilde{R}_2 \right) \tilde{\Psi}_{j}^{ \pm,2}(t)\,dx-\int_{\R}\partial_t \tilde{R}_1 \tilde{\Psi}_{j}^{ \pm,2}(t)\,dx+\int_{\R} \varepsilon_1(t) \partial_t \tilde{\Psi}_{j}^{ \pm,2}(t)\,dx\\
 &\quad+\int_{\R} \partial_{x}\left(\varepsilon_1+\tilde{R}_1-\partial_{xx}\tilde{R}_1 -\partial_{xx} \varepsilon_1 -f(\varepsilon_1+\tilde{R}_1)\right)\tilde{\Psi}_{j}^{ \pm,1}(t)\,dx\\
 &\quad-\int_{\R}\partial_t \tilde{R}_2 \tilde{\Psi}_{j}^{ \pm,1}(t)\,dx+\int_{\R} \varepsilon_2(t) \partial_t \tilde{\Psi}_{j}^{ \pm,1}(t)\,dx.
 \end{split} 
\]
				Now, note that from Lemma \ref{solitons2}, 
    \[				f(\varepsilon+\tilde{R}_1) =f(\tilde{R}_1)+f^{\prime}(\tilde{R}_1)+O(\|\varepsilon_1\|_{L^2(\R)}^2)=|\tilde{R}_1|^{2p}\varepsilon_1+|\tilde{R}_1|^{2p}\tilde{R}_1+2p|\tilde{R}_1|^{2p}\varepsilon_1+O(\|\varepsilon\|_{L^2(\R)}^2)		.\]
				Then, from \eqref{pcrit} and Lemma \ref{solitons2}, we have
				\begin{equation*}
					\begin{aligned}
						&\frac{d}{dt}\gamma_{j}^{\pm,1}(t)+ \frac{d}{dt}\gamma_{j}^{\pm,2}(t)
						\\ &=\int_{\R}\varepsilon_2 \partial_{x} \tilde{\Psi}_{j}^{ \pm,2}(t)\,dx+\int_{\R} \tilde{R}_2 \partial_x\tilde{\Psi}_{j}^{ \pm,2}(t)\,dx\\&\quad-\sum_{k=1}^{N}(\omega_k+\partial_t \tilde{x}_k)\int_{\R} \tilde{R}_1 \partial_x\tilde{\Psi}_{j}^{ \pm,2}(t)\,dx-(\omega_j+\partial_t \tilde{x}_j)\int_{\R} \varepsilon_1(t) \partial_x \tilde{\Psi}_{j}^{ \pm,2}(t)\,dx\\
						&\quad+\int_{\R} \partial_{xx} \varepsilon_1\partial_{x}\tilde{\Psi}_{j}^{ \pm,1}(t)\,dx-\int_{\R}  \varepsilon_1\partial_{x}\tilde{\Psi}_{j}^{ \pm,1}(t)\,dx+(2p+1)\int_{\R}|\tilde{R}_{j}^{(1)}|^{2p}\varepsilon_1 \partial_{x}\tilde{\Psi}_{j}^{ \pm,1}(t)\,dx\\
						&\quad-\int_{\R} \left(\tilde{R}_1-\partial_{xx}\tilde{R}_1 -\sum_{k=1}^N|\tilde{R}_{j}^{(1)}|^{2p+1})\right)\,dx-\sum_{k=1}^{N}(\omega_k+\partial_t \tilde{x}_k)\int_{\R} \tilde{R}_2 \partial_x\tilde{\Psi}_{j}^{ \pm,1}(t)\,dx\\
						&\quad-\int_{\R} |\tilde{R}_{1}|^{2p+}\varepsilon_1(t) \partial_x \tilde{\Psi}_{j}^{ \pm,1}(t)\,dx-(\omega_j+\partial_{t}\tilde{x}_j)\int_{\R} \varepsilon_2(t) \partial_x \tilde{\Psi}_{j}^{ \pm,1}(t)\,dx+O(\|\varepsilon_1\|_{L^2(\R)}^2)\\
						&=-\langle \mathcal{L}_j \var , \partial_x \tilde{\Psi}_{j}^{\pm}\rangle -\sum_{k=1}^{N}\partial_t \tilde{x}_k\int_{\R} \tilde{R}_2 \partial_x\tilde{\Psi}_{j}^{ \pm,1}(t)\,dx-\sum_{k=1}^{N}\partial_t \tilde{x}_k\int_{\R} \tilde{R}_1 \partial_x\tilde{\Psi}_{j}^{ \pm,2}(t)\,dx\\
						&\quad -\partial_t \tilde{x}_j\int_{\R} \varepsilon_1(t) \partial_x \tilde{\Psi}_{j}^{ \pm,2}(t)\,dx-\partial_t \tilde{x}_j\int_{\R} \varepsilon_2(t) \partial_x \tilde{\Psi}_{j}^{ \pm,1}(t)\,dx\\
						&=-\langle  \var , \mathcal{L}_j\partial_x \tilde{\Psi}_{j}^{\pm}\rangle +\|\var\|_{L^2(\R)\times L^2(\R)}^2+e^{-3\omega_{\star}^{\frac32}t}\\
						&=-\lambda_{0}^j \omega_j^{\frac32}(\pm\gamma_{j}^{\pm,2}(t)\pm \gamma_{j}^{\pm,2}(t))+\|\var\|_{L^2(\R)\times L^2(\R)}^2+e^{-3\omega_{\star}^{\frac32}t}.
					\end{aligned} 
				\end{equation*}
				Hence, the result follows.
			\end{proof}
			With this lemma already established, we begin the control of the remaining directions. The following lemma was proved in \cite[Lemma 6]{CoteMartel}.
			\begin{lemma} \label{existencea}
				For $T_0$ large enough, there exists ${\bm a}^{-}\in \R^{2N}$ such that 
				\[\|{\bm a}^{-}\|^2\leq e^{-3\omega_{\star}^{\frac32}t}\]
				and $T({\bm a}^{-})=T_0.$
			\end{lemma}
			 
			\begin{lemma}\label{2direction}
				For all $t \in [T_0,T_n]$  the following estimate holds:
	\[\left\|\gamma^{\pm}(t)\right\|_{H}^2\leq e^{-3\omega_{\star}^{\frac32}t}.\]
			\end{lemma}
			\begin{proof}
				From Lemma \ref{direction}, we have 
				\[\begin{aligned}
					\left|\frac{d}{dt}\left(e^{+\lambda_0^j \omega_j^{\frac32}t}\gamma_j^{+}(t)\right)\right|\leq  e^{\lambda_0^j \omega_j^{\frac32}t}\left(e^{-2\omega_{\star}^{\frac32}t}+e^{-3\omega_{\star}^{\frac32}t}\right).
				\end{aligned}\]
				So, applying the fundamental Theorem of calculus from $t$ to $T_n$,
				\[\begin{aligned}
					\left|e^{\lambda_0^j \omega_j^{\frac32}T_n}\gamma_j^{+}(T_n)-e^{\lambda_0^j \omega_j^{\frac32}t}\gamma_j^{+}(t)\right|\leq  \int_{t}^{T_n}e^{-\lambda_0^j \omega_j^{\frac32}t}\left(e^{-2\omega_{\star}^{\frac32}s}+e^{-3\omega_{\star}^{\frac32}s}\right)\,ds.
				\end{aligned}\]
Then \[\begin{aligned}
					\left|e^{\lambda_0^j \omega_j^{\frac32}t}\gamma_j^{+}(t)\right|&\leq  \left|e^{\lambda_0^j \omega_j^{\frac32}T_n}\gamma_j^{+}(T_n)\right|+\int_{t}^{T_n}e^{\lambda_0^j \omega_j^{\frac32}t}\left(e^{-2\omega_{\star}^{\frac32}s}+e^{-3\omega_{\star}^{\frac32}s}\right)\,ds\\
					&\lesssim  e^{-2\omega_{\star}^{\frac32}T_n}\|\mathfrak{b}\|+ e^{-3\omega_{\star}^{\frac32}t} \lesssim  e^{-2\omega_{\star}^{\frac32}T_n}\|{\bm a}^{-}\|+ e^{-3\omega_{\star}^{\frac32}t} \lesssim e^{-3\omega_{\star}^{\frac32}t}.
				\end{aligned}\]
				Hence, $|\gamma_j^{+}(t)|\lesssim  e^{-3\omega_{\star}^{\frac32}t}.$
				On the other hand, from Definition \ref{definia} and Lemma \ref{existencea}, the result for $\gamma^{-}$ follows. 
			\end{proof}
			
			Now, recalling the localization applied in the subcritical case \eqref{localiz}, we know that from Lemma \ref{solitons} and Lemma \ref{conservationquase}, we have for the localized energy and momentum, defined in \eqref{energialoca} and \eqref{massaloca},
			\begin{equation}\label{quasemomen}
				\left|\mathcal{M}_j(\mathbf{u}(t))-\mathcal{M}_j(\mathbf{u}(T^{n}))\right| \leq \frac{C}{\sqrt{t}}e^{-2\omega_{\star}^{\frac32} t}, \,\, \textrm{j=1,2,\ldots,N}.  
			\end{equation}
			And from the conservation of energy 
			\begin{equation}\label{quaseener}
				\left|\mathcal{E}_j(\mathbf{u}(t))-\mathcal{E}_j(\mathbf{u}(T^{n}))\right| \leq \frac{C}{\sqrt{t}}e^{-2\omega_{\star}^{\frac32} t}, \,\, \textrm{j=1,2,\ldots,N}.    
			\end{equation}
			Moreover, we can obtain the following result, following the same process as in the supercritical case (see Lemma \ref{taylorS}). Taking into account that the only existing change now is the fact that the velocity has not been modulated, therefore it follows for $\mathcal{S}_{\text{loc}}^{j}$ the localized action defined, for all $\bw \in  H $, by
			\begin{equation}\label{2definicaoSj}
				\mathcal{S}_{\operatorname{loc}}^{j}(\bw,t):=\mathcal{E}_j(\bw,t)+\omega_j\mathcal{M}_j (\bw,t)  
			\end{equation}
			and we define an action-type functional for multi-solitons by
			\begin{equation}\label{2Sdefin}
				\mathcal{S} (\bw,t):=\sum_{j=1}^{N} \mathcal{S}_{\operatorname{loc}}^{j}(\bw,t).    
			\end{equation}
			\begin{lemma}\label{taylorS2}
				There exists $T_{0}$ such that if $t_{0} >T_{0}$, then for every $t \in [t_{0},T^{n}]$,
				\begin{equation*}
					\begin{aligned}
						\mathcal{S}( \bu(t))=\sum_{j=1}^N\left\{\mathcal{E}(\ru_j)+\frac{\omega_j}{2}\int_{\R} R_{j}^{(1)} R_{j}^{(2)} \,dx\right\}+\langle \mathcal{H}\var,\var\rangle +\mathcal{O}\left(e^{-3\omega_{\star}^{\frac32}t}\right),      
					\end{aligned}
				\end{equation*}
				where  \[\begin{aligned}
					\langle \mathcal{H}\var,\var\rangle&=\frac{1}{2}\int_{\R}|\partial_{x}\varepsilon_1|^2 \,dx+\frac{1}{2}\int_{\R}\left|\varepsilon_1\right|^{2} \,dx+\frac{1}{2}\int_{\R}\left|\varepsilon_2\right|^{2} \,dx\\
					&\quad-(2p+1)\sum_{j=1}^N\left(\frac{1}{2}\int_{\R}\left|\tilde{R}_{j}^{(1)}\right|^{2p}|\varepsilon_{1}|^2\,dx+\frac{\omega_j}{2}\int_{\R} \varepsilon_1 \varepsilon_2 \phi_{j}\,dx\right).   
				\end{aligned}
			\]
			\end{lemma} 
			
			Also, we can obtain the following classical result.
			\begin{proposition}\label{coercivity2}
				There exists $K>0$ such that
				\[\begin{aligned}
					\langle \mathcal{H} \var,\var\rangle &\geq K\|\var\|_{H}^2-K\sum_{j=1}^{N}\left(\left(\var,\tilde{\bm{\Psi}}_j^{+}\right)_{L^2(\R)}^2+\left(\var,\tilde{\bm{\Psi}}_j^{-}\right)_{L^2(\R)}^2\right).
				\end{aligned}\]
			\end{proposition}
			\begin{proof}
				First, we give a localized version of Proposition \ref{coercivity1}. Let $\Phi: \mathbb{R} \rightarrow \mathbb{R}$ be a $C^{2}$-function such that $\Phi(x)=\Phi(-x), \Phi^{\prime} \leqslant 0$ on $\mathbb{R}^{+}$with
				\[
				\begin{array}{ll}
					\Phi(x)=1 & \text { on }[0,1] ; \quad \Phi(x)=e^{-x} \quad \text { on }[2,+\infty) \\
					& e^{-x} \leqslant \Phi(x) \leqslant 3 e^{-x} \quad \text { on } \mathbb{R} .
				\end{array}
				\]
				
				Let $B>0$, and let $\Phi_{B}(x)=\Phi(x / B)$. Set
				\[
				\begin{aligned}
					\langle \mathcal{L}_{\Phi_{B}}\var, \var\rangle= & \int_{\R} \Phi_{B}\left(\cdot-x_{0}\right)\left\{\left|\partial_{x} \varepsilon_1\right|^{2}+|\varepsilon_1|^{2}+|\varepsilon_2|^{2}+\omega_0 \varepsilon_1 \varepsilon_2\right\}\,dx   -(2p+1)\int_{\R}|\Phi_{\omega_0}^{(1)}|^{2p}\varepsilon_1^2\,dx.
				\end{aligned}
		\]

				For the sake of simplicity, we assume that $x_{0}=0$. We set $z_1=\varepsilon_1\sqrt{\Phi_{B}}$ and $z_2=\varepsilon_2\sqrt{\Phi_{B}}$. Then, by simple calculations,
\[
 	\int_{\R}\left|\partial_{x} \varepsilon_1\right|^{2} \Phi_{B}\,dx=\int_{\R}\left|\partial_{x} z_1\right|^{2}\,dx+\frac{1}{4} \int|z_1|^{2}\left(\frac{\Phi_{B}^{\prime}}{\Phi_{B}}\right)^{2}\,dx-2 \int_{\R} \partial_{x} z_1 \bar{z}_1 \frac{\Phi_{B}^{\prime}}{\Phi_{B}}\,dx,\]
  \[
  \int_{\R}|\varepsilon_j|^{2} \Phi_{B}\,dx=\int_{\R}|z_j|^{2}\,dx,
  \quad j=1,2,\] 
				and \[\int_{\R} \varepsilon_1 \varepsilon_2 \Phi_{B}\,dx=\int_{\R}z_1 z_2\,dx.
\]
	 	Since, by definition of $\Phi_{B}$, we have $\left|\Phi_{B}^{\prime}\right| \leqslant(C / B) \Phi_{B}$, we obtain
 \[
 \begin{split}
 \int_{\R}\left|\partial_{x} z_1\right|^{2}\,dx-\frac{C}{B} \int_{\R}\left(\left|\partial_{x} z_1\right|^{2}+|z_1|^{2}\right) \,dx& \leqslant \int_{\R}\left|\partial_{x} \varepsilon_1\right|^{2} \Phi_{B}\,dx\\&
 \leqslant \int_{\R}\left|\partial_{x} z_1\right|^{2}\,dx+\frac{C}{B} \int_{\R}\left(\left|\partial_{x} z_1\right|^{2}+|z_1|^{2}\right)\,dx.
 \end{split}\]
 We also have
\[
 \begin{aligned} 	(2p+1)\int_{\R}|\Phi_{\omega_0}^{(1)}|^{2p}\varepsilon_1^2\,dx=(2p+1)\int_{\R}|\Phi_{\omega_0}^{(1)}|^{2p}z_1^2\frac{1}{\Phi_{B}}\,dx.
				\end{aligned}
\]
								Since $\Phi_{B} \equiv 1$ on $[-B, B]$ and $\Phi_{\omega_{0}}^{(1)}(x) \leqslant C e^{-\left(\sqrt{\omega_{0}} / 2\right)|x|}$, we have, for all $x \in \mathbb{R}$,
				
\[
				\left|\frac{1}{\Phi_{B}}-1\right| \Phi_{\omega_{0}}^{(1)}(x) \leqslant e^{-\left(\sqrt{\omega_{0}}-2 / B\right)|x| / 2} \leqslant C e^{-\sqrt{\omega_{0}} B / 4} \leqslant \frac{1}{B}
\]
								for $B$ large enough. Thus,
\[
				\begin{aligned}
					(2p+1)\int_{\R}|\Phi_{\omega_0}^{(1)}|^{2p}\varepsilon_1^2\,dx\leqslant (2p+1)\int_{\R}|\Phi_{\omega_0}^{(1)}|^{2p}z_1^2\,dx+\frac{C}{B} \int_{\R}|z_1|^{2}\,dx.
				\end{aligned}
\]
Collecting these calculations, we obtain
\[
				\langle \mathcal{L}_{\Phi_{B}}\var, \var\rangle  \geqslant \langle \mathcal{L}_{0}{\bm z}, {\bm z}\rangle-\frac{C}{B} \int\left(\left|\partial_{x} z_1\right|^{2}+|z_1|^{2}+ |z_2|^{2}\right).
\]
								Thanks to the orthogonality conditions on ${\bm z}=(z_1,z_2)$, we verify easily using the property of $\Phi_{B}$ that
$ 
({\bm z},\partial_{x}\Phi_{\omega_0})_{L^2(\R)}=0,
$ 
				for $B$ large enough. By Proposition \ref{coercivity1}, we obtain  for $B$ large enough that
\[
				\begin{aligned}
   \langle \mathcal{L}_{\Phi_{B}}\var, \var\rangle +C\left(\left({\bm z},{\bm Z}_0^{+}\right)_{L^2(\R)}^2+\left({\bm z},\bm{Z}_0^{-}\right)_{L^2(\R)}^2\right)  & \geqslant\left(C-\frac{C}{B}\right)\|{\bm z}\|_{H}^{2}\\& \geqslant \frac{C}{2}\|{\bm z}\|_{H}^{2} \\&\geqslant \frac{C}{2}\left(1-\frac{C}{B}\right) \int_{\R}\left(|\varepsilon_1|^{2}+\left|\partial_{x} \varepsilon_1\right|^{2}+|\varepsilon_1|^{2}\right) \Phi_{B}\,dx \\
					& \geqslant \frac{C}{4} \int_{\R}\left(|\varepsilon_1|^{2}+\left|\partial_{x} \varepsilon_1\right|^{2}+|\varepsilon_1|^{2}\right) \Phi_{B}\,dx,
				\end{aligned}
\]
				implying \begin{equation}\label{claim1}
					\langle \mathcal{L}_{\Phi_{B}}\var, \var\rangle +K\left(\left(\var,\bm{Z}_0^{+}\right)_{L^2(\R)}^2+\left(\var,\bm{Z}_0^{-}\right)_{L^2(\R)}^2\right)\geqslant K\|\var\|_{H}^{2}.   
				\end{equation}
 Now, let $B>B_{0}$, and   $L>0$. Since $\sum_{k=1}^{N} \phi_{k}(t) \equiv 1$, we decompose $\langle \mathcal{H}\var, \var\rangle$ as follows:
\[
				\begin{aligned}
					\langle \mathcal{H}\var, \var\rangle= & \sum_{k=1}^{N} \int_{\R} \Phi_{B}\left(\cdot-\omega_kt-\tilde{x}_{k}\right)\left\{\left|\partial_{x} \varepsilon_1\right|^{2}+|\varepsilon_1|^{2}+|\varepsilon_2|^{2}+\omega_k \varepsilon_1 \varepsilon_2\right\}\,dx \\
					& -(2p+1)\sum_{k=1}^{N} \int_{\R}|R_k^{(1)}|^{2p}\varepsilon_1^2\,dx \\
					& +\sum_{k=1}^{N}\int_{\R}\left(\phi_{k}-\Phi_{B}\left(\cdot-\omega_kt-\tilde{x}_{k}\right)\right)\left\{\left|\partial_{x} \varepsilon_1\right|^{2}+|\varepsilon_1|^{2}+|\varepsilon_2|^{2}+\omega_k \varepsilon_1 \varepsilon_2 \right\}\,dx.
				\end{aligned}
	\]
				By \eqref{claim1}, for any $k=1, \ldots, N$, we have  for $B$ large enough that
\[
				\begin{aligned}
			 \int_{\R} \Phi_{B}\left(\cdot-x_k(t)\right)		&\left\{\left|\partial_{x} \varepsilon_1\right|^{2}+|\varepsilon_1|^{2}+|\varepsilon_2|^{2}+\omega_k \varepsilon_1 \varepsilon_2\right\}\,dx  -(2p+1)\int_{\R}|R_k^{(1)}|^{2p}\varepsilon_1^2\,dx\\&
					\quad+K\left(\left(\var,\tilde{\bm{\Psi}}_k^{+}\right)_{L^2(\R)}^2+\left(\var,\tilde{\bm{\Psi}}_k^{-}\right)_{L^2(\R)}^2\right)\\&\geqslant \lambda_{k} \int_{\R}\Phi_{B}\left(\cdot-x_{k}(t)\right)\left(\left|\partial_{x} \varepsilon_1\right|^{2}+|\varepsilon_2|^{2}+|\varepsilon_1|^{2}\right)\,dx,
				\end{aligned}
\]
				where $x_k(t)=\omega_k t +\tilde{x}_k$.
								Moreover, by the properties of $\Phi_{B}$ and $\varphi_{k}(t)$, for $L$ large enough, we have
				
\[
				\varphi_{k}(t)-\Phi_{B}\left(\cdot-x_{k}(t)\right) \geqslant-e^{-L /(4 B)},
\]
	and  for
	 $\delta_{k}=\delta_{k}\left(\omega_{k}\right)>0$,
\[
				\left|\partial_{x} \varepsilon_1\right|^{2}+|\varepsilon_1|^{2}+|\varepsilon_2|^{2}+\omega_k\varepsilon_1 \varepsilon_2 \geqslant \delta_{k}\left(\left|\partial_{x} \varepsilon_1\right|^{2}+|\varepsilon_1|^{2}+|\varepsilon_2|^{2}\right) \geqslant 0.
\]
				So,
\[
				\begin{aligned}
					& \int_{\R}\left(\phi_{k}-\Phi_{B}\left(\cdot-x_k(t)\right)\right)\left\{\left|\partial_{x} \varepsilon_1\right|^{2}+|\varepsilon_1|^{2}+|\varepsilon_2|^{2}+\omega_k \varepsilon_1 \varepsilon_2 \right\}\,dx \\
					&  \geqslant \delta_{k} \int_{\R}\left(\varphi_{k}(t)-\Phi_{B}\left(\cdot-x_{k}(t)\right)\right)\left(\left|\partial_{x} \varepsilon_1\right|^{2}+|\varepsilon_1|^{2}+|\varepsilon_2|^{2}\right)\,dx\\
					&\geqslant-C e^{-L /(4 B)} \int_{\R}\left(\left|\partial_{x} \varepsilon_1\right|^{2}+|\varepsilon_1|^{2}+|\varepsilon_2|^{2}\right)\,dx.
				\end{aligned}
\]
    Thus, our above considerations reveal that 
\[
				\begin{aligned}
					&\langle \mathcal{H}\var, \var\rangle+ K\left(\left(\var,\tilde{\bm{\Psi}}_k^{+}\right)_{L^2(\R)}^2+\left(\var,\tilde{\bm{\Psi}}_k^{-}\right)_{L^2(\R)}^2\right) \\&\geqslant K \int_{\R}\left(\sum_{k=1}^{N} \phi_{k}\right)\left(\left|\partial_{x} \varepsilon_1\right|^{2}+|\varepsilon_1|^{2}+|\varepsilon_2|^{2}\right)\,dx-C e^{-L /(4 B)} \int_{\R}\left(\left|\partial_{x} \varepsilon_1\right|^{2}+|\varepsilon_1|^{2}+|\varepsilon_2|^{2}\right)\,dx  
				\end{aligned}
\]
				
				and since $\sum_{k=1}^{N} \phi_{k}(t) \equiv 1$, we obtain the result by taking $L$ large enough.
			\end{proof}
			Finally, we proceed to prove   Lemma \ref{2Bootstrap3}.
			\begin{proof}[Proof of Lemma \ref{2Bootstrap3}]
				Let us fix $t \in \left[t_{0}, T^{n}\right]$. Note that we have
				\begin{equation}\label{2estimateuniforme1}
					\begin{aligned}
						\left\|\bu^{n}(t)-\ru(t)\right\|_{H} &\leq\left\|\bu^{n}(t)-\tilde{\ru}(t)\right\|_{H}+\left\|\tilde{\ru}(t)-\ru(t)\right\|_{H} \\
						&\leq \left\|\var(t)\right\|_{H}+C\sum_{j=1}^N|\tilde{x}_{j}(t)-x_{j}|.
					\end{aligned}
				\end{equation}
				Next, we observe that from the definition of $\mathcal{E}_b$ (see \eqref{energia}) and of $\mathcal{S}$, we deduce,
\[
				\begin{aligned}
					\left|\mathcal{S}(\mathbf{u}(t))-\mathcal{S}(\mathbf{u}(T^{n}))\right|\leq\left| \mathcal{E}_b(\mathbf{u}(t))-\mathcal{E}_b(\mathbf{u}(T^{n}))\right|+C\left|  \sum_{j=1}^{N}\left(\mathcal{M}_j(\mathbf{u}(t))-\mathcal{M}_j(\mathbf{u}(T^{n}))\right)\right|.
				\end{aligned}
\]
			Inequalities 	\eqref{quaseener} and \eqref{quasemomen} imply that
				\begin{equation}\label{2SuTn}
					\begin{aligned}
						\left|\mathcal{S}(\mathbf{u}(t))-\mathcal{S}(\mathbf{u}(T^{n}))\right|\leq \frac{C}{\sqrt{t}} e^{-2 \omega_{\star}^{\frac32}t}.
					\end{aligned}   
				\end{equation}
				Therefore, from \eqref{2SuTn}, Lemma \ref{2direction} and Lemma \ref{taylorS2}, we get
\[
				\begin{aligned}
					\left\|\mathbf{\var}(t)\right\|_{H}^{2}&\leq C\langle \mathcal{H}\var, \var\rangle +K\sum_{j=1}^{N}\left(\left(\var,\tilde{\bm{\Psi}}_j^{+}\right)_{L^2(\R)}^2+\left(\var,\tilde{\bm{\Psi}}_j^{-}\right)_{L^2(\R)}^2\right) \\
					&=\mathcal{S}(\mathbf{u}(t))-\mathcal{S}(\mathbf{u}(T^{n}))+\mathcal{O}\left(\frac{1}{\sqrt{t}}e^{-2\omega_{\star}^{\frac32}t}\right) \leq  \frac{C}{\sqrt{t}}e^{-2\omega_{\star}^{\frac32}t}.
				\end{aligned}
\]
				This, combined with the last inequality and \eqref{omega},  shows that
				\begin{equation}\label{2epsilonestimative}
					\begin{aligned}
						\left\|\mathbf{\var}(t)\right\|_{H}^{2} \leq  \frac{C}{\sqrt{t}}e^{-2\omega_{\star}^{\frac32}t},
					\end{aligned}
				\end{equation}
				for $T_0$ sufficiently large. This last estimate plays a crucial role in proving  the required result.
      In fact,  we have from \eqref{modulated}, \eqref{omega}, and \eqref{2epsilonestimative} for $T_{0}$ sufficiently large that
\[
				\begin{split}
					\left|\partial_{t } \tilde{x}_{j}(t)\right|&\lesssim  \left\|\var(t)\right\|_{H}+ e^{-3 \omega_{\star}^{\frac32}t} 
				\lesssim\frac{1}{\sqrt{t}} e^{-2 \omega_{\star}^{\frac32}t}+ \left\|\var(t)\right\|_{H} 
					\lesssim \frac{1}{\sqrt{t}} e^{-2 \omega_{\star}^{\frac32}t}+\frac{1}{\sqrt{t}} e^{- \omega_{\star}^{\frac32}t} 
					\lesssim \frac{1}{\sqrt{t}} e^{- \omega_{\star}^{\frac32}t}.
				\end{split}
	\]
				Since $\tilde{x}_j(T^n) = x_j$, according to the fundamental theorem of calculus and the above estimate, 
 \begin{equation}\label{2estimativeCal}
					\begin{split}
						\left|\tilde{x}_{j}(t)-x_{j}\right|&=\left|\int_{t}^{T^n}\partial_{s}\tilde{x}_{j}(s)ds   \right| \leq  \int_{t}^{T^n}\left|\partial_{s}\tilde{x}_{1}(s)   \right|ds  \leq \frac{C}{\sqrt{t}} \int_{t}^{T^n}e^{- \omega_{\star}^{\frac32}s}\,ds\\
						&  \leq \frac{C}{\sqrt{t}}\left(e^{-\omega_{\star}^{\frac32}t}-e^{-\omega_{\star}^{\frac32}T_{n}}\right) \leq \frac{C}{\sqrt{t}} e^{- \omega_{\star}^{\frac32}t}.
					\end{split}  
 \end{equation}
				Therefore, using  \eqref{2epsilonestimative} and \eqref{2estimativeCal} in   estimate \eqref{2estimateuniforme1}, it follows that
\[
				\begin{aligned}
					\left\|\bu^{n}(t)-\ru(t)\right\|_{H} &\leq \frac{C}{\sqrt{t}}e^{-\omega_{\star}^{\frac32}t} 
     , \quad \, \text{for all $t \in\left[t_{0}, T^{n}\right]$.}
				\end{aligned}
\] 
				Therefore, for $T_{0}$ sufficiently large,
				\begin{equation*}
					\begin{aligned}
						\left\|\mathbf{u}^{n}(t) -\mathbf{R}(t) \right \|_{H}\leq  \frac{C}{\sqrt{t}}e^{- \sqrt{\omega_{\star} }\omega_{\star} t} , \quad \, \text{for all $t \in\left[t_{0}, T^{n}\right]$.}
					\end{aligned}
				\end{equation*}
				Taking $t \geq (2C)^2$, it follows that for all $t \in [t_0, T^n]$
\[
				\begin{aligned}
					\left\|\mathbf{u}^{n}(t)-\mathbf{R}(t)\right \|_{H} \leq   \frac{1}{2} e^{- \omega_{\star}^{\frac32}t}, 
				\end{aligned}
\]
				which shows the desired result.
				
			\end{proof}

			\section*{Acknowledgment}
			
			The authors are supported by Nazarbayev University under the Faculty Development Competitive Research Grants Program for 2023-2025 (grant number 20122022FD4121).
			

			\section*{Conflict of interest} The authors declare that they have no conflict of interest. 
			
			\section*{Data Availability}
			There is no data in this paper.

		\end{document}